\documentclass[nonblindrev]{informsprepr}

\makeatletter
\long\def\theARTICLETOPLEFT{}%
\def\theARTICLETOPRIGHT{}
\makeatother

\usepackage{amsmath,amssymb,amsfonts,mathtools}
\usepackage{mathrsfs}
\usepackage[bookmarks=true,linkcolor=red,citecolor=blue,urlcolor=blue,colorlinks=true,breaklinks,pdfencoding=unicode,psdextra]{hyperref}
\usepackage{enumerate}
\usepackage[textsize=footnotesize,textwidth=1.5cm]{todonotes}
\usepackage[ruled,vlined,procnumbered]{algorithm2e}
\usepackage{bm}
\usepackage{dsfont}
\usepackage{nicematrix}
\usepackage{booktabs}
\usepackage{threeparttable}
\usepackage{tikz}
\usepackage{pgfplots}
\usepackage{bookmark}
\usepackage{multirow}
\usepackage[normalem]{ulem}
\usepackage{longtable}
\usepackage[size=small]{subcaption}
\usepackage{lscape}
\usepackage{makecell}
\usepackage{rotating}
\usepackage{fancyvrb}
\usepackage{anyfontsize}
\usetikzlibrary{positioning}

\usepackage{comment}
\usepackage{tabularx}
\usepackage[margin=1in]{geometry}
\usepackage{graphicx}
\usepackage{rotating}
\usepackage{array}
\usepackage{makecell}
\usepackage{siunitx}  

\usepackage{pgfplots}
\pgfplotsset{compat=newest,compat/show suggested version=false}

\usepackage{tcolorbox}

\usepackage{dcolumn}
\newcolumntype{d}[1]{D{/}{\,/\,}{#1}}

\usetikzlibrary{shapes, calc, arrows}

\usepackage{apacite}
\usepackage{natbib}
\bibpunct[, ]{(}{)}{,}{a}{}{,}%
\def\bibfont{\small}%
\def\newblock{\ }%
\newcommand{\R}{\mathds{R}}
\newcommand{\Z}{\mathds{Z}}
\newcommand{\B}[1]{\{0,1\}^{#1}}

\newcommand{\define}{\coloneqq}
\newcommand{\card}[1]{|#1|}
\newcommand{\T}{^{\top}}
\newcommand{\sprod}[2]{{#1}\T {#2}}

\OneAndAHalfSpacedXI
\def\bibfont{\small}%
\def\newblock{\ }%
\pgfplotsset{compat=1.16}
\bibpunct[, ]{(}{)}{,}{a}{}{,}%
\graphicspath{{Images/}}

\usepackage{etoolbox}
\newcommand{\zerodisplayskips}{%
  \setlength{\abovedisplayskip}{3pt}%
  \setlength{\belowdisplayskip}{3pt}%
  \setlength{\abovedisplayshortskip}{0pt}%
  \setlength{\belowdisplayshortskip}{0pt}}
\appto{\normalsize}{\zerodisplayskips}
\appto{\small}{\zerodisplayskips}
\appto{\footnotesize}{\zerodisplayskips}

\newcommand{\potvar}{\lambda}
\newcommand{\potvarvec}{\boldsymbol \lambda}
\newcommand{\potatoes}{\mathscr{S}}
\newcommand{\potatoessub}{\potatoes^{\subseteq}}
\newcommand{\cliques}{\mathscr{C}}
\newcommand{\edges}{\mathcal E}
\newcommand{\vertices}{\mathcal V}
\newcommand{\graph}{\mathcal G}
\newcommand{\bx}{\boldsymbol{x}}
\newcommand{\by}{\boldsymbol{y}}
\newcommand{\bc}{\boldsymbol{c}}

\newcommand{\bb}{\boldsymbol{b}}
\newcommand{\bA}{\boldsymbol{A}}

\newcommand{\subgraph}{\mathcal F}
\newcommand{\SUBedges}{\edges(\subgraph)}
\newcommand{\SUBvertices}{\vertices(\subgraph)}
\newcommand{\SUBgraphs}{\mathscr F_{\graph}}

\newcommand{\oldBP}{$\text{ILP}_\text{E1}$\xspace}
\newcommand{\zigbase}{$\text{ILP}_\text{E2}$\xspace}
\newcommand{\zigbases}{$\text{ILP}_\text{E2}^\star$\xspace}
\newcommand{\zigbasesa}{$\text{ILP}_\text{E2, a}^\star$\xspace}
\newcommand{\zigbasesb}{$\text{ILP}_\text{E2, b}^\star$\xspace}

\newcommand{\ILPnat}{$\text{ILP}_\text{N}$\xspace}
\newcommand{\ILPhyb}{$\text{ILP}_\text{H}$\xspace}
\newcommand{\ILPcomp}{$\text{ILP}_\text{C}$\xspace}
\newcommand{\ILPcompstar}{$\text{ILP}_\text{C}^\star$\xspace}

\newcommand{\lpval}{\zeta}
\newcommand{\lpvalnat}{\lpval_{\mathrm{N}}}
\newcommand{\lpvalbp}{\lpval_{\mathrm{E1}}}
\newcommand{\lpvalzig}{\lpval_{\mathrm{E2}}}

\newcommand{\lexgeq}[1]{\geq_{\mathrm{lex}{#1}}}

\newcommand{\BPs}{BP$^\star$}
\newcommand{\BPCs}{BP$^\star$-C}

\newcommand{\oldBPLP}{$\text{LP}_\text{E1}$\xspace}
\newcommand{\zigbaseLP}{$\text{LP}_\text{E2}$\xspace}

\newcommand{\ILPnatLP}{$\text{LP}_\text{N}$\xspace}

\usetikzlibrary{calc}

\TheoremsNumberedThrough     
\ECRepeatTheorems

\EquationsNumberedThrough    

\MANUSCRIPTNO{}

\begin{document}

\RUNAUTHOR{}

\RUNTITLE{Branch-and-price strikes back for the $k$-vertex cut problem}

\TITLE{Branch-and-price strikes back for the $k$-vertex cut problem}

\ARTICLEAUTHORS{

\AUTHOR{Fabio Ciccarelli}
\AFF{Department of Computer, Control, and Management Engineering “Antonio Ruberti”, Sapienza University of Rome, Rome, Italy, \EMAIL{f.ciccarelli@uniroma1.it}}

\AUTHOR{Fabio Furini}
\AFF{Department of Computer, Control, and Management Engineering “Antonio Ruberti”, Sapienza University of Rome, Rome, Italy, \EMAIL{fabio.furini@uniroma1.it}}

\AUTHOR{Christopher Hojny}
\AFF{Department of Mathematics and Computer Science, Eindhoven University of Technology, Eindhoven, The Netherlands \\ \EMAIL{c.hojny@tue.nl}}

\AUTHOR{Marco Lübbecke}
\AFF{Chair of Operations Research, RWTH Aachen University, Aachen, Germany \\ \EMAIL{luebbecke@or.rwth-aachen.de}}

} 

\ABSTRACT{
Given an undirected graph, the \textit{$k$-vertex cut problem} ($k$-VCP) asks for a minimum-cost set of vertices whose removal yields at least $k$ connected components in the resulting graph. The $k$-VCP is an important problem in network optimization, with applications in infrastructure protection and epidemic containment.
We present a new extended integer linear programming (ILP) formulation that unifies and strengthens existing models and serves as the foundation for a new branch-and-price algorithm for the $k$-VCP. An in-depth theoretical study enables us to devise algorithmic components such as tailored branching rules that preserve the structure of the pricing problems, as well as valid inequalities and symmetry-handling techniques.
We also show that our new model dominates all previous ILP formulations of the $k$-VCP in terms of their linear relaxations, which theoretically justifies the computational effectiveness of our approach.
Extensive computational experiments against state-of-the-art methods demonstrate substantially improved performance, both in terms of instances solved to proven optimality and running times. 

}%

\KEYWORDS{graph partitioning; critical vertex detection problems; integer linear programming; branch-and-price}
\HISTORY{\today}

\maketitle

\section{Introduction}
\label{sec:intro}

Graph disconnection problems via critical vertex deletion are fundamental in network optimization, due to their theoretical significance as well as their broad applications in telecommunication and transportation systems, computational biology, network vulnerability assessment, and immunization against epidemics~\citep{Lalou2018}.
As connectivity underlies the behavior, function, and resilience of complex networks, understanding the mechanisms through which it can be altered, influenced, or reorganized is important for both graph theory and real-world applications. 
Within this broad family of \textit{critical vertex detection problems} (CVDPs), a prominent class is concerned with fragmenting the graph through the deletion of vertices.

In this paper, we focus on the \emph{$k$-vertex cut problem} ($k$-VCP), an important CVDP that has received considerable attention in recent years; see, e.g.,~\cite{Marx,BGZ14,CFLMMM18,FurLMP2020}. 

Formally, let $\graph = (\vertices, \edges)$ be a simple undirected graph, where $\vertices$ denotes the set of vertices and $\edges$ the set of edges.  
For any subset of vertices $S \subseteq \vertices$, the induced subgraph $\graph[S] = (S, \edges[S])$ consists of the vertices in $S$ together with the edges
\[
\edges[S] = \{\, \{u,v\} \in \edges : u, v \in S \,\},
\]
i.e., all edges of $\graph$ whose endpoints lie in $S$.  
A \emph{connected component} is an inclusion-wise maximal connected induced subgraph of~$\graph$.
A \emph{vertex cut} is a subset of vertices $\vertices_0 \subseteq \vertices$ whose removal disconnects the graph, i.e., the induced subgraph $\graph[\vertices \setminus \vertices_0]$ contains at least two connected components.  
If the induced subgraph contains at least $k$ connected components, then $\vertices_0$ is called a \emph{$k$-vertex cut}.
The (weighted) $k$-VCP can therefore be formally defined as follows:

\medskip
\begin{definition}[$k$-vertex cut problem]
Given a graph $\graph = (\vertices, \edges)$, a non-negative cost $c_v$ for each vertex $v \in \vertices$, and an integer $k \ge 2$, the \emph{$k$-vertex cut problem} ($k$-VCP) consists in finding a minimum-cost set $\vertices_0 \subseteq \vertices$ such that the induced subgraph $\graph[\vertices \setminus \vertices_0]$ contains at least $k$ connected components.
\end{definition}
\medskip

An equivalent structural characterization 
of $k$-vertex cuts that we use throughout the paper is as follows.
A subset of vertices $\vertices_0 \subseteq \vertices$ is a $k$-vertex cut if and only if the 
induced subgraph $\graph[\vertices \setminus \vertices_0]$ admits a partition of its vertex 
set into $k$ non-empty subsets 
$\vertices_1, \vertices_2, \ldots, \vertices_k$
that are pairwise disconnected, in the sense that no edge of $\edges$ has endpoints 
in two different subsets.
Each subset may contain one or more connected components; what matters is that the 
$k$ subsets are mutually disconnected in the induced subgraph $\graph[\vertices \setminus \vertices_0]$.
For ease of reference, in the remainder of the paper, we will refer to each subset 
$\vertices_1, \vertices_2, \ldots, \vertices_k$ as a  \emph{cluster} of the graph.
From a complementary perspective, the $k$-VCP can also be viewed as finding a 
maximum-cost set of vertices $\bigcup_{i=1}^k \vertices_i$ such that 
the induced subgraph $\graph\!\left[\bigcup_{i=1}^k \vertices_i\right]$ contains at least 
$k$ connected components.

Figure~\ref{fig:karate_example} illustrates optimal solutions of the $k$-VCP on the classical Zachary’s karate club network \citep{zachary1977information} for $k \in \{3,5,10\}$ with unit vertex costs. The vertices of the graph are labeled with the numbers displayed inside each node. This instance has also been used to illustrate optimal solutions of other related optimization problems on graphs; see, e.g., \citet{Furini2019, furini2021branch, FuriniLMP22}.
For each value of $k$, the corresponding subfigure shows the vertices of an optimal $k$-vertex cut in gray, together with their incident edges drawn as dashed lines to indicate the portion of the graph removed by the vertex cut. The remaining vertices are partitioned into exactly $k$ connected components, each highlighted by a different color and shape. For $k=3$, it is sufficient to remove a single vertex; for $k=5$, two vertices must be removed; and for $k=10$, four vertices are required. The figure illustrates both the increasing fragmentation of the graph and the change in the structure of optimal $k$-vertex cuts as $k$ grows. Further details regarding these optimal solutions are provided in Appendix~\ref{sec:karate}.

\begin{figure}[htbp]
\centering
\begin{subfigure}{0.32\textwidth}
    \centering
    \begin{tikzpicture}[
  scale=.4,
  transform shape,
  >=stealth',
  auto,
  thick,
  %
  sty1/.style={
    circle,
    minimum size=5ex,
    inner sep=0pt,
    fill=gray!15,
    draw=gray!70,
    thick,
    dashed,
    font=\sffamily\bfseries,
  },
  %
  sty2/.style={
    diamond,
    minimum size=6ex,
    inner sep=0pt,
    fill=blue!35,
    draw=blue!80!black,
    thick,
    font=\sffamily\bfseries,
  },
  sty3/.style={
    rectangle,
    minimum width=4.5ex,
    minimum height=4.5ex,
    rounded corners=2pt,
    fill=red!35,
    draw=red!80!black,
    thick,
    font=\sffamily\bfseries,
  },
  sty4/.style={
    regular polygon,
    regular polygon sides=5,
    minimum size=6ex,
    inner sep=0pt,
    fill=yellow!40,
    draw=yellow!70!black,
    thick,
    font=\sffamily\bfseries,
  },
  sty5/.style={
    star,
    star points=5,
    minimum size=5ex,
    inner sep=0pt,
    fill=green!35,
    draw=green!60!black,
    thick,
    font=\sffamily\bfseries,
  },
  sty6/.style={
    regular polygon,
    regular polygon sides=6,
    minimum size=5ex,
    inner sep=0pt,
    fill=violet!35,
    draw=violet!70!black,
    thick,
    font=\sffamily\bfseries,
  },
  sty7/.style={
    regular polygon,
    regular polygon sides=8,
    minimum size=5ex,
    inner sep=0pt,
    fill=orange!40,
    draw=orange!80!black,
    thick,
    font=\sffamily\bfseries,
  },
  sty8/.style={
    circle,
    minimum size=5ex,
    inner sep=0pt,
    fill=cyan!30,
    draw=cyan!80!black,
    thick,
    font=\sffamily\bfseries,
  },
  sty9/.style={
    rectangle,
    minimum width=4.2ex,
    minimum height=4.2ex,
    inner sep=0pt,
    fill=magenta!30,
    draw=magenta!75!black,
    double,
    semithick,
    font=\sffamily\bfseries,
  },
  sty10/.style={
    regular polygon,
    regular polygon sides=3,
    minimum size=4.4ex,
    inner sep=0pt,
    fill=red!25!orange!45,
    draw=red!80!black,
    thick,
    font=\sffamily\bfseries,
  },
  sty11/.style={
    ellipse,
    minimum width=5.5ex,
    minimum height=4ex,
    inner sep=0pt,
    fill=teal!25,
    draw=teal!80!black,
    thick,
    font=\sffamily\bfseries,
  }
]

\def\styleList{{
    "sty1", 
    "sty3", 
    "sty3", 
    "sty3", 
    "sty2", 
    "sty2", 
    "sty2", 
    "sty3", 
    "sty3", 
    "sty3", 
    "sty2", 
    "sty11", 
    "sty3", 
    "sty3", 
    "sty3", 
    "sty3", 
    "sty2", 
    "sty3", 
    "sty3", 
    "sty3", 
    "sty3", 
    "sty3", 
    "sty3", 
    "sty3", 
    "sty3", 
    "sty3", 
    "sty3", 
    "sty3", 
    "sty3", 
    "sty3", 
    "sty3", 
    "sty3", 
    "sty3", 
    "sty3"  
}}
                 
\foreach [count=\i] \name/\vtx/\x/\y in
   { 1/0/15.5211316596282/41.0925504659646,
     2/1/14.6390815702147/33.7269542345312,
     3/2/18.2112847875881/27.1660631628311,
     4/3/2.91137130520034/31.1035010896308,
     5/4/20.6221982850349/56.0684161971344,
     6/5/10.9027663935224/53.9950027520863,
     7/6/14.4855568384822/57.9297127520071,
     8/7/6.17121144000743/37.298845170394,
     9/8/23.4773971576921/22.4742150655837,
    10/9/9.82770988134235/15.6136283527571,
    11/10/23.086699286335/51.2128793953601,
    12/11/29.2322906856967/48.4193234425546,
    13/12/0.55/39.7024040650424,
    14/13/8.05597879799012/26.2501984343514,
    15/14/22.9907261607264/0.55,
    16/15/18.122707203845/2.33663683106423,
    17/16/9.76086350345417/61.162177007488,
    18/17/25.3425174790231/41.883373696669,
    19/18/29.0523816114879/1.25228807736234,
    20/19/23.9690771969756/30.850916135437,
    21/20/41.9960199044359/10.4921075296759,
    22/21/6.38761511368505/44.9441642578604,
    23/22/13.8330379043146/11.58374465326659,
    24/23/38.9875226611639/16.2165693313103,
    25/24/42.4177271499389/27.5322772135241,
    26/25/43.8725928142509/22.0754523644347,
    27/26/34.2769364671048/2.39168442379882,
    28/27/36.0924435922861/21.4727181445174,
    29/28/28.6887912037747/27.4400642389826,
    30/29/38.0987050326266/6.33523851750264,
    31/30/15.6740692174605/19.0960440954581,
    32/31/34.3654379473733/26.7049121440653,
    33/32/21.1954250619251/10.84851085183699,
    34/33/31.8965595144377/17.0453513984002    }
{
    {
    \pgfmathparse{\styleList[\i-1]} 
    \let\stileCorrente\pgfmathresult
    \node[\stileCorrente] (\vtx) at ($(0,0)!0.25!(\x,\y)$) {\small\name};
}
}

\foreach \tail/\head in
{0/1,0/3,0/4,0/5,0/6,0/7,0/8,0/10,0/11,0/12,0/13,0/17,0/19,0/21,0/31}
\path[dashed, draw=gray!50] (\tail) edge (\head);
    
\foreach \tail/\head in
    {32/22, 32/30,32/2,32/8,32/31,32/33,32/23,32/20,32/29,32/18,32/14,32/15, 
     1/3,1/7,1/13,1/17,1/19,1/21,1/30,
     3/7,3/12,3/13,
     8/30,
     23/25,23/27,23/29,
     24/25,24/27,
     26/29,
     1/2, 2/3,2/7,2/8,2/9,2/13,2/27,2/28,
     31/28,31/24,31/25,
     33/19, 33/28,33/31,33/27,33/23,33/20,33/29,33/26,33/18,33/14,33/15,33/32,33/22,33/9,33/30,33/8}
    \path[draw=red!80!black] (\tail) edge (\head);

\foreach \tail/\head in
    {4/6,4/10,
     5/6,5/10,5/16,6/16}
    \path[draw=blue!80!black] (\tail) edge (\head);

\end{tikzpicture} 
    \subcaption{$k = 3$ \label{subfig:karatek3}}
\end{subfigure}
\begin{subfigure}{0.32\textwidth}
    \centering
    \begin{tikzpicture}[
  scale=.4,
  transform shape,
  >=stealth',
  auto,
  thick,
  %
  sty1/.style={
    circle,
    minimum size=5ex,
    inner sep=0pt,
    fill=gray!15,
    draw=gray!70,
    thick,
    dashed,
    font=\sffamily\bfseries,
  },
  %
  sty2/.style={
    diamond,
    minimum size=6ex,
    inner sep=0pt,
    fill=blue!35,
    draw=blue!80!black,
    thick,
    font=\sffamily\bfseries,
  },
  sty3/.style={
    rectangle,
    minimum width=4.5ex,
    minimum height=4.5ex,
    rounded corners=2pt,
    fill=red!35,
    draw=red!80!black,
    thick,
    font=\sffamily\bfseries,
  },
  sty4/.style={
    regular polygon,
    regular polygon sides=5,
    minimum size=6ex,
    inner sep=0pt,
    fill=yellow!40,
    draw=yellow!70!black,
    thick,
    font=\sffamily\bfseries,
  },
  sty5/.style={
    star,
    star points=5,
    minimum size=5ex,
    inner sep=0pt,
    fill=green!35,
    draw=green!60!black,
    thick,
    font=\sffamily\bfseries,
  },
  sty6/.style={
    regular polygon,
    regular polygon sides=6,
    minimum size=5ex,
    inner sep=0pt,
    fill=violet!35,
    draw=violet!70!black,
    thick,
    font=\sffamily\bfseries,
  },
  sty7/.style={
    regular polygon,
    regular polygon sides=8,
    minimum size=5ex,
    inner sep=0pt,
    fill=orange!40,
    draw=orange!80!black,
    thick,
    font=\sffamily\bfseries,
  },
  sty8/.style={
    circle,
    minimum size=5ex,
    inner sep=0pt,
    fill=cyan!30,
    draw=cyan!80!black,
    thick,
    font=\sffamily\bfseries,
  },
  sty9/.style={
    rectangle,
    minimum width=4.2ex,
    minimum height=4.2ex,
    inner sep=0pt,
    fill=magenta!30,
    draw=magenta!75!black,
    double,
    semithick,
    font=\sffamily\bfseries,
  },
  sty10/.style={
    regular polygon,
    regular polygon sides=3,
    minimum size=4.4ex,
    inner sep=0pt,
    fill=red!25!orange!45,
    draw=red!80!black,
    thick,
    font=\sffamily\bfseries,
  },
  sty11/.style={
    ellipse,
    minimum width=5.5ex,
    minimum height=4ex,
    inner sep=0pt,
    fill=teal!25,
    draw=teal!80!black,
    thick,
    font=\sffamily\bfseries,
  }
]

\def\styleList{{
    "sty1", 
    "sty1", 
    "sty3", 
    "sty3", 
    "sty2", 
    "sty2", 
    "sty2", 
    "sty3", 
    "sty3", 
    "sty3", 
    "sty2", 
    "sty11", 
    "sty3", 
    "sty3", 
    "sty3", 
    "sty3", 
    "sty2", 
    "sty7", 
    "sty3", 
    "sty3", 
    "sty3", 
    "sty4", 
    "sty3", 
    "sty3", 
    "sty3", 
    "sty3", 
    "sty3", 
    "sty3", 
    "sty3", 
    "sty3", 
    "sty3", 
    "sty3", 
    "sty3", 
    "sty3"  
}}
                 
\foreach [count=\i] \name/\vtx/\x/\y in
   { 1/0/15.5211316596282/41.0925504659646,
     2/1/14.6390815702147/33.7269542345312,
     3/2/18.2112847875881/27.1660631628311,
     4/3/2.91137130520034/31.1035010896308,
     5/4/20.6221982850349/56.0684161971344,
     6/5/10.9027663935224/53.9950027520863,
     7/6/14.4855568384822/57.9297127520071,
     8/7/6.17121144000743/37.298845170394,
     9/8/23.4773971576921/22.4742150655837,
    10/9/9.82770988134235/15.6136283527571,
    11/10/23.086699286335/51.2128793953601,
    12/11/29.2322906856967/48.4193234425546,
    13/12/0.55/39.7024040650424,
    14/13/8.05597879799012/26.2501984343514,
    15/14/22.9907261607264/0.55,
    16/15/18.122707203845/2.33663683106423,
    17/16/9.76086350345417/61.162177007488,
    18/17/25.3425174790231/41.883373696669,
    19/18/29.0523816114879/1.25228807736234,
    20/19/23.9690771969756/30.850916135437,
    21/20/41.9960199044359/10.4921075296759,
    22/21/6.38761511368505/44.9441642578604,
    23/22/13.8330379043146/11.58374465326659,
    24/23/38.9875226611639/16.2165693313103,
    25/24/42.4177271499389/27.5322772135241,
    26/25/43.8725928142509/22.0754523644347,
    27/26/34.2769364671048/2.39168442379882,
    28/27/36.0924435922861/21.4727181445174,
    29/28/28.6887912037747/27.4400642389826,
    30/29/38.0987050326266/6.33523851750264,
    31/30/15.6740692174605/19.0960440954581,
    32/31/34.3654379473733/26.7049121440653,
    33/32/21.1954250619251/10.84851085183699,
    34/33/31.8965595144377/17.0453513984002    }
{
    {
    \pgfmathparse{\styleList[\i-1]} 
    \let\stileCorrente\pgfmathresult
    \node[\stileCorrente] (\vtx) at ($(0,0)!0.25!(\x,\y)$) {\small\name};
}
}

\foreach \tail/\head in
    {0/1,0/3,0/4,0/5,0/6,0/7,0/8,0/10,0/11,0/12,0/13,0/17,0/19,0/21,0/31,
    1/2, 1/3,1/7,1/13,1/17,1/19,1/21,1/30}
    \path[dashed, draw=gray!50] (\tail) edge (\head);

\foreach \tail/\head in
    {4/6,4/10,
     5/6,5/10,5/16,6/16}
    \path[draw=blue!80!black] (\tail) edge (\head);

\foreach \tail/\head in
    {3/7, 3/12, 3/13,
     32/22, 32/30,32/2,32/8,32/31,32/33,32/23,32/20,32/29,32/18,32/14,32/15,
     8/30,
     23/25,23/27,23/29,
     24/25,24/27,
     26/29,
     2/3,2/7,2/8,2/9,2/13,2/27,2/28,
     31/28,31/24,31/25,
     33/19, 33/28,33/31,33/27,33/23,33/20,33/29,33/26,33/18,33/14,33/15,33/32,33/22,33/9,33/30,33/8}
    \path[draw=red!80!black] (\tail) edge (\head);

\end{tikzpicture} 
    \subcaption{$k = 5$ \label{subfig:karatek5}}
\end{subfigure}
\begin{subfigure}{0.32\textwidth}
    \centering
    \begin{tikzpicture}[
  scale=.4,
  transform shape,
  >=stealth',
  auto,
  thick,
  %
  sty1/.style={
    circle,
    minimum size=5ex,
    inner sep=0pt,
    fill=gray!15,
    draw=gray!70,
    thick,
    dashed,
    font=\sffamily\bfseries,
  },
  %
  sty2/.style={
    diamond,
    minimum size=6ex,
    inner sep=0pt,
    fill=blue!35,
    draw=blue!80!black,
    thick,
    font=\sffamily\bfseries,
  },
  sty3/.style={
    rectangle,
    minimum width=4.5ex,
    minimum height=4.5ex,
    rounded corners=2pt,
    fill=red!35,
    draw=red!80!black,
    thick,
    font=\sffamily\bfseries,
  },
  sty4/.style={
    regular polygon,
    regular polygon sides=5,
    minimum size=6ex,
    inner sep=0pt,
    fill=yellow!40,
    draw=yellow!70!black,
    thick,
    font=\sffamily\bfseries,
  },
  sty5/.style={
    star,
    star points=5,
    minimum size=5ex,
    inner sep=0pt,
    fill=green!35,
    draw=green!60!black,
    thick,
    font=\sffamily\bfseries,
  },
  sty6/.style={
    regular polygon,
    regular polygon sides=6,
    minimum size=5ex,
    inner sep=0pt,
    fill=violet!35,
    draw=violet!70!black,
    thick,
    font=\sffamily\bfseries,
  },
  sty7/.style={
    regular polygon,
    regular polygon sides=8,
    minimum size=5ex,
    inner sep=0pt,
    fill=orange!40,
    draw=orange!80!black,
    thick,
    font=\sffamily\bfseries,
  },
  sty8/.style={
    circle,
    minimum size=5ex,
    inner sep=0pt,
    fill=cyan!30,
    draw=cyan!80!black,
    thick,
    font=\sffamily\bfseries,
  },
  sty9/.style={
    rectangle,
    minimum width=4.2ex,
    minimum height=4.2ex,
    inner sep=0pt,
    fill=magenta!30,
    draw=magenta!75!black,
    double,
    semithick,
    font=\sffamily\bfseries,
  },
  sty10/.style={
    regular polygon,
    regular polygon sides=3,
    minimum size=4.4ex,
    inner sep=0pt,
    fill=red!25!orange!45,
    draw=red!80!black,
    thick,
    font=\sffamily\bfseries,
  },
  sty11/.style={
    ellipse,
    minimum width=5.5ex,
    minimum height=4ex,
    inner sep=0pt,
    fill=teal!25,
    draw=teal!80!black,
    thick,
    font=\sffamily\bfseries,
  }
]


\def\styleList{{
    "sty1", 
    "sty4", 
    "sty1", 
    "sty4", 
    "sty2", 
    "sty2", 
    "sty2", 
    "sty4", 
    "sty4", 
    "sty10", 
    "sty2", 
    "sty11", 
    "sty4", 
    "sty4", 
    "sty7", 
    "sty8", 
    "sty2", 
    "sty4", 
    "sty6", 
    "sty4", 
    "sty5", 
    "sty4", 
    "sty9", 
    "sty3", 
    "sty3", 
    "sty3", 
    "sty3", 
    "sty3", 
    "sty3", 
    "sty3", 
    "sty4", 
    "sty3", 
    "sty1", 
    "sty1"  
}}
                 
\foreach [count=\i] \name/\vtx/\x/\y in
   { 1/0/15.5211316596282/41.0925504659646,
     2/1/14.6390815702147/33.7269542345312,
     3/2/18.2112847875881/27.1660631628311,
     4/3/2.91137130520034/31.1035010896308,
     5/4/20.6221982850349/56.0684161971344,
     6/5/10.9027663935224/53.9950027520863,
     7/6/14.4855568384822/57.9297127520071,
     8/7/6.17121144000743/37.298845170394,
     9/8/23.4773971576921/22.4742150655837,
    10/9/9.82770988134235/15.6136283527571,
    11/10/23.086699286335/51.2128793953601,
    12/11/29.2322906856967/48.4193234425546,
    13/12/0.55/39.7024040650424,
    14/13/8.05597879799012/26.2501984343514,
    15/14/22.9907261607264/0.55,
    16/15/18.122707203845/2.33663683106423,
    17/16/9.76086350345417/61.162177007488,
    18/17/25.3425174790231/41.883373696669,
    19/18/29.0523816114879/1.25228807736234,
    20/19/23.9690771969756/30.850916135437,
    21/20/41.9960199044359/10.4921075296759,
    22/21/6.38761511368505/44.9441642578604,
    23/22/13.8330379043146/11.58374465326659,
    24/23/38.9875226611639/16.2165693313103,
    25/24/42.4177271499389/27.5322772135241,
    26/25/43.8725928142509/22.0754523644347,
    27/26/34.2769364671048/2.39168442379882,
    28/27/36.0924435922861/21.4727181445174,
    29/28/28.6887912037747/27.4400642389826,
    30/29/38.0987050326266/6.33523851750264,
    31/30/15.6740692174605/19.0960440954581,
    32/31/34.3654379473733/26.7049121440653,
    33/32/21.1954250619251/10.84851085183699,
    34/33/31.8965595144377/17.0453513984002    }
{
    {
    \pgfmathparse{\styleList[\i-1]} 
    \let\stileCorrente\pgfmathresult
    \node[\stileCorrente] (\vtx) at ($(0,0)!0.25!(\x,\y)$) {\small\name};
}
}

          \foreach \tail/\head in
    {32/22, 32/30,32/2,32/8,32/31,32/33,32/23,32/20,32/29,32/18,32/14,32/15,
    0/1,0/3,0/4,0/5,0/6,0/7,0/8,0/10,0/11,0/12,0/13,0/17,0/19,0/21,0/31, 
    1/2, 2/3,2/7,2/8,2/9,2/13,2/27,2/28}
    \path[dashed,draw=gray!50] (\tail) edge (\head);  
    
       \foreach \tail/\head in
    {33/19, 33/28,33/31,33/27,33/23,33/20,33/29,33/26,33/18,33/14,33/15,33/32,33/22,33/9,33/30,33/8}
    \path[dashed,draw=gray!50] (\tail) edge (\head);

    \foreach \tail/\head in
    {1/3,1/7,1/13,1/17,1/19,1/21,1/30,
     3/7,3/12,3/13, 8/30}
    \path[draw=yellow!60!black] (\tail) edge (\head); 

    \foreach \tail/\head in
    {23/25,23/27,23/29,
     24/25,24/27,
     26/29,
     31/28,31/24,31/25}
    \path[draw=red!80!black] (\tail) edge (\head);

\foreach \tail/\head in
    {4/6,4/10, 5/6,5/10,5/16,6/16}
    \path[draw=blue!80!black] (\tail) edge (\head);

\end{tikzpicture} 
    \subcaption{$k = 10$ \label{subfig:karatek10}}
\end{subfigure}
\caption{Optimal $k$-vertex cuts (gray vertices with dashed incident edges) for Zachary’s karate club network with unit vertex costs, shown for $k \in \{3,5,10\}$. The colored vertices represent the connected components remaining in the graph.}
\label{fig:karate_example}
\end{figure}

As for the specific applications of the $k$-VCP, this problem plays a central role whenever the decision maker aims to disrupt communication, isolate subnetworks, or assess the robustness and structural vulnerability of complex systems; see e.g., \cite{FurLMP2020}.
The computational complexity of the $k$-VCP depends critically on the value of~$k$. 
When the deletion budget is taken as parameter, the decision version of the problem is 
W[1]-hard \citep{Marx}. 
If $k$ is part of the input, the problem is already $\mathcal{NP}$\nobreakdash-hard, 
as follows from the hardness of the $k$-way vertex cut problem \citep{BGZ14}, 
whose decision version coincides with the decision version of the $k$-VCP. 
Moreover, even when $k$ is fixed, the problem remains $\mathcal{NP}$\nobreakdash-hard 
for every $k \ge 3$ \citep{CFLMMM18}.
On the other hand, the case $k = 2$ is polynomially solvable, as it reduces to 
computing $O(|\mathcal V|^2)$ minimum $s$--$t$ cuts via vertex-splitting 
over all pairs of vertices, which in the unit-cost case coincides with 
computing the vertex connectivity of the graph.

Three main approaches currently represent the most effective exact methods for solving the $k$-VCP. 
A first, off-the-shelf option consists in using a general-purpose \textit{integer linear programming} (ILP) solver applied to a compact ILP formulation with a polynomial number of variables and constraints, as proposed in~\cite{CFLMMM18}. 
Beyond this generic strategy, two dedicated exact algorithms have been introduced in the literature: 
a \textit{branch-and-price} (BP) algorithm~\citep{CFLMMM18}, based on an ILP formulation with an exponential number of variables, 
and a \textit{branch-and-cut} (BC) algorithm~\citep{FurLMP2020}, based on an ILP formulation with an exponential number of constraints.
Although these exact methods are quite effective, they still fail to solve many benchmark instances to optimality. 
Motivated by these limitations, the main methodological contribution of this paper is to introduce a new exact algorithm for the $k$-VCP that markedly improves upon the BP approach and enables the solution of significantly larger and more challenging instances.

\subsection{ILP models for the \texorpdfstring{$k$}{k}-VCP from the literature}
\label{sec:lit}

In this section we recall the three ILP formulations that have been proposed in the literature for the $k$-VCP.
For each model, we outline its modeling structure and the exact solution method originally associated with it.
These models, together with their corresponding algorithms, will be included in the computational experiments of this article and used as benchmarks against the new BP algorithm proposed in this paper.

\subsubsection{The compact ILP model}
\label{sec:comp}

Let $\mathcal{K}$ denote the set of cluster indices, namely $\{1,2,\dots, k\}$. For each vertex $v \in \vertices$ and each index $i \in \mathcal{K}$, let $y_{iv}$ be a binary variable taking value $1$ if~$v$ is assigned to the $i$-th cluster of the graph and value $0$ otherwise.
Vertices assigned to at least one cluster are retained in the graph, whereas vertices 
not assigned to any cluster belong to the set $\vertices_0$ and form the 
$k$-vertex cut.
Using these variables,
the \textit{compact} ILP formulation for the $k$-VCP, introduced by \cite{CFLMMM18} and 
denoted by \ILPcomp, reads as follows:
\begin{subequations}
\begin{align}
\label{VKSP_CF_OBJ}  \text{(\ILPcomp)}  & & \max_{{\boldsymbol{y} \in \{0,1\}^{\mathcal K \times \vertices}}}  \sum_{i\in \mathcal K}\sum_{v\in \mathcal V} c_v \, y_{iv} \\[1 ex]
\label{VKSP_CF_1}  & &  \sum_{i \in \mathcal K} y_{iv} &\le 1, & v\in \mathcal V,\\[1 ex]
\label{VKSP_CF_2_bis}& & y_{iu} + \sum_{j\in \mathcal K \setminus \{i\}} y_{jv} & \leq 1, &   i \in \mathcal K, \, \{u,v\} \in \mathcal E, \\[1 ex]
\label{VKSP_CF_3}   & &  \sum_{v \in \mathcal V} y_{iv} & \ge 1, & i \in \mathcal K. 
\end{align}
\end{subequations}
This formulation is closely related to the one by \cite{BFM98} for the $k$-separator 
problem and uses the same assignment variables. 
The objective function~\eqref{VKSP_CF_OBJ} maximizes the total {cost} of vertices assigned to the clusters of the graph. {Throughout the paper, bold symbols below the optimization operator indicate the decision variables of the model together with their domains.}
Constraints~\eqref{VKSP_CF_1} ensure that each vertex is assigned to at most one 
cluster.  
Constraints~\eqref{VKSP_CF_2_bis} prevent adjacent vertices from being assigned to 
different clusters, thereby ensuring mutual disconnection of the clusters.  
Constraints~\eqref{VKSP_CF_3} require each cluster to be non-empty, enforcing the 
presence of at least $k$ connected components.
The compact formulation involves a polynomial number of binary variables and 
constraints, and can therefore be solved directly by a general-purpose ILP solver.

\subsubsection{The extended ILP model}
\label{sec:ext}

Let $\potatoes {\subseteq 2^{\vertices}}$ be the family of all non-empty subsets of vertices, where $2^\vertices$ denotes the power set of $\vertices$. {Hence, $\potatoes = \{ S \subseteq \vertices : S \neq \emptyset \}$}.
Furthermore, for any 
$U \subseteq \vertices$, let~$\potatoes(U) = \{ S \in \potatoes : S \cap U \neq \emptyset \}$
denote the subsets in $\potatoes$ that contain at least one vertex of $U$.
For singleton sets~$\{v\}$, $v \in \vertices$, we write~$\potatoes(v)$ instead of~$\potatoes(\{v\})$, i.e., $\potatoes(v) = \{ S \in \potatoes : v \in S \}$ denotes the subsets in $\potatoes$ that contain vertex $v$.
For each subset $S \in \potatoes$, let $\potvar_S$ be a binary variable taking value $1$ if $S$ is selected as one of the $k$ clusters of the solution and value $0$ otherwise.
Denote by~$\cliques$ the family of all cliques in~$\graph$.
Let $\mathscr{C}' \subseteq \cliques$ be an \textit{edge-covering} family of cliques of $\graph$, 
that is, a family such that for every edge $\{u,v\} \in \edges$ there exists at least one 
clique~$C \in \mathscr{C}'$ containing both $u$ and $v$. Using the~$\potvarvec$-variables, the \textit{extended} ILP formulation for the $k$-VCP, introduced by 
\cite{CFLMMM18} and denoted by \oldBP, reads as follows:
\begin{subequations}
  \label{eq:oldbp}
  \begin{align}
    \text{(\oldBP)}   
        & & \max_{{\potvarvec \in \{0,1\}^{\potatoes}}} \sum_{S \in \potatoes} \bigg( \sum_{v \in S} c_v \bigg)  \, \potvar_S    
        & \label{eq:EF_obj} \\[1ex]
    & &  \sum_{S \in \potatoes} \potvar_S 
        & = k, 
        \label{eq:EF_3} \\[1ex]
    & &  \sum_{S \in \potatoes(v)} \potvar_S 
        & \le 1, 
        & v \in \vertices,
        \label{eq:EF_1} \\[1ex]
    & & \sum_{S \in \potatoes(C)} \potvar_S
        & \le 1, 
        & C \in \mathscr{C}',
        \label{eq:EF_2}
  \end{align}
\end{subequations}
The objective function~\eqref{eq:EF_obj} maximizes the total cost of the $k$ clusters 
selected in the solution. It generalizes the original unweighted version—which 
maximizes the cardinality $|S|$ of each chosen subset—by incorporating vertex 
costs.
Constraint~\eqref{eq:EF_3} ensures that exactly $k$ subsets are selected, and thus 
that the solution consists of $k$ clusters. 
Constraints~\eqref{eq:EF_1} require that each vertex $v \in \vertices$ appears in 
at most one selected subset, guaranteeing that the chosen subsets are vertex-disjoint 
and therefore represent distinct clusters. 
Constraints~\eqref{eq:EF_2} enforce that the selected clusters are mutually 
disconnected. Since $\mathscr{C}'$ is an edge-covering family of cliques, requiring 
that at most one subset in $\potatoes(C)$ can be selected for each $C \in \mathscr{C}'$ 
ensures that no two chosen subsets contain adjacent vertices. Consequently, the $k$ 
selected clusters of vertices are pairwise disconnected in the graph.
{Figure~\ref{fig:ef2_example} illustrates the logic of  Constraints~\eqref{eq:EF_2} using a small graph with six vertices and an edge-covering family of cliques given by its edges. As an example, consider the 
clique $C=\{b,e\}$ and two potential clusters $S_1=\{a,b,c\}$ and $S_2=\{d,e,f\}$. Since $S_1 \cap C=\{b\}$ and~$S_2 \cap C=\{e\}$, both clusters 
belong to~$\potatoes(C)$. Therefore, Constraint~\eqref{eq:EF_2} associated with $C$ prevents the variables corresponding to $S_1$ and $S_2$ from both taking 
value~$1$. Indeed, selecting both clusters would violate the requirement that the selected clusters be pairwise disconnected.} In \cite{CFLMMM18}, a family of cliques of polynomial size with respect to the number of edges of the graph is considered.

\begin{figure}[htbp]
\centering
\begin{tikzpicture}[scale=0.9]
  \tikzstyle{v} += [circle,draw=black,thick,inner sep=2pt,minimum size=2mm];

  \def\scale{2}

  \coordinate (A) at (0,1.05);
  \coordinate (C) at (0,-1.05);
  \coordinate (B) at (1.65,0);
  \coordinate (E) at (4.15,0);
  \coordinate (D) at (5.8,1.05);
  \coordinate (F) at (5.8,-1.05);

  \coordinate (G1) at (barycentric cs:A=1,B=1,C=1);
  \coordinate (G2) at (barycentric cs:D=1,E=1,F=1);

  \draw[fill=blue!30, opacity=0.5, rounded corners=20pt] 
    ($(G1)!\scale!(A)$) -- ($(G1)!\scale!(B)$) -- ($(G1)!\scale!(C)$) -- cycle;

  \draw[fill=red!30, opacity=0.5, rounded corners=20pt] 
    ($(G2)!\scale!(D)$) -- ($(G2)!\scale!(E)$) -- ($(G2)!\scale!(F)$) -- cycle;

  \node (a) at (A) [v,label=below:{\scriptsize $a$}] {};
  \node (c) at (C) [v,label=below:{\scriptsize $c$}] {};
  \node (b) at (B) [v,label=below:{\scriptsize $b$}] {};
  \node (e) at (E) [v,label=below:{\scriptsize $e$}] {};
  \node (d) at (D) [v,label=below:{\scriptsize $d$}] {};
  \node (f) at (F) [v,label=below:{\scriptsize $f$}] {};

  \draw[-,thick] (a) -- (b) -- (c);
  \draw[-,thick] (b) -- (e);
  \draw[-,thick] (d) -- (e) -- (f);

  \node[blue] at (0.65,0) {\scriptsize $S_1$};
  \node[red] at (5.15,0) {\scriptsize $S_2$};
\end{tikzpicture}
\caption{{Illustration of Constraint~\eqref{eq:EF_2}. The candidate clusters $S_1=\{a,b,c\}$ (blue) and $S_2=\{d,e,f\}$ (red) both intersect the clique $C=\{b,e\}$, so Constraint~\eqref{eq:EF_2} forbids selecting them simultaneously.}}
\label{fig:ef2_example}
\end{figure}

The extended formulation \oldBP{} contains an exponential number of variables. 
As a consequence, solving the model requires a \textit{column generation} (CG) 
approach embedded within a BP algorithm. At each node of the 
branching tree, the LP relaxation of Formulation~\eqref{eq:oldbp} is solved through 
a CG procedure, where new variables are generated by solving the so-called \textit{pricing 
problem} (PP). We refer the interested reader to~\citet{BarnhartJNSV98,LubbeckeD05} 
for further details on CG and BP algorithms, as well as to the book by 
\citet{DesaulniersDesrosiersSolomon2005} and the two recent monographs by~\citet{G-2024-36} and~\citet{uchoa2024optimizing}. 
In what follows, we briefly describe the two principal components of the BP algorithm of 
\cite{CFLMMM18}: the pricing problem  and the 
branching scheme.

At the root node of the BP algorithm, \citet{CFLMMM18} show that the pricing problem can be solved via solving a sequence of minimum cut problems on an auxiliary network, enabling the use of efficient max-flow/min-cut algorithms.
Although the pricing problem is originally described for unit vertex costs, the same approach extends naturally to the weighted setting by simply adjusting the vertex profits in the reduced-cost computation, and both the complexity and the solution methods remain unchanged.
We refer the reader to Section~\ref{sec:priceRMP} for further details.

A dedicated branching scheme is proposed in \citet{CFLMMM18}, operating on two
levels. At the first level, a ``partially'' assigned vertex~$v \in \vertices$ is identified,
namely one for which the sum of the variables over subsets containing~$v$ is
fractional. Two branches are created: one forcing~$v$ into the vertex cut and the
other forcing~$v$ to belong to exactly one cluster. This branching preserves the
structure of the pricing problem: in the first branch,~$v$ is simply forbidden in
newly generated subsets, while in the second the pricing formulation is unchanged.
Since this is not sufficient to guarantee integrality, a second branching
level, akin to Ryan-Foster branching~\citep{RyanFoster81}, is applied to vertex pairs. Here, two
vertices~$u \in \vertices$ and~$v \in \vertices$ outside the vertex cut that are ``partially'' assigned to the
same cluster (i.e., the sum of variables over subsets containing both is
fractional) trigger two additional branches: enforcing that they belong to the same
cluster or to different clusters.
In the
``same-cluster'' branch, the pricing problem is preserved by contracting
$\{u,v\}$ into a single supervertex, whereas in the ``different-clusters'' branch
an incompatibility constraint forbidding subsets containing both vertices must be
enforced. The latter destroys the min-cut structure of the pricing problem, which
must then be solved using an ILP formulation.

\subsubsection{The natural ILP model}
\label{sec:nat}

A \emph{subgraph} of $\graph$ is any graph 
$\subgraph=(\SUBvertices,\SUBedges)$ such that 
$\SUBvertices\subseteq \vertices$ and $\SUBedges\subseteq \edges$, with every edge 
of $\SUBedges$ having both endpoints in $\SUBvertices$.
For any subgraph $\subgraph$ of $\graph$, we denote by $\deg_{\subgraph}(v)$ the degree 
of vertex $v$ in $\subgraph$, i.e., the number of edges of $\subgraph$ incident to $v$.
A subgraph~$\subgraph$ is \emph{spanning} if $\SUBvertices=\vertices$, and it is 
\emph{cycle-free} if it contains no cycles (equivalently, if $\subgraph$ is a \textit{forest}).
We denote by $\SUBgraphs$ the family of all \emph{spanning cycle-free subgraphs} 
of $\graph$, that is,
\[
\SUBgraphs = \{\, \subgraph = (\vertices, \SUBedges) : \subgraph \text{ is acyclic} \,\}.
\]
For each vertex $v \in \vertices$, we introduce a binary variable $x_v$ that takes value~1 
if $v$ belongs to the $k$-vertex cut and~0 otherwise.
Using the~$\bx$-variables,  
the \textit{natural} ILP formulation for the $k$-VCP, introduced by \cite{FurLMP2020} 
and denoted by \ILPnat, reads as follows:
\begin{subequations}
\label{BILEVEL}
\begin{align}
\label{INT_R_OBJ}
\text{(\ILPnat)}
& & \min_{{\bx \in \{0,1\}^{\vertices}}} ~~\sum_{v \in \vertices} {c_v} \, x_v \\[1ex]
\label{INT_R_0}  
& &  \sum_{v \in \vertices} \big( \deg_{\subgraph}(v) - 1 \big) x_v 
    & \ge k - |\vertices| + |\SUBedges|, 
    & \subgraph \in \SUBgraphs.
\end{align}
\end{subequations}
The objective function~\eqref{INT_R_OBJ} minimizes the total weight of the vertices
removed in the $k$-vertex cut. Constraints~\eqref{INT_R_0} form an exponential family
of inequalities derived from a bilevel interpretation of the $k$-VCP and from the key
property that a graph has at least $k$ connected components if and only if every
acyclic {spanning} subgraph contains at most $|\vertices|-k$ edges. We refer to
\cite{FurLMP2020} for further details. 

The natural formulation \ILPnat{} contains an exponential number of constraints.
As a consequence, solving the model requires a \textit{cut generation} approach
embedded within a branch-and-cut algorithm. At each node of the branching tree,
the LP relaxation of Formulation~\eqref{BILEVEL} is solved by iteratively
generating violated inequalities from the family~\eqref{INT_R_0}. These cuts are
produced by solving a separation problem that searches for a maximum-weight
acyclic subgraph of the current residual graph. As shown by~\citet{FurLMP2020},
this separation can be performed in polynomial time for both integer and
fractional solutions, allowing the exponential family of constraints to be
added only when violated.

\subsection{Paper outline and contributions}

This section reviews the main contributions of our work and specifies the sections in which they are presented.
Section~\ref{sec:NewForm} presents our main theoretical contribution: a new extended ILP formulation that unifies and extends both the natural {formulation} and the previous extended formulation by simultaneously making use of the binary variables from both models.
Crucially, we prove that integrality needs to be enforced only on the variables inherited from the natural formulation, whereas the variables coming from the extended formulation can be safely relaxed to continuous without affecting optimality—an essential property that greatly simplifies the design of the new branch-and-price algorithm.
The \textit{linear programming} (LP) relaxation of the new extended model strictly dominates the LP relaxations of both classical models, thereby providing a substantially tighter foundation for exact solution methods.
Building on the new formulation, Section~\ref{sec:ALGO} introduces our main methodological contributions, namely a branch-and-price algorithm that incorporates a tailored column generation procedure, a branching strategy fully compatible with {the} pricing problem, and dedicated symmetry-handling and primal-heuristic components that improve the computational performance.
Section~\ref{sec:COMP} reports computational results on standard benchmark instances from the literature, including comparisons with state-of-the-art exact methods and ablation tests designed to assess the contribution of each main component of our new approach.
The experiments show that our method establishes a new state-of-the-art for the $k$-VCP, outperforming all existing exact algorithms and solving 73 open instances within the time limit of one hour.
Finally, Section~\ref{sec:CONCL} summarizes the contributions of the paper and outlines directions for future research.

\section{A new extended ILP model for the \texorpdfstring{$k$}{k}-VCP}
\label{sec:NewForm}

In this section, we introduce a new extended ILP formulation for the $k$-VCP that
simultaneously exploits the vertex-cut variables $\bx \in \B{\vertices}$ of
\ILPnat{} and the cluster-selection variables $\potvarvec \in \B{\potatoes}$ of
\oldBP. By combining the strengths of both formulations within a unified model,
we obtain an enhanced representation of the $k$-VCP given by \zigbase below.
We will show, through a detailed theoretical study, that the variables inherited from \oldBP can in fact be treated as continuous.
The new model therefore does not only shed light on hidden properties of \oldBP, but has also important algorithmic consequences for solving the~$k$-VCP as we will discuss in detail in Section~\ref{sec:ALGO}. As in \ILPnat, variable~$x_v$ indicates whether vertex~$v \in \vertices$ is contained in the vertex cut, and, analogously to \oldBP, variable~$\potvar_S$ indicates whether~$S \in \potatoes$ forms one of the~$k$ mutually disconnected clusters of the remaining graph.
Using these variables,
the new \textit{extended} ILP formulation for the $k$-VCP  
and denoted by \zigbase, reads as follows:
\begin{subequations}
    \label{eq:zig-base}
        \begin{align}
          \label{eq:zig-base:obj}\text{(\zigbase)} & & \min_{{\substack{\bx \in \{0,1\}^{\vertices}\\ \potvarvec \in \{0,1\}^{\potatoes}}}} ~~ \sum_{v \in \vertices} c_v \, x_v &&&\\[1 ex]
          & & \sum_{S \in \potatoes} \potvar_S &= k, &&\label{eq:zig-base:cardinality}\\[1 ex]
          & &  \sum_{S \in \potatoes(v)} \potvar_S  + x_v&= 1, && v \in \vertices,\label{eq:zig-base:vertex_cover}\\[1 ex]
          & & \sum_{S \in \potatoes(C)} \potvar_S &\leq 1, && C \in \cliques.\label{eq:zig-base:clique_constrs}
        \end{align}
\end{subequations}
The objective function \eqref{eq:zig-base:obj} coincides with the one of \ILPnat and minimizes the total cost of the the selected $k$-vertex cut.
Constraints~\eqref{eq:zig-base:cardinality} and~\eqref{eq:zig-base:clique_constrs} serve the same purpose as in \oldBP, i.e., they guarantee to select exactly~$k$ mutually disconnected clusters and enforce adjacent vertices in the remaining graph to be contained in the same cluster, respectively. {Nonetheless, unlike the formulation of \citet{CFLMMM18}, where the corresponding constraints are imposed on an edge-covering family of cliques, Constraints~\eqref{eq:zig-base:clique_constrs} are defined over the whole clique family~$\cliques$. This yields a stronger formulation, but also introduces an exponential number of inequalities. Therefore, these constraints require a tailored separation procedure, described in Section~\ref{sec:sepaRMP}. Furthermore, it is worth noticing that it is sufficient to define Constraints~\eqref{eq:zig-base:clique_constrs} only over the family of maximal cliques of~$\graph$. Indeed, if $C_1 \subseteq C_2$ for two cliques $C_1,C_2 \in \cliques$, then $\potatoes(C_1) \subseteq \potatoes(C_2)$, and therefore the constraint associated with $C_1$ is dominated by the one associated with $C_2$. 
}
To explain Constraints~\eqref{eq:zig-base:vertex_cover}, observe that~\eqref{eq:EF_1} in \oldBP takes left-hand side value~1 if~$v \in \vertices$ is not removed from the graph and of~0 otherwise.
That is, $x_v$ in~\eqref{eq:zig-base:vertex_cover} can be interpreted as a slack variable in Constraint~\eqref{eq:EF_1}.
The correctness of \zigbase immediately follows from the correctness of \oldBP.

We now discuss the main theoretical contribution of this section. {Notice the domain of the $\potvarvec$-variables.}
\begin{theorem}
  \label{thm:zig-basecont}
  Let~$\graph = (\vertices,\edges)$ be an undirected graph, let~$\bc \in \R^{\vertices}$, and let~$k$ be a positive integer.
  Then,
  \begin{align}
    \label{eq:zig-basecont}
    \min_{{\substack{\bx \in \B{\vertices}\\ \potvarvec \ge \boldsymbol{0}}}}\left\{ \sum_{v \in \vertices} c_vx_v : \, {\text{\eqref{eq:zig-base:cardinality}, \eqref{eq:zig-base:vertex_cover}, and~\eqref{eq:zig-base:clique_constrs}}}\right\}
  \end{align}
  is a correct model of the~$k$-vertex cut problem.
  That is, for every optimal solution~$(\bx,\potvarvec)$ of~\eqref{eq:zig-basecont}, the set of vertices~$\{v \in \vertices : x_v = 1\}$ is an optimal solution of the~$k$-vertex cut problem.
\end{theorem}
In the remainder of this section, we prove Theorem~\ref{thm:zig-basecont} and compare the strength of the LP relaxations of~\ILPnat, \oldBP, and \zigbase.
\subsection{Partially relaxing integrality in \texorpdfstring{\zigbase}{ILPE2}}

\medskip
To prove Theorem~\ref{thm:zig-basecont}, we proceed in two steps.
First, we derive an alternative formulation of the~$k$-VCP that contains both the~$\bx$-variables and~$\potvarvec$-variables as well as additional variables.
An important feature of this model is that integrality is only assumed for the~$\bx$-variables.
Afterwards, we show that, although only the~$\bx$-variables are enforced to take integer values, there always exists an optimal solution in which all variables take integer values.
This fact is exploited by adding further cutting planes and projecting this model to~\eqref{eq:zig-basecont}.

\subsubsection{An alternative formulation}
\label{sec:alternativeForm}
To derive our alternative formulation, we make use of an observation by~\cite{FurLMP2020}.
Let~$\bx \in \B{\vertices}$ and let~$\tilde{\vertices} = \vertices \setminus \{v \in \vertices : x_v = 1\}$.
We define the reduced graph~$\graph(\bx) = (\tilde{\vertices}, \edges[\tilde{\vertices}]) = (\tilde{\vertices}, \tilde{\edges})$.
That is, $\graph(\bx)$ is the induced subgraph of~$\graph$ that arises by removing the vertices~$v \in \vertices$ with~$x_v = 1$.
Let~$\Phi(\bx)$ denote the maximum number of edges in a spanning forest of~$\graph(\bx)$.
Then, \citet{FurLMP2020} have shown that the~$k$-VCP can be modeled as
\begin{equation}
  \label{eq:bilevel}
  \min_{{\bx \in \B{\vertices}}} \left\{\sum_{v \in \vertices} c_v x_v : \Phi(\bx) \leq n - k - \sum_{v \in \vertices} x_v \right\}.
\end{equation}
Since~\eqref{eq:bilevel} involves the non-linear constraint~$\Phi(\bx) \leq n - k - \sum_{v \in \vertices} x_v$, Model~\eqref{eq:bilevel} cannot be solved by standard ILP solvers.
\citet{FurLMP2020} therefore linearize this constraint using cutting planes.
For deriving our alternative model, we instead express~$\Phi(\bx)$ as an LP model.

\begin{lemma}
  \label{lem:treeLP}
  Let~$\graph = (\vertices,\edges)$ be an undirected graph and let~$\bx \in \B{\vertices}$.
  Then,
  \begin{align}
    \Phi(\bx) = \max_{{\boldsymbol{y} \ge \boldsymbol{0}}} \, \sum_{e \in \edges} y_{e} &&&\notag\\[1 ex]
    y(S) &\leq \card{S} - 1, && S \subseteq \vertices,\; \card{S} \geq 2,\tag{$\potvar_S$}\\[1 ex]
    y_{e} &\leq 1 - x_u, && e=\{u,v\} \in \edges,\tag{$\beta_{eu}$}\\[1 ex]
    y_{e} &\leq 1 - x_v, && e=\{u,v\} \in \edges,\tag{$\beta_{ev}$}
  \end{align}
  where~$y(S)$ denotes~$\sum_{e \in \edges[S]} y_v$.
\end{lemma}
\begin{proof}{Proof}
  This is a consequence of a result by \citet{Edmonds1970} that yields an inequality description of the convex hull of incidence vectors of acyclic subgraphs, see Appendix~\ref{sec:proofs} for details.
    \qed
\end{proof}

\medskip
Based on Lemma~\ref{lem:treeLP}, we know that~$\Phi(\bx)$ can be computed by solving an LP model.
That is, Model~\eqref{eq:bilevel} is a bilevel optimization problem with an LP model in the lower level.
We can turn the problem into a single-level problem by dualizing the LP model for computing~$\Phi(\bx)$.
To derive the dual, we use the dual variables as indicated by the Greek letters in brackets in the LP in Lemma~\ref{lem:treeLP}, and introduce the notation~$\potatoessub(U) \define \{S \in \potatoes: U \subseteq S\}$ for~$U \subseteq \vertices$, {denoting the family of all supersets of $U$ in $\potatoes$}.
The dual then reads as
\begin{subequations}
  \label{eq:dualLP}
  \begin{align}
    \Phi(\bx) = \min_{{\substack{\potvarvec \ge \boldsymbol{0}\\ \boldsymbol{\beta} \ge \boldsymbol{0}}}} 
    \, \sum_{\substack{S \in \potatoes\colon\\ \card{S} \geq 2}} (\card{S} - 1)\potvar_S
    +
    \sum_{e=\{u,v\} \in \edges}((1 - x_u)\beta_{eu} + (1 - x_v)\beta_{ev})&&&\label{eq:dualLP:obj}\\[1 ex]
    \sum_{S \in \potatoessub(e)} \potvar_S + \beta_{eu} + \beta_{ev} &\geq 1, && e=\{u,v\} \in \edges.\label{eq:dualLP:cons}
  \end{align}
\end{subequations}
For notational convenience, in the following {we} also include variables~$\potvar_S$ for~$S \in \potatoes$ with~$\card{S} = 1$.
This is possible without altering the structure of the problem, because these variables have neither a contribution to the objective function nor Constraint~\eqref{eq:dualLP:cons}.

Based on the dual, we can substitute~$\Phi(\bx)$ in~\eqref{eq:bilevel} and obtain
\begin{subequations}
  \label{eq:singlelevelNonlin}
  \begin{align}
    \min_{{\substack{\bx \in \{0,1\}^{\vertices}\\ \potvarvec \ge \boldsymbol{0}\\ \boldsymbol{\beta} \ge \boldsymbol{0}}}} \sum_{v \in \vertices} c_vx_v &&&\\[1 ex]
    \, \sum_{S \in \potatoes} (\card{S} - 1)\potvar_S
    +
    \sum_{e = \{u,v\} \in \edges}((1 - x_u)\beta_{eu} + (1 - x_v)\beta_{ev})
                               &\leq n  {- k} - \sum_{v \in \vertices} x_v  , &&\label{eq:singlelevelNonlin:linking}\\[1 ex]
    \sum_{S \in \potatoessub(e)} \potvar_S + \beta_{eu} + \beta_{ev} &\geq 1, && e=\{u,v\} \in \edges.\label{eq:singlelevelNonlin:dual}
  \end{align}
\end{subequations}
This single-level reformulation forms the basis for our alternative model that we use to prove Theorem~\ref{thm:zig-basecont}.
In the following, we investigate properties of this model, which allow us to simplify and strengthen {Model}~\eqref{eq:singlelevelNonlin}.
\subsubsection{Simplifying the alternative formulation}
Observe that Model~\eqref{eq:singlelevelNonlin} is a mixed-integer non-linear programming model because it contains the product expressions~$(1-x_u)\beta_{eu}$ and~$(1-x_v)\beta_{ev}$ in Constraint~\eqref{eq:singlelevelNonlin:linking}.
As we show next, this non-linearity can be resolved by restricting the problem to a subspace that contains an optimal solution.
\begin{theorem}
  \label{thm:propertiesSolEF}
  Let~$\graph = (\vertices,\edges)$ be an undirected graph, let~$\bc \in \R^\vertices$, and let~$k$ be a positive integer such that the~$k$-vertex cut problem  is feasible.
  Then there exists an optimal solution~$(\bar{\bx},\bar{\potvarvec},\bar{\boldsymbol{\beta}})$ to~\eqref{eq:singlelevelNonlin} such that~$\bar{\potvar}_S \in \B{}$ for all~$S \in \potatoes$ and~$\bar{\beta}_{ew} = \bar{x}_w$ for all~$e \in \edges$, $w \in e$.
  In particular, if~$\mathscr{K}$ is the set of connected components of the graph arising from~$\graph$ by removing the nodes in~$\{v \in \vertices : \bar{x}_v = 1\}$, then
  \[
    \bar{\potvar}_S
    =
    \begin{cases}
      1, & \text{if } S \in \mathscr{K},\\
      0, & \text{otherwise}.
    \end{cases}
  \]
\end{theorem}
\begin{proof}{Proof}
  In Appendix~\ref{sec:proofs}, we provide an explicit construction of such a solution.
  \qed
\end{proof}
\begin{corollary}
  \label{cor:strengtheningZIG}
  Let~$\graph = (\vertices,\edges)$ be an undirected graph and let~$k$ be a positive integer.
  Let~$\cliques$ be the family of all cliques in~$\graph$.
  Then, the constraints
  \begin{align*}
    \sum_{S \in \potatoes} \potvar_S = k,~~~~
    \sum_{S \in \potatoes(v)} \potvar_S + x_v = 1,~~  v \in \vertices,~~~~
    \sum_{S \in \potatoes(C)} \potvar_S \leq 1,~~ C \in \cliques,
  \end{align*}
  can be added to Model~\eqref{eq:singlelevelNonlin} without changing the optimal objective value.
\end{corollary}
\begin{proof}{Proof}
    We defer the proof to Appendix~\ref{sec:proofs}.
    \qed
\end{proof}
Based on this theorem and corollary, we can simplify the model by dropping the non-linear terms from~\eqref{eq:singlelevelNonlin:linking}, replacing the~$\boldsymbol{\beta}$-variables by~$\bx$-variables, and adding additional constraints.
This results into the model
\begin{subequations}
  \label{eq:singlelevel}
  \begin{align}
    \min_{{\substack{\bx \in \{0,1\}^{\vertices}\\ \potvarvec \ge \boldsymbol{0}}}} \sum_{v \in \vertices} c_vx_v &&&\\[1 ex]
    \sum_{S \in \potatoes} (\card{S} - 1)\potvar_S  &\leq n - \sum_{v \in \vertices} x_v - k , &&\label{eq:singlelevel:linking}\\[1 ex]
    \sum_{S \in \potatoessub(e)} \potvar_S + x_{u} + x_{v} &\geq 1, && e=\{u,v\} \in \edges,\label{eq:singlelevel:dual}\\[1 ex]
    \sum_{S \in \potatoes} \potvar_S &= k, &&\label{eq:singlelevel:card}\\
    \sum_{S \in \potatoes(v)} \potvar_S + x_v &= 1, && v \in \vertices,\label{eq:singlelevel:iff}\\
    \sum_{S \in \potatoes(C)} \potvar_S &\leq 1, && C \in \cliques.\label{eq:singlelevel:cliques}
  \end{align}
\end{subequations}
Comparing~\eqref{eq:singlelevel} and \zigbase, observe that both models share the same variables (although the~$\potvarvec$-variables in Model~\eqref{eq:singlelevel} are continuous), objective, and constraints~\eqref{eq:singlelevel:card}, \eqref{eq:singlelevel:iff}, and~\eqref{eq:singlelevel:cliques}.
To prove Theorem~\ref{thm:zig-basecont}, we therefore need to show that the remaining constraints are redundant.
\subsubsection{Proof of Theorem~\ref{thm:zig-basecont}}
\label{sec:proof}
To show redundancy of~\eqref{eq:singlelevel:linking}, we sum~\eqref{eq:singlelevel:iff} for all~$v \in \vertices$ to derive
\[
  \card{\vertices} - \sum_{v \in \vertices} x_v
  =
  \sum_{v \in \vertices} \sum_{S \in \potatoes(v)}\potvar_S
  =
  \sum_{S \in \potatoes} \sum_{v \in S}\potvar_S
  =
  \sum_{S \in \potatoes} \card{S}\potvar_S.
\]
We thus find
\[
  \sum_{S \in \potatoes} (\card{S} - 1)\potvar_S
  =
  \sum_{S \in \potatoes} \card{S}\potvar_S
  -
  \sum_{S \in \potatoes} \potvar_S
  =
  \card{\vertices} - \sum_{v \in \vertices}x_v - \sum_{S \in \potatoes} \potvar_S
  \overset{\eqref{eq:singlelevel:card}}{\leq}
  \card{\vertices} - \sum_{v \in \vertices}x_v - k,
\]
i.e., \eqref{eq:singlelevel:linking} is redundant.

The redundancy of~\eqref{eq:singlelevel:dual} for edge~$e=\{u,v\}$ can be shown by substituting~$x_u$ and~$x_v$ via~\eqref{eq:singlelevel:iff}:
\begin{align*}
  &&1
  &\leq
    \sum_{S \in \potatoessub(e)} \potvar_S + x_{u} + x_{v}
    =
    \sum_{S \in \potatoessub(e)} \potvar_S + 2 - \sum_{S \in \potatoes(u)} \potvar_S - \sum_{S \in \potatoes(v)}\potvar_S\\[1 ex]
  &\iff&
         1 + \sum_{S \in \potatoessub(e)} \potvar_S
  &\geq
    \sum_{S \in \potatoes(u)} \potvar_S + \sum_{S \in \potatoes(v)}\potvar_S
    =
    \sum_{S \in \potatoes(e)} \potvar_S + \sum_{S \in \potatoessub(e)}\potvar_S\\[1 ex]
  &\iff&
         1
  &\geq
    \sum_{S \in \potatoes(e)} \potvar_S.
\end{align*}
Since every edge is a clique, Constraint~\eqref{eq:singlelevel:dual} is implied by~\eqref{eq:singlelevel:cliques}.
This concludes the proof.\qed
\subsection{Strength of the LP relaxations}
\label{sec:strength}
In the previous section, we have derived \zigbase for  the~$k$-VCP, thus complementing the existing models \ILPnat and \oldBP from the literature.
At first glance, however, it is unclear which model is best suited for solving the~$k$-VCP.
One common means to compare the strength of an ILP model is to assess the objective value of its LP relaxation:
the closer the LP objective value is to the objective value of an optimal mixed-integer solution, the arguably more powerful the bounding mechanism of branch-and-bound (or branch-and-price) is.
We therefore compare the objective values of the LP relaxations of the different methods in the following.

Let~$\graph = (\vertices,\edges)$ be an undirected graph that is endowed with vertex costs~$\bc \in \R^\vertices$, and let~$k$ be a positive integer.
We then denote by~$\lpvalnat(\graph,\bc,k)$, $\lpvalbp(\graph,\bc,k)$, and~$\lpvalzig(\graph,\bc,k)$ the optimal objective value of the LP relaxations of \ILPnat, \oldBP, and~\zigbase, respectively. These LP models are called: \ILPnatLP, \oldBPLP, and~\zigbaseLP.
{Since~\oldBP is in maximization form, $\lpvalbp(\graph,\bc,k)$ is defined as the total vertex cost minus the optimal value of~\oldBPLP, so that it is directly comparable with the optimal LP relaxation values of the other two models, namely~\ILPnatLP and~\zigbaseLP, which are in minimization form.}
\begin{theorem}
  \label{thm:strength}
  Let~$\graph = (\vertices,\edges)$ be an undirected graph, let~$\bc \in \R^\vertices$, and let~$k$ be a positive integer such that the~$k$-vertex cut problem is feasible.
  Then,
  \[
    \lpvalnat(\graph,\bc,k) \leq \lpvalzig(\graph,\bc,k)
    \qquad\text{and}\qquad
    \lpvalbp(\graph,\bc,k) \leq \lpvalzig(\graph,\bc,k).
  \]
\end{theorem}
{That is to say}, \zigbaseLP is not weaker than both,~\ILPnatLP and~\oldBPLP.
In the remainder of this section, we provide a proof of this theorem. {It is worth noticing that, throughout this section, we focus on feasible instances of the $k$-VCP, as the comparison of LP relaxation values is meaningful only in this case. The infeasible case is discussed separately in Section \ref{sec:initRMP}.}
\begin{lemma}
  \label{lem:auxstrength}
  Let~$\graph = (\vertices,\edges)$ be an undirected graph and let~$\subgraph = (\SUBvertices, \SUBedges) \in \SUBgraphs$.
  Let~$(\bx,\potvarvec)$ be a feasible solution to~\zigbaseLP.
  Then,
  \[
    \sum_{e \in \SUBedges} \sum_{S \in \potatoessub(e)} \potvar_S \leq \card{\vertices} - \sum_{v \in \vertices} x_v - k.
  \]
\end{lemma}

\begin{proof}{Proof}
  We observe
  \begin{align*}
    \sum_{e \in \SUBedges} \sum_{S \in \potatoessub(e)} \potvar_S
    &=
      \sum_{S \in \potatoes} \card{\SUBedges[S]} \; \potvar_S
    \leq
      \sum_{S \in \potatoes} (\card{S} - 1) \; \potvar_S
    =
      \sum_{S \in \potatoes} \card{S} \potvar_S - \sum_{S \in \potatoes} \potvar_S\\[1 ex]
    &=
      \sum_{S \in \potatoes} \sum_{v \in S} \potvar_S - \sum_{S \in \potatoes} \potvar_S
    =
      \sum_{v \in \vertices} \sum_{S \in \potatoes(v)} \potvar_S - \sum_{S \in \potatoes} \potvar_S
    \overset{\eqref{eq:zig-base:vertex_cover}}{=}
      \sum_{v \in \vertices} (1 - x_v) - \sum_{S \in \potatoes} \potvar_S\\[1 ex]
    &=
      \card{\vertices} - \sum_{v \in \vertices} x_v - \sum_{S \in \potatoes} \potvar_S
    \overset{\eqref{eq:zig-base:cardinality}}{=}
      \card{\vertices} - \sum_{v \in \vertices} x_v - k,
  \end{align*}
  where the first inequality holds, because a forest contains at most~$\card{S} - 1$ edges from a set~$S$ and~$\potvarvec$ is non-negative.\qed
\end{proof}
\begin{proposition}
  Let~$\graph = (\vertices,\edges)$ be an undirected graph, let~$\bc \in \R^\vertices$, and let~$k$ be a positive integer such that the~$k$-vertex cut problem is feasible.
  Then, $\lpvalnat(\graph,\bc,k) \leq \lpvalzig(\graph,\bc,k)$.
\end{proposition}
\begin{proof}{Proof}
  Let~$(\bx,\potvarvec)$ be feasible for~\zigbaseLP.
  It suffices to prove that~$\bx$ satisfies~\eqref{INT_R_0}.
  To this end, let $\subgraph = (\SUBvertices,\SUBedges) \in \SUBgraphs$.
  Recall that, for all~$e \in \edges$, Inequality~\eqref{eq:singlelevel:dual} is redundant for~\zigbaseLP.
  Hence, these inequalities are valid and summing them for all~$\{u,v\} \in \SUBedges$, we obtain
  \begin{align*}
    \card{\SUBedges}
    &\leq
    \sum_{\{u,v\} \in \SUBedges} \bigg( \sum_{S \in \potatoessub(\{u,v\})} \potvar_S + x_u + x_v\bigg)\\[1 ex]
    &=
    \sum_{v \in \vertices} \deg_\subgraph(v) x_v + \sum_{e \in \SUBedges} \sum_{S \in \potatoessub(e)} \potvar_S
    \overset{\text{Lemma~\ref{lem:auxstrength}}}{\leq}
    \card{\vertices} - k + \sum_{v \in \vertices} (\deg_\subgraph(v) - 1) x_v.
  \end{align*}
  Hence, $\bx$ is feasible for~\ILPnatLP, which proves the assertion.
  \qed
\end{proof}

\begin{proposition}
  Let~$\graph = (\vertices,\edges)$ be an undirected graph, let~$\bc \in \R^\vertices$, and let~$k$ be a positive integer such that the~$k$-vertex cut problem is feasible.
  Then, $\lpvalbp(\graph,\bc,k) \leq \lpvalzig(\graph,\bc,k)$.
  Moreover, if~$\cliques' = \cliques$ for defining~\eqref{eq:EF_2} in \oldBP, then~$\lpvalbp(\graph,\bc,k) = \lpvalzig(\graph,\bc,k)$.
\end{proposition}
\begin{proof}{Proof}
  To prove the assertion, we eliminate the~$\bx$-variables from~\zigbaseLP.
  By substituting~$x_v$ via~\eqref{eq:zig-base:vertex_cover}, we obtain
  \begin{align*}
     \min_{{\potvarvec \ge \boldsymbol{0}}} \left\{ \sum_{v \in \vertices} c_v \bigg(1 - \sum_{S \in \potatoes(v)} \potvar_S\bigg): ~~
     \sum_{S \in \potatoes} \potvar_S  = k, ~~
     \sum_{S \in \potatoes(C)} \potvar_S \leq 1,~  C \in \cliques,~~
     \sum_{S \in \potatoes(v)} \potvar_S  \leq 1,~  v \in \vertices
     \right\}. 
  \end{align*}
  Since~\eqref{eq:zig-base:clique_constrs} is not weaker than~\eqref{eq:EF_2}, $\potvarvec$ is feasible for~\oldBPLP.
  
  Moreover, note that the objective function simplifies to
  \[
    \sum_{v \in \vertices} \bigg(c_v - c_v \sum_{S \in \potatoes(v)} \potvar_s\bigg)
    =
    \sum_{v \in \vertices} c_v - \sum_{S \in \potatoes}\sum_{v \in S} c_v\potvar_s
  \]
  and is thus equivalent to the objective of~\oldBP.
  Consequently, $ \lpvalbp(\graph,\bc,k) \leq  \lpvalzig(\graph,\bc,k) $.

  If~$\cliques' = \cliques$, let~$\potvarvec$ be a feasible solution to \oldBPLP.
  Then, $(\bx,\potvarvec)$, where~$x_v = 1 - \sum_{S \in \potatoes(v)} \potvar_S$, $v \in \vertices$, is easily seen to be feasible to \zigbaseLP.
  Since arguments analogous to the above show that both solutions have the same objective value, we conclude~$\lpvalzig(\graph,\bc,k) \leq \lpvalbp(\graph,\bc,k)$, and thus~$\lpvalzig(\graph,\bc,k) = \lpvalbp(\graph,\bc,k)$.
  \qed
\end{proof}

\medskip
Combining these results thus yields Theorem~\ref{thm:strength}.

\section{A branch-and-price algorithm to solve the new extended formulation}
\label{sec:ALGO}
In this section, we describe the branch-and-price framework that we use to solve~\zigbase.
Since~\zigbase has exponentially many variables and possibly exponentially many constraints, we do not initialize it with all~$\potvarvec$-variables and constraints of type~\eqref{eq:zig-base:clique_constrs}.
Instead, we describe which subsets of~$\potvarvec$-variables and constraints are used to initialize Model~\zigbase to obtain the so-called restricted master problem (RMP) in Section~\ref{sec:initRMP}.
We continue the discussion by analyzing the pricing problem at the root node (Section~\ref{sec:priceRMP}), which is followed by describing our branching scheme and how it is respected by the pricing problem (Section~\ref{sec:branchRMP}).
Afterwards, we explain the separation problem for~\eqref{eq:zig-base:clique_constrs} (Section~\ref{sec:sepaRMP}).
We conclude the section by providing symmetry handling techniques (Section~\ref{sec:symmetry}) and further algorithmic enhancements (Section~\ref{sec:furtherEnhancements}) for accelerating the BP algorithm.
\subsection{The branch-and-price framework}
\label{sec:RMP}
In the following, we describe the building blocks of our branch-and-price framework to solve~\zigbase.
Before doing so, we observe that all equality constraints of~\zigbase can be relaxed {into inequalities} and still yield a correct model of the~$k$-VCP.
Although relaxing these constraints can, in principle, weaken the LP relaxation, we observed that the model with relaxed constraints performs better in practice, and we discuss possible explanations for this behavior later in Section~\ref{sec:ablation}.
\begin{proposition}
\label{prop:zigbases}
  Let~$\graph = (\vertices,\edges)$ be an undirected graph with vertex costs~$\bc \in \R^\vertices$ and let~$k$ be a positive integer.
  Moreover, let~$\cliques'$ be an edge-covering family of cliques of~$\graph$.
  Then, 
  
  \begin{subequations}
  \label{eq:zig-bases}
        \begin{align}
        \text{(\zigbases)} & & \min_{{\substack{\bx \in \{0,1\}^{\vertices}\\ \potvarvec \ge \boldsymbol{0}}}} \, \, \sum_{v \in \vertices} c_v \, x_v &&&\\[1 ex]
          & & \sum_{S \in \potatoes} \potvar_S &\geq k, &&\label{eq:zig-bases:cardinality}\\[1 ex]
          & &  \sum_{S \in \potatoes(v)} \potvar_S  + x_v&\geq 1, && v \in \vertices,\label{eq:zig-bases:vertex_cover}\\[1 ex]
          & & \sum_{S \in \potatoes(C)} \potvar_S &\leq 1, && C \in \cliques'.\label{eq:zig-bases:clique_constrs}
        \end{align}    
  \end{subequations}

\bigskip
  
  is a correct model of the~$k$-VCP.
\end{proposition}
The branch-and-price framework that we discuss in the following is based on Model \zigbases.
\begin{proof}{Proof}
  Based on Theorem~\ref{thm:zig-basecont}, we can relax the integrality requirement on the~$\potvarvec$-variables.
  To conclude the proof, we first show that it is possible to relax the equality in Constraint~\eqref{eq:zig-base:cardinality} and to define~\eqref{eq:zig-base:clique_constrs} only in terms of an edge-covering family of cliques rather than all cliques.
  Afterwards, we also argue that~\eqref{eq:zig-base:vertex_cover} can be relaxed.
  
  Based on Theorem~\ref{thm:propertiesSolEF} and~Corollary~\ref{cor:strengtheningZIG}, replacing~$\boldsymbol{\beta}$-variables in Model~\eqref{eq:singlelevelNonlin} by~$\bx$-variables and augmenting the resulting model by~$\sum_{S \in \potatoes} \potvar_S \geq k$ and~$\sum_{S \in \potatoes(v)} \potvar_S + x_v = 1$ for all~$v \in \vertices$ is a correct model of the $k$-VCP.
  Following the same arguments as in Section~\ref{sec:proof}, Constraint~\eqref{eq:singlelevelNonlin:linking} is redundant in the augmented model.
  Moreover, for every~$\{u,v\} \in \edges$, we substitute~$x_u$ and~$x_v$ in~\eqref{eq:singlelevelNonlin:dual} via~\eqref{eq:singlelevel:iff} to obtain
  \begin{align*}
    &&1
    &\leq
    \sum_{S \in \potatoessub(\{u,v\})} \potvar_S
    +
    \bigg(1 -  \sum_{S \in \potatoes(u)} \potvar_S\bigg)
    +
    \bigg(1 -  \sum_{S \in \potatoes(v)} \potvar_S\bigg)\\
    \iff&&
    1
    +
    \sum_{S \in \potatoessub(\{u,v\})} \potvar_S
    &\geq
      \sum_{S \in \potatoes(u)} \potvar_S
      +
      \sum_{S \in \potatoes(v)} \potvar_S
      =
      \sum_{S \in \potatoes(\{u,v\})} \potvar_S
      +
      \sum_{S \in \potatoessub(\{u,v\})} \potvar_S.
  \end{align*}
  That is, the \emph{edge constraint} $\sum_{S \in \potatoes(\{u,v\})} \potvar_S \leq 1$ is an equivalent representation of Constraint~\eqref{eq:singlelevelNonlin:dual} in the augmented model.
  Since the clique constraints~\eqref{eq:singlelevel:cliques} are valid for the augmented model by Corollary~\ref{cor:strengtheningZIG} and such a constraint for a clique~$C$ dominates the edge constraints for edges~$\{u,v\} \subseteq C$, it is sufficient to define~\eqref{eq:zig-base:clique_constrs} in terms of an edge-covering family of cliques.

  To conclude the proof, we show that we can relax equality in~\eqref{eq:zig-base:vertex_cover}.
  Let~$(\bx,\potvarvec)$ be an optimal solution of~\eqref{eq:zig-bases}.
  The assertion follows if we can establish that~$(\bx,\potvarvec)$ satisfies~$\sum_{S \in \potatoes(v)} \potvar_S  + x_v = 1$.
  This is indeed the case, since every vertex~$v \in \vertices$ is contained in a clique~$C \in \cliques'$.
  Hence, \eqref{eq:zig-bases:clique_constrs} and non-negativity of the~$\potvarvec$-variables imply~$\sum_{S \in \potatoes(v)} \potvar_S \leq 1$.
  If there was a~$v \in \vertices$ such that
  $\sum_{S \in \potatoes(v)} \potvar_S + x_v > 1$,
  we could improve the objective value by decreasing~$x_v$ because~$c_v > 0$.
  This, however, contradicts optimality of~$(\bx,\potvarvec)$.
  \qed
\end{proof}
\subsubsection{Initial sets of variables and constraints}
\label{sec:initRMP}
\zigbases contains two families of variables, the~$\bx$-variables and the~$\potvarvec$-variables.
All the $\card{\vertices}$ variables of the first type are added to the RMP.
For the second type of variables, we follow~\citet{CFLMMM18} and only add~$\potvar_{\{v\}}$, $v \in \vertices$, to the RMP.
Using the same arguments as~\citet{CFLMMM18}, this choice guarantees that the RMP is feasible whenever the~$k$-VCP admits a feasible solution.
To make this article self-contained, we provide the proof of this statement.
\begin{lemma}
  \label{lem:stableset}
  Let~$\graph = (\vertices,\edges)$ be an undirected graph and let~$k$ be a positive integer.
  Then, $\graph$ admits a~$k$-vertex cut if and only if~$\graph$ admits a stable set
  of size~$k$.
  In particular, if~$S \subseteq \vertices$ is a stable set of size~$k$, then~$\vertices
  \setminus S$ is a~$k$-vertex cut.
\end{lemma}
\begin{proof}{Proof}
  On the one hand, if~$\graph$ admits a~$k$-vertex cut, let~$D \subseteq \vertices$ be a
  set of vertices whose removal disconnects~$\graph$ into at least~$\ell \geq k$
  connected components~$K_1,\dots,K_{\ell}$.
  Then, for all distinct~$i,j \in \{1,\dots,\ell\}$, no pair of vertices~$v_i \in K_i$ and~$v_j
  \in K_j$ is adjacent.
  Thus, by selecting a vertex from each of the connected components, we
  create a stable set of size~$\ell \geq k$.

  On the other hand, if~$\graph$ admits a stable set~$S$ of size~$k$, then~$\vertices
  \setminus S$ is a~$k$-vertex cut of~$\graph$.\qed
\end{proof}
Based on Theorem~\ref{thm:zig-basecont} and Lemma~\ref{lem:stableset} we thus obtain the following proposition.
\begin{proposition}
  Let~$\graph = (\vertices,\edges)$ be an undirected graph with vertex costs~$\bc \in \R^\vertices$ and let~$k$ be a positive integer.
  If~$\graph$ admits a~$k$-vertex cut, then the RMP for \zigbases for the choice~$\potatoes = \{ \{v\} : v \in \vertices\}$ is feasible.
\end{proposition}
That is, whenever~$\graph$ admits a~$k$-vertex cut, the restricted master problem is always feasible.
Moreover, if the restricted master problem is infeasible, we can conclude that~$\graph$ does not admit a~$k$-vertex cut.

The RMP contains the Constraint~\eqref{eq:zig-bases:cardinality} and the $\card{\vertices}$ Constraints~\eqref{eq:zig-bases:vertex_cover}.
{To construct the edge-covering family of cliques $\cliques'$ used in Constraints~\eqref{eq:zig-bases:clique_constrs}, we propose the following constructive algorithm that is an iterative procedure. At each step, it builds one inclusion-wise maximal clique and maintains a set of uncovered edges, which is initially equal to the original edge set.
To build a clique $C$, the algorithm selects an uncovered edge $\{u,v\}$ and initializes $C$ as the set containing its two endpoints. It then repeatedly attempts to enlarge $C$ by adding a vertex $w \in \vertices \setminus C$ such that $w$ is adjacent to every vertex already in $C$. 
When no further vertex can be added, the clique $C$ is inclusion-wise maximal. The clique is then added to the family~$\cliques'$, and all edges whose endpoints both belong to $C$ are removed from the set of uncovered edges.
The algorithm terminates when no uncovered edges remain. Both edges and vertices are processed according to the natural ordering induced by the input instance. By construction, this constructive algorithm produces at most $\card{\edges}$ cliques.}

\subsubsection{The pricing problem}
\label{sec:priceRMP}
{We solve an} LP relaxation of \zigbases via the RMP, {by relaxing the integrality
constraints on the~$\bx$-variables} to non-negativity constraints, and {by considering} only a subset~$\overline{\potatoes} \subseteq \potatoes$ {for $\potvarvec$-variables}. 

Note that we do not impose the upper bound~$x_v \leq 1$, $v \in \vertices$, in the RMP, because in every \emph{optimal} solution~$x_v \leq 1$ will hold, which is a consequence of {the last paragraph of} the proof of Proposition~\ref{prop:zigbases}. {Dropping the upper bounds on the $\bx$-variables does not affect the pricing problem.}

To find an optimal solution of the LP relaxation of \zigbases, it might not be enough to solve the RMP for the selected set~$\overline{\potatoes}$.
Given a solution~$(\bx, \potvarvec)$ of the RMP, we therefore need to decide whether there exists a variable~$\potvar_S$ for some~$S \in \potatoes \setminus \overline{\potatoes}$ with negative reduced cost.
This problem is called the pricing problem and coincides with the separation problem of a dual solution associated with~$(\bx, \potvarvec)$.
Let~$\sigma \geq 0$ be the dual variable associated with~\eqref{eq:zig-bases:cardinality}, let~$\mu_v \geq 0$, $v \in \vertices$, be the dual variables associated with~\eqref{eq:zig-bases:vertex_cover}, and let~$\pi_C \geq 0$, $C \in \cliques'$, be the dual variables associated with~\eqref{eq:zig-bases:clique_constrs}.
The dual constraints associated to the $\potvarvec$-variables are
\begin{equation*}
  \sigma + 
  \sum_{v \in S} \mu_v - \sum_{\substack{C \in \cliques' :\\ S \cap C \neq \emptyset}} \pi_C  \le 0,~~~~ S \in \overline{\potatoes}.
\end{equation*}

Given a dual solution~$(\sigma, \boldsymbol{\mu}, \boldsymbol{\pi})$, the pricing problem for~$\potvarvec$-variables is to decide if there exists~$S \in \potatoes \setminus \overline{\potatoes}$ such that
\[
  \sigma
  +
  \sum_{v \in S} \mu_v
  -
  \sum_{\substack{C \in \cliques'\colon\\ S \cap C \neq \emptyset}} \pi_C
  >
  0
  \qquad
  \iff
  \qquad
  \sum_{v \in S} \mu_v
  -
  \sum_{\substack{C \in \cliques' :\\ S \cap C \neq \emptyset}} \pi_C
  >
  -\sigma.
\]
Deciding whether such~$S$ exists is therefore equivalent to solving the optimization problem
\begin{equation}
  \label{eq:pricing}
  \max_{S \subseteq \vertices}
  \left\{
    \sum_{v \in S} \mu_v - \sum_{\substack{C \in \cliques'\colon\\ S \cap C \neq \emptyset}} \pi_C : \card{S} \geq 1
  \right\}.
\end{equation}
Indeed, there exists a violated dual constraint if and only if the maximum takes a value that is larger than~$-\sigma$.

In the following, we discuss how to solve the pricing problem.
To better understand its structure, it is useful to express it as an ILP.
To model the selected set~$S$, let~$\varphi_v$, $v \in \vertices$, be a binary variable
that takes value~1 if and only if~$v$ belongs to~$S$.
Moreover, for every clique~$C \in \cliques'$, let~$\psi_C$ be a binary variable that
takes value~1 whenever the generated subset~$S$ intersects clique~$C$ (though the
converse does not necessarily hold).
An ILP model of the pricing problem, for a given dual
solution~$(\sigma, \boldsymbol{\mu}, \boldsymbol{\pi})$, is then
\begin{align}
\label{PP_ILP}
   \max_{{\substack{\boldsymbol{\varphi} \in \{0,1\}^{\vertices}\\[0.25ex]
                   \boldsymbol{\psi} \in \{0,1\}^{\cliques'}}}}
  \left\{
     \sum_{v \in \vertices} \mu_v \, \varphi_v 
     \;-\;
     \sum_{C \in \cliques'} \pi_C \, \psi_C
     :
     \quad
     \sum_{v \in \vertices} \varphi_v \ge 1,
     \quad
     \varphi_v \le \psi_C,\;
       C \in \cliques',~ v \in C
  \right\}.
\end{align}
The objective function in~\eqref{PP_ILP} maximizes the violation of the dual constraint.
The first constraint (\textit{cardinality constraint}) prevents selecting the empty set,
while the second family of constraints (\textit{precedence constraints}) ensures that
$\psi_C$ takes value 1 whenever the chosen set~$S$ contains a vertex belonging to the
clique~$C$.

The pricing problem~\eqref{eq:pricing} can be solved in polynomial time{:} as shown by \citet{CFLMMM18}, if the cardinality constraint  is removed from~\eqref{PP_ILP}, 
the resulting ILP has a totally unimodular constraint matrix. 
Indeed, the precedence constraints define a bipartite
incidence structure between vertices and cliques, while all remaining conditions
are simple (integer) bounds. Consequently, the LP relaxation of the reduced model is
integral and can be solved efficiently.
However, removing the cardinality constraint may yield the trivial solution
$S=\emptyset$. To enforce $S\neq\emptyset$, one can solve the LP relaxation of the
reduced model once for each vertex~$v\in\vertices$, adding the fixing constraint
$\varphi_v = 1$. Total unimodularity is preserved, so each resulting LP again
admits an integral optimal solution, which corresponds to a valid non-empty set
$S_v$. Among all such sets, one achieving the largest objective value is
selected as the solution of the pricing problem.
Since a polynomial number of LPs is solved and each LP can be solved in
polynomial time, the pricing routine is polynomial as well. By the classical
equivalence between optimization and separation
\citep{GroetschelLovaszSchrijver1981}, this also implies that the LP relaxation of \zigbases can be solved in polynomial time.

Complementing the LP techniques for solving the pricing problem, we discuss a combinatorial algorithm that is faster in practice.
A variant of this algorithm was already used by \citet{CFLMMM18} for \oldBP.
To make this paper self-contained and since their explanation of this approach omits many details, we provide a detailed explanation. Moreover, we provide some important  algorithmic improvements that fully preserve the use of the combinatorial algorithm for the pricing phase. 
We exploit a result by \citet{Barahona1998} who describe how to optimize a linear objective subject to precedence constraints (like the one{s} in \eqref{PP_ILP}).
Their idea is to solve the maximization problem by computing a minimum cut in an auxiliary network~$\mathcal N = (\vertices(\mathcal N), \mathcal A(\mathcal N))$.
Note that this approach neglects the cardinality constraint and we will discuss how to enforce it at the end of the section.

The vertex set of the network is given by
$
  \vertices(\mathcal N) = \{s,t\} ~\cup~ \vertices ~\cup~ \cliques',
$
where~$s$ and~$t$ are the source and sink of the network.
The arc set~$\mathcal{A}(\mathcal N)$ consists of (i) the arcs~$(s,v)$ with capacity~$\mu_v$ for all~$v \in \vertices$, (ii) the arcs~$(C,t)$ with capacity~$\pi_C$ for all~$C \in \cliques'$, and (iii) the arcs~$(v,C)$ with infinite capacity for all~$C \in \cliques'$ and~$v \in C$. {An example of this construction is shown in Figure~\ref{fig:pricingNetwork}.}
An~$s$--$t$ cut in~$\mathcal{N}$ is a partition~$(S',T')$ of~$\vertices(\mathcal{N})$ such that~$s \in S'$ and~$t \in T'$.
The arcs crossing a cut~$(S',T')$ are the arcs~$\varUpsilon(S',T') = \{(u,v) \in \mathcal{A}(\mathcal{N}) : u \in S',\; v \in T'\}$; the capacity of the cut is the total capacity of its crossing arcs.
\citet{Barahona1998} noted that there is a one-to-one correspondence between $s$--$t$ cuts with finite capacity and feasible solutions of \eqref{PP_ILP} without the cardinality constraint.
Indeed, given a cut~$(S',T')$ of finite capacity, the corresponding set~$S=S(S',T')$ is~${\{ v \in \vertices : (s,v) \notin \varUpsilon(S',T')\}}$.
The capacity of this cut is
\[
  \sum_{v \in \vertices \setminus S} \mu_v + \sum_{\substack{C \in \cliques'\colon\\ C \cap S \neq \emptyset}} \pi_C,
\]
Thus, since
\[
  \sum_{v \in S} \mu_v - \sum_{\substack{C \in \cliques'\colon\\ C \cap S \neq \emptyset}} \pi_C
  =
  \sum_{v \in \vertices} \mu_v - \sum_{v \in \vertices \setminus S} \mu_v - \sum_{\substack{C \in \cliques'\colon\\ C \cap S \neq \emptyset}} \pi_C
  =
  \sum_{v \in \vertices} \mu_v
  -
  \bigg(
  \sum_{v \in \vertices \setminus S} \mu_v
  +
  \sum_{\substack{C \in \cliques'\colon\\ C \cap S \neq \emptyset}} \pi_C
  \bigg),
\]
an optimal solution of~\eqref{PP_ILP} without the cardinality constraint  can be found by computing a minimum cut in the auxiliary network~$\mathcal N$.

\begin{figure}[t]
  \begin{subfigure}[t]{0.2\textwidth}
    \centering
    \begin{tikzpicture}
      \tikzstyle{v} += [circle,draw=black,thick,inner sep=2pt,minimum size=2mm];

      \draw[fill=blue!30,opacity=0.5] (0,1) ellipse (0.65cm and 1.5cm);
      \draw[fill=red!30,opacity=0.5] (0,-1) ellipse (0.65cm and 1.5cm);

      \node (v1) at (0,2) [v,label=left:{\scriptsize $u$}] {};
      \node (v2) at (0,0) [v,label=left:{\scriptsize $v$}] {};
      \node (v3) at (0,-2) [v,label=left:{\scriptsize $w$}] {};

      \draw[-] (v1) -- (v2) -- (v3);

      \node (C1) at (0.25,1) {\textcolor{blue}{\scriptsize $C_1$}};
      \node (C2) at (0.25,-1) {\textcolor{red}{\scriptsize $C_2$}};
    \end{tikzpicture}
  \end{subfigure}
  \quad
  \begin{subfigure}[t]{0.35\textwidth}
    \centering
    \begin{tikzpicture}
      \tikzstyle{v} += [circle,draw=black,thick,inner sep=2pt,minimum size=2mm];


      \node (s) at (0,0) [v,fill=black,label=below:{\scriptsize $s$}] {};
      \node (v1) at (2,2) [v,label=below:{\scriptsize $u$}] {};
      \node (v2) at (2,0) [v,label=below:{\scriptsize $v$}] {};
      \node (v3) at (2,-2) [v,label=below:{\scriptsize $w$}] {};
      \node (c1) at (4,1) [v,label=below:{\scriptsize $C_1$}] {};
      \node (c2) at (4,-1) [v,label=below:{\scriptsize $C_2$}] {};
      \node (t) at (5,0) [v,label=below:{\scriptsize $t$}] {};

      \draw[->,dashed] (s) to node [above] {\scriptsize $1$} (v1);
      \draw[->,dashed] (s) to node [above] {\scriptsize $1$} (v2);
      \draw[->,dashed] (s) to node [above] {\scriptsize $1$} (v3);
      \draw[->] (v1) to node [above] {\scriptsize $\infty$} (c1);
      \draw[->] (v2) to node [above] {\scriptsize $\infty$} (c1);
      \draw[->] (v2) to node [above] {\scriptsize $\infty$} (c2);
      \draw[->] (v3) to node [above] {\scriptsize $\infty$} (c2);
      \draw[->] (c1) to node [above] {\scriptsize $3$} (t);
      \draw[->] (c2) to node [above] {\scriptsize $3$} (t);
    \end{tikzpicture}
    \subcaption{$S = \emptyset$.}\label{fig:pricingNetwork:cutA}
  \end{subfigure}
  \quad
  \begin{subfigure}[t]{0.4\textwidth}
    \centering
    \begin{tikzpicture}
      \tikzstyle{v} += [circle,draw=black,thick,inner sep=2pt,minimum size=2mm];


      \node (s) at (0,0) [v,fill=black,label=below:{\scriptsize $s$}] {};
      \node (v1) at (2,2) [v,fill=black,label=below:{\scriptsize $u$}] {};
      \node (v2) at (2,0) [v,label=below:{\scriptsize $v$}] {};
      \node (v3) at (2,-2) [v,label=below:{\scriptsize $w$}] {};
      \node (c1) at (4,1) [v,fill=black,label=below:{\scriptsize $C_1$}] {};
      \node (c2) at (4,-1) [v,label=below:{\scriptsize $C_2$}] {};
      \node (t) at (5,0) [v,label=below:{\scriptsize $t$}] {};

      \draw[->] (s) to node [above,xshift=-2mm] {\scriptsize $1+3$} (v1);
      \draw[->,dashed] (s) to node [above] {\scriptsize $1$} (v2);
      \draw[->,dashed] (s) to node [above] {\scriptsize $1$} (v3);
      \draw[->] (v1) to node [above] {\scriptsize $\infty$} (c1);
      \draw[->] (v2) to node [above] {\scriptsize $\infty$} (c1);
      \draw[->] (v2) to node [above] {\scriptsize $\infty$} (c2);
      \draw[->] (v3) to node [above] {\scriptsize $\infty$} (c2);
      \draw[->,dashed] (c1) to node [above] {\scriptsize $3$} (t);
      \draw[->] (c2) to node [above] {\scriptsize $3$} (t);
    \end{tikzpicture}
    \subcaption{$S = \{u\}$.}\label{fig:pricingNetwork:cutB}
  \end{subfigure}

  \caption{{Illustration of the auxiliary network for solving the pricing problem on a graph with three vertices $u, v, w$ and two maximal cliques, $C_1=\{u,v\}$ (blue) and $C_2=\{v,w\}$ (red). Arc capacities are given next to the arcs. 
  }}\label{fig:pricingNetwork}
\end{figure}

To solve the pricing problem, we propose a two stage procedure.
In the first stage, we find a minimum cut~$(S',T')$ in the network defined above and extract the set~$S = S(S',T')$.
If~$S \neq \emptyset$ 

The second stage is entered only when~$S = \emptyset$ and therefore
\[
  -\sigma
  <
  \sum_{v \in S} \mu_v
  -
  \sum_{\substack{C \in \cliques'\colon\\ C \cap S \neq \emptyset}} \pi_C
  =
  0.
\]
We then need to decide if there exists an alternative minimum cut~$(S'',T'')$ with~$S(S'',T'') \neq \emptyset$ whose objective value in~\eqref{PP_ILP} is greater than~$-\sigma$.
A situation in which this happens is illustrated in Figure~\ref{fig:pricingNetwork:cutA}.
Here, a cut corresponding to~$S = \emptyset$ is produced although there exist a variable with a negative reduced cost associated to the non-empty set $S = \{u\}$.
To determine such type of variables, it turns out that it is sufficient to solve~$\card{\vertices}$ auxiliary minimum cut problems as we explain next.

Assume there exists a set~$S \subseteq \vertices$ such that~$\potvar_S$ has negative reduced cost.
Let us assume that we had access to a set~$S$ of minimum reduced cost.
We then could construct an alternative dual solution~$(\bar{\sigma}, \bar{\boldsymbol{\mu}}, \bar{\boldsymbol{\pi}})$ such that the optimal value in~\eqref{PP_ILP} remains invariant by selecting~$\bar{v} \in S$ and defining
\begin{align*}
  \bar{\sigma} &= 0,
  &
  \bar{\mu}_v
  &=
  \begin{cases}
    \mu_{\bar{v}} + \sigma, &\text{if } v = \bar{v},\\[1 ex]
    \mu_v, & \text{otherwise},
  \end{cases}
  &
  \bar{\boldsymbol{\pi}}
  &=
  \boldsymbol{\pi}.
\end{align*}
Indeed, for every non-empty subset~$S' \subseteq \vertices$ that contains~$\bar{v}$, the objective value in~\eqref{PP_ILP} remains unchanged and, if~$\bar{v} \notin S'$, the objective value drops by~$\sigma$.
Therefore, the set~$S$ that we started with is still an optimizer of~\eqref{PP_ILP}.
Since the second stage is only entered when~$\sigma > 0$, this means that solving the minimum cut problem on the auxiliary network~$\mathcal{N}$ for the dual solution~$(\bar{\sigma}, \bar{\boldsymbol{\mu}}, \bar{\boldsymbol{\pi}})$ will result in finding a set~$S$ that contains~$\bar{v}$.

In practice, of course, we do not know a set~$S$ of minimum reduced cost.
But nevertheless we can adopt the ideas above.
For each~$v \in \vertices$, we solve a minimum cut problem in the network~$\mathcal{N}$, where we increase the capacity of arc~$(s,v)$ by~$\sigma$.
If one of these minimum cuts yields a non-empty set~$S$ such that its associated variable~$\potvar_S$ has negative reduced cost, we add it as an improving column to the RMP.
Otherwise, no variable with negative reduced cost exists.

{Figure~\ref{fig:pricingNetwork} illustrates the auxiliary network on an example graph with three vertices~$u$, $v$, $w$, whose edge-covering clique family consists of two cliques, $C_1 = \{u,v\}$ (highlighted in blue) and $C_2 = \{v,w\}$ (highlighted in red). In the auxiliary network, black-filled nodes denote the source side~$S'$ of the min-cut, white nodes denote the sink side~$T'$, and dashed arcs are the arcs of~$\varUpsilon(S',T')$ crossing the cut.
In this example, the dual solution is $\mu_u = \mu_v = \mu_w = 1$, $\pi_{C_1} = \pi_{C_2} = 3$ (defining the arc capacities shown next to the arcs), and $\sigma = 3$.
In Figure~\ref{fig:pricingNetwork:cutA}, the auxiliary network yields a min-cut associated to the arcs $(s,u)$, $(s,v)$, $(s,w)$, with capacity~$3$, producing the solution~$S = \emptyset$.
Augmenting the capacity of arc~$(s,u)$ by~$\sigma = 3$ (Figure~\ref{fig:pricingNetwork:cutB}) changes the min-cut: it is now associated to the arcs~$(s,v)$, $(s,w)$, and~$(C_1,t)$ with total capacity~$5 < 1 + 4 = 6$, yielding the solution~$S = \{u\}$.}

\subsubsection{The branching scheme}
\label{sec:branchRMP}
A central component of a branch-and-price algorithm is the branching scheme.
For \oldBP, \citet{CFLMMM18} suggest at two-step branching procedure.
In \zigbases, however, the~$\potvarvec$-variables are continuous and it is not necessary to branch on them.
Instead, we only branch on the~$\bx$-variables. We use the reliable pseudo cost branching rule~\citep{Achterberg2007}.
{When solving the RMP at a node of the branching tree that is different from the root node, the branching decisions on the~$\bx$-variables must be respected by the pricing problem.}

{We describe how the branching decisions are incorporated into the combinatorial pricing algorithm.}
If there exists a vertex~$v \in \vertices$ such that~$x_v = 1$ holds due to branching, then~$v$ is contained in the vertex cut and, therefore, cannot be included in the generated set $S$.
This decision can be incorporated into the pricing problem by adding the arc~$(v,t)$ of infinite capacity to the network used for solving the pricing problem.
Due to the infinite capacity, this arc will never be contained in a minimum cut, thus representing the branching decision~$x_v = 1$.

To model the branching decision~$x_v = 0$, we need to make sure that~$v$ is not contained in the neighborhood
\[
N(S) = \{w \in \vertices \setminus S : \exists u \in S \text{ with } \{u,w\} \in \edges\}
\]
of the set~$S$ that is generated by the pricing problem.
Indeed, since~$x_v = 0$, vertex~$v$ needs to be contained in some cluster.
But if~$v \in N(S)$ this would mean that~$v$ is not contained in~$S$ and the cluster~$S'$ containing~$v$ is connected with~$S$---a contradiction to a~$k$-vertex cut.
To incorporate the neighborhood condition into the pricing problem, we can add arcs~$(w, v)$ with infinite capacity for all~$w \in \vertices$ with~$\{v,w\} \in \edges$.
Thus, in a minimum cut, it is impossible to include a neighbor of~$v$ into~$S$ while excluding~$v$ from~$S$.

{For the sake of completeness, we now describe how the ILP-based pricing problem~\eqref{PP_ILP} can be modified according to the branching decisions.
In case~$x_v$ is fixed to $1$, vertex~$v$ belongs to the vertex cut and cannot be part of any generated cluster~$S$; this can be modeled by fixing~$\varphi_v = 0$.
Conversely, if~$x_v$ is fixed to $0$, a neighbor $w$ of vertex~$v$ can belong to $S$ only if $v$ is in $S$ as well; this can be modeled by including the following additional precedence constraints: $\varphi_w \le \varphi_v$, for all $w \in N(v)$. In both cases, the total unimodularity of the constraint matrix is preserved.}
\subsubsection{The separation problem for the clique-partitioning constraints~\eqref{eq:zig-base:clique_constrs}}
\label{sec:sepaRMP}

In principle, it is possible to strengthen formulation \zigbases by explicitly separating violated constraints of type \eqref{eq:zig-bases:clique_constrs}, rather than restricting ourselves to the inequalities associated with the initial edge-covering family of cliques $\cliques'$. Given a fractional solution $(\boldsymbol{x}, \potvarvec)$ of the LP relaxation, the corresponding separation problem asks whether there exists a clique $C \subseteq \vertices$ such that~$\sum_{S \in \potatoes(C)} \potvar_S > 1$,
and, if so, to identify a clique with maximum violation. 
In order to fully assess the strength of formulation~\zigbases, we also implemented a branch-and-price-and-cut algorithm in which the clique partitioning constraints~\eqref{eq:zig-bases:clique_constrs} are separated exactly at the root node. This setting allows us to exploit the full potential of these inequalities and to benchmark the best bound that can be obtained from~\zigbases. The computational results reported in Appendix~\ref{sec:sep} indicate that, for the instances considered, a comparable dual bound can already be achieved by adding only the constraints associated with the edge-covering family of cliques $\cliques'$ introduced in Section~\ref{sec:RMP}. In other words, the initial edge-covering family of cliques $\cliques'$ suffices to obtain a bound of very high quality, which is in most cases almost indistinguishable from that obtained with exact separation at the root. Since the corresponding separation problem is $\mathcal{NP}$-hard and adds a computational overhead, our final algorithm therefore does not include the exact separation of clique partitioning constraints and relies exclusively on the inequalities generated from~$\cliques'$. Further details, including the branch-and-price-and-cut performance and the proof of $\mathcal{NP}$-hardness of the separation problem, are provided in Appendix~\ref{sec:sep}.

\subsection{Symmetry handling techniques}
\label{sec:symmetry}
To further enhance solving \zigbases, we observe that the~$k$-vertex cut problem admits structurally equivalent solutions if the underlying graph is symmetric.
Consider, for example, the minimum 2-vertex cut problem on a cycle of length~$n \geq 4$, and assume its vertices are consecutively labeled by~$1,\dots, n$.
Then, a family of solutions is given by~$\{1,3\},\{2,4\},\dots,\{n-2,n\},\{1,n-1\}$, i.e., one removes two vertices which are at distance~2 from each other.
Although these solutions are structurally identical, they look differently to a solver because they are encoded by different variables.
In the MIP literature, it is well-known that, when not handled appropriately, the presence of such symmetries can considerably hinder solving integer programs by branch-and-bound \citep{PfetschRehn2019}.
For this reason, a wealth of techniques has been developed to handle symmetries within branch-and-bound in the last two decades.
The main goal of these techniques is to prune parts of the search tree when it is known that another part contains symmetric information, which arguably accelerates branch-and-bound.
The branch-and-price literature knows problem specific symmetry breaking and, of course, elimination of model symmetry by aggregation of identical subproblems. However, we may additionally have symmetry in the input instance. This can be exploited for pruning the branch-and-price tree using the generic symmetry handling machinery built into solvers.
In the following, we do exactly this when solving \zigbases. To the best of our knowledge, our application of symmetry handling (on what can be called the original variables) is the first of its kind in a branch-and-price context.

To formally introduce our approach, we require some notation.
A permutation~$\gamma$ of~$\{1,\dots,n\}$ acts on an~$n$-dimensional vector~$\by \in \R^n$ by permuting its indices, i.e., $\gamma(\by) \define (y_{\gamma^{-1}(1)},\dots,y_{\gamma^{-1}(n)})$.
Such a permutation~$\gamma$ defines a \emph{symmetry} of the integer program~$\min_{{\by \in \Z^n}}\{\sprod{\bc}{\by} : \bA\by \leq \bb\}$, where~$\bA$ is a matrix and~$\bb$ and~$\bc$ are vectors of suitable dimensions, if for all~$\by \in \Z^n$ we have that~$\bA\by \leq \bb$ holds if and only if~$\bA\gamma(\by) \leq \bb$ holds, and~$\sprod{\bc}{\gamma(\by)} = \sprod{\bc}{\by}$.
That is, $\gamma$ preserves both feasibility and the objective value of a solution.
In the following, we want to derive a class of symmetries for \zigbases.
As mentioned above, a natural class of symmetries is given by symmetries of the underlying graph.
To formally define graph symmetries, let~$\graph = (\vertices,\edges)$ be an undirected graph with vertex costs~$\bc \in \R^V$.
An \emph{automorphism} of~$\graph$ is a map~${\pi\colon \vertices \to \vertices}$ such that~$c_v = c_{\pi(v)}$ for all~$v \in \vertices$ and~$\{\pi(u),\pi(v)\} \in \edges$ for all~$\{u,v\} \in \edges$.
That is, $\pi$ can be considered a relabeling of the vertices of~$\graph$ that preserves vertex costs and adjacency.
\begin{proposition}
  \label{prop:symmetries}
  Let~$\graph = (\vertices,\edges)$ be an undirected graph with vertex costs~$\bc \in \R^V$, let~$\gamma$ be an automorphism of~$\graph$, and let~$k$ be a positive integer.
  Let~$\potvarvec = (\potvar_S)_{S \in \potatoes}$.
  Then, $\gamma$ defines a symmetry of~\zigbases, where~$\gamma(\bx) = (x_{\gamma^{-1}(1)},\dots,x_{\gamma^{-1}(n)})$ and~$\gamma(\potvarvec) = (\potvar_{\gamma^{-1}(S)})_{S \in \potatoes}$, where~$\gamma^{-1}(S) = \{\gamma^{-1}(v) : v \in S\}$.
\end{proposition}
For a proof, we refer the reader to Appendix~\ref{sec:proofs}.

Knowing a class of symmetries of~\zigbases, we can apply state-of-the-art methods for handling them.
A common feature of most existing symmetry handling techniques is to define a lexicographic order on the variable space and to enforce that only those solutions can be computed that are lexicographically maximal in their class of symmetric solutions \citep{DoornmalenHojny2024a}.
In a branch-and-price context, however, the challenge arises that not all variables are known explicitly.
We therefore suggest to define the lexicographic order only on the~$\bx$-variables and ignore the~$\potvarvec$-variables for symmetry considerations.

We have experimented with two approaches for handling symmetries.
For both approaches, we assume that we are given access to a set~$\Gamma$ of automorphisms of~$\graph$.
To explain the differences between these approaches, let us assume~${\vertices = \{1,\dots,n\}}$.

The first approach defines a global lexicographic order~$\lexgeq{}$, where~$\bx \lexgeq{} \by$ holds for vectors~$\bx,\by \in \R^n$ if~$\bx = \by$ or, for the first position~$i$ at which~$\bx$ and~$\by$ differ, we have~$x_i > y_i$.
Given an automorphism~$\gamma\in\Gamma$, we enforce~$\bx \lexgeq{} \gamma(\bx)$ by the lexicographic reduction algorithm~\citep{DoornmalenHojny2024a}.

The second approach defines for every node of the branch-and-bound tree a possibly different partial lexicographic order.
To define the order at a node~$b$ of the branch-and-bound tree, let~$x_{i_1},\dots,x_{i_j}$ be the ordered sequence of branching variables from the root node to node~$b$.
Two vectors~$\bx, \by \in \R^n$ then satisfy the partial lexicographic order~$\lexgeq{,b}$ at node~$b$ if~$(x_{i_1},\dots,x_{i_j}) \lexgeq{} (y_{i_1},\dots,y_{i_j})$.
\citet{Ostrowski2009} has shown that a valid symmetry handling approach is to enforce~$\lexgeq{,b}$ at every node~$b$ of the branch-and-bound tree.
We enforce this by the lexicographic reduction algorithm and the orbital fixing method~\citep{Margot2003,OstrowskiEtAl2011}, which are both shown to be compatible by~\citet{DoornmalenHojny2024a}.

Preliminary experiments have shown that the second approach yields better reductions of running time.
For this reason, we will only discuss the impact of the second approach in our computational study.

\subsection{Further algorithmic enhancements}
\label{sec:furtherEnhancements}

{Two additional components are used to enhance the performance of our branch-and-price algorithm.
The first is the \emph{connectivity cut}, a valid inequality that strengthens the LP relaxation of \zigbases by enforcing a lower bound on the total vertex cut cost derived from the weighted vertex connectivity of the graph; full details are provided in Appendix~\ref{sec:connectivity_cut}.
The second is the \emph{Iterative Disconnection Heuristic} (IDH), a constructive algorithm that produces a high-quality initial feasible solution for the $k$-VCP, enabling early pruning of the BP tree; the implementation details are given in Appendix~\ref{sec:heur}.}

\color{black}
\section{Computational experience}
\label{sec:COMP}

{Our computational study pursues three main objectives.
The first goal is to assess the impact of the main algorithmic enhancements introduced in the new BP algorithm based on \zigbases, described in Section~\ref{sec:ALGO}, denoted by BP$^\star$, and implemented in \texttt{SCIP},  a framework and solver for ILP models that natively supports all elements of a full branch-and-price algorithm.
The second is to compare the performance of BP$^\star$ with that of the BP algorithm based on \oldBP proposed by \citet{CFLMMM18} and denoted by BP$^{\dagger}$.
The third goal is to benchmark BP$^\star$ against state-of-the-art exact approaches from the literature, namely those based on \ILPcomp and \ILPnat, both tackled using \texttt{CPLEX}, a commercial solver for ILP models.
Throughout this section, \textsc{COMP} denotes the approach that uses \texttt{CPLEX} to solve \ILPcomp, strengthened by a symmetry-breaking preprocessing as proposed in \citet{CFLMMM18}.
We denote by \textsc{HYB} the BC algorithm based on an enhanced version of \ILPnat, implemented within \texttt{CPLEX}, which natively supports branch-and-cut.
This model is further strengthened by incorporating a coefficient downlifting procedure for Constraints~\eqref{INT_R_0} together with additional valid inequalities, as proposed by \citet{FurLMP2020}.}

All experiments are conducted on the benchmark instances for the $k$-VCP introduced by \citet{FurLMP2020}.
The benchmark contains 304 instances, each considered in two variants: an unweighted version, where all vertices have unit cost, and a weighted version obtained by assigning integer vertex costs in the interval $[1,10]$.
The instances cover values of $k \in \{5,10,15,20\}$.
A detailed description of the full benchmark set of instances is provided in {Appendix}~\ref{sec:compInstances}.

All experiments comparing BP$^\star$ against BP$^\dag$, HYB, and COMP were executed single-threaded with a time limit of one hour on the compute cluster of the Chair of Operations Research at RWTH Aachen University, running Debian~11. The implementations of HYB and COMP are publicly available at \url{https://github.com/paoloparonuzzi/k-Vertex-Cut-Problem}, while the code for BP$^\dag$ was obtained through personal communication with the authors.
The cluster nodes are equipped with Intel Xeon L5630 quad-core processors operating at \SI{2.13}{\giga\hertz} and \SI{16}{\giga\byte} of DDR3 RAM.

The remainder of this section is organized as follows.
Section~\ref{sec:compDetails} describes the computational setup  and the implementation details of BP$^\star$.
Section~\ref{sec:ablation} presents an ablation study aimed at quantifying the contribution of each algorithmic component to the overall performance of BP$^\star$ and {at comparing BP$^\star$ with BP$^\dagger$.}
In Section~\ref{sec:comparison_exact} we report a detailed performance comparison between BP$^\star$ and the other state-of-the-art exact approaches, namely \textsc{COMP} and \textsc{HYB}. Finally, in Section~\ref{sec:lp_comparison}, we compare the strength of the LP relaxations of the different ILP models.

\subsection{Computational setup and implementation details  of BP$^\star$}
\label{sec:compDetails}

 BP$^\star$ has been implemented in \texttt{C++} using the BP framework \texttt{SCIP}~9.2.3 \citep{SCIP9}, which we extended through custom plugins to support the tailored column generation scheme.
LP relaxations within \texttt{SCIP} are solved using the LP solver \texttt{SoPlex}~7.1.5. {This choice is motivated by its native integration within the \texttt{SCIP} optimization suite and its availability, which aligns with our goal of developing a fully open-source algorithm. Moreover, preliminary tests on the benchmark set of unweighted instances indicated that using \texttt{CPLEX} as an LP solver did not yield substantial performance gains: when using \texttt{SoPlex}, we were indeed able to solve one additional instance within a one-hour time limit, compared to \texttt{CPLEX}.}
Maximum flow problems arising in the pricing subproblem and in the initial IDH are solved using the \texttt{LEMON} graph library~\citep{lemon}.
The complete \texttt{C++} source code, together with all benchmark instances and detailed computational results, is publicly available at \url{https://doi.org/10.5281/zenodo.17897865}.
This implementation is intended to facilitate reproducibility and to serve as a reference platform for the development and fair comparison of future exact algorithms for the $k$-VCP and related problems.

\subsection{Evaluation of components of BP$^\star$ and comparison with BP$^\dagger$}
\label{sec:ablation}

{Extensive preliminary experiments showed that the construction of the edge-covering family of cliques used in Constraints~\eqref{eq:zig-bases:clique_constrs} has a significant impact on the performance of BP$^\star$. For this reason, in addition to the constructive algorithm described in Section~\ref{sec:initRMP}, we test two alternative strategies.
The first alternative uses a modified version of the constructive algorithm in which, during the construction of a clique, a vertex is added only if it is connected to all vertices currently in the clique by uncovered edges. This modification produces an \textit{edge-partitioning} family of cliques, meaning that each edge belongs to exactly one clique. We denote the corresponding variant of BP$^\star$ by BP$^\star_a$.
The second alternative uses the edge set itself as an edge-partitioning family of cliques. The corresponding variant of BP$^\star$ is denoted by BP$^\star_b$.}

{A second key factor in the performance of BP$^\star$ is the relaxation of Constraints~\eqref{eq:zig-base:cardinality} and \eqref{eq:zig-base:vertex_cover} from equalities to inequalities, as in \eqref{eq:zig-bases:cardinality} and \eqref{eq:zig-bases:vertex_cover}. We therefore test a third variant of BP$^\star$, denoted as BP$^\star_c$, which uses the equality constraints.}

Figure~\ref{fig:ablation_perf} shows performance profiles \citep[see][]{Dolan2002} comparing BP$^\star$, its three variants, and BP$^\dagger$ on the 304 unweighted instances, with a time limit of one hour. For a fair comparison, we re-ran BP$^\dagger$ on our hardware using a recent version of \texttt{CPLEX} (22.1.0) as LP solver, obtaining results consistent with those reported in \citet{FurLMP2020}. Indeed, \citet{FurLMP2020} report only three more instances solved within the time limit: this gap can be explained by the significantly higher computational power of the hardware used in their experiments. The solid red curve represents BP$^\star$, the three dash-dotted curves correspond to the BP$^\star$ variants introduced above, and the dotted blue curve represents BP$^\dagger$.

\begin{figure}[h]
    \centering
    \input{FIGURES/ablation/all_k}
    \caption{Performance profiles comparing BP$^\star$ with its three variant and BP$^\dagger$ on the benchmark set of unweighted instances, with a time limit of one hour.
    }
    \label{fig:ablation_perf}
\end{figure}

Figure~\ref{fig:ablation_perf} highlights the strong impact of the construction of the clique family $\cliques'$ on the performance of~BP$^\star$. In particular, both BP$^\star_a$ and BP$^\star_b$ are clearly dominated by BP$^\star$: moving from an edge-covering family of cliques to an edge-partitioning one, and then to single edges, leads to a progressive deterioration of performance, confirming that the construction of $\cliques'$ is a key component of the algorithm. This is also reflected in the number of instances solved within the time limit: BP$^\star$ solves 242 instances, compared with 189 for~BP$^\star_a$ and only 122 for BP$^\star_b$.
A similar effect is observed when comparing BP$^\star$ with BP$^\star_c$. The performance profile of BP$^\star_c$ grows substantially more slowly and plateaus at a lower fraction of solved instances, with 205 instances solved versus 242 for BP$^\star$. This confirms that using Constraints~\eqref{eq:zig-bases:cardinality} and~\eqref{eq:zig-bases:vertex_cover} yields a significant performance gain.
Finally, we note that the implementation of BP$^\dagger$ uses an edge-partitioning family of cliques for Constraints~\eqref{eq:EF_2}, as in BP$^\star_a$. BP$^\dagger$ solves 165 instances and is therefore substantially outperformed by BP$^\star$ and all its variants except BP$^\star_b$. Interestingly, the comparison between BP$^\dagger$ and BP$^\star_a$ shows that, even when using the same type of clique family, the new algorithm achieves a marked improvement.

{In the following, we highlight the main differences between BP$^\star$, based on \zigbases, and BP$^\dagger$, based on~\oldBP, and discuss the reasons we think are behind their performance gap. We first observe that the construction of the clique family $\cliques'$ has a significant impact on the strength of the LP relaxation of both ILP formulations (see also Section~\ref{sec:lp_comparison}). In particular, the edge-covering family of cliques used in BP$^\star$ induces stronger constraints than the edge-partitioning family adopted in BP$^\dagger$. Moreover, on the unweighted benchmark instances, constructing an edge-partitioning family of cliques produces 392 cliques on average, whereas the edge-covering construction used in BP$^\star$ produces 194 cliques on average. Hence, the latter not only yields stronger constraints, but also does so with a smaller initial family, which further contributes to the improved performance of BP$^\star$.}

{Another important difference concerns Constraints \eqref{eq:EF_3} and \eqref{eq:zig-bases:cardinality}. BP$^\dagger$ enforces the equality constraint~\eqref{eq:EF_3}, which requires selecting exactly $k$ clusters, whereas BP$^\star$ uses the inequality constraint~\eqref{eq:zig-bases:cardinality}, which only requires selecting at least $k$ clusters. This difference can significantly affect the column generation process. In particular, the inequality constraint does not force the algorithm to generate columns corresponding to clusters that merge multiple connected components, if these are associated with already existing columns.
For instance, suppose that columns associated with two clusters, each containing a single connected component, have already been generated: if an optimal solution of \oldBP requires a single cluster obtained by merging these two components, BP$^\dagger$ must continue the column generation process until the corresponding merged cluster is produced. In contrast, BP$^\star$ may terminate earlier, since the two connected components can be selected separately and still satisfy Constraint \eqref{eq:zig-bases:cardinality}.}

{The column generation process of the two algorithms is also affected by an additional element of difference between them. BP$^\dagger$ enforces Constraints~\eqref{eq:EF_1}, which require that each vertex belongs to at most one cluster in any feasible solution of \oldBP. As introduced in Section~\ref{sec:ext}, BP$^\dagger$ adopts a bilevel branching scheme: at the first level, a partially covered vertex $v \in \vertices$ (i.e., one for which $\sum_{S \in \potatoes(v)} \potvar_S$ is fractional) is selected, and two child nodes are created by forcing $v$ either to belong to the cut ($\sum_{S \in \potatoes(v)} \potvar_S = 0$) or to a cluster ($\sum_{S \in \potatoes(v)} \potvar_S = 1$). Such branching decisions may require the generation of new columns in order to restore feasibility of the RMP.
In contrast, BP$^\star$ uses Constraints~\eqref{eq:zig-bases:vertex_cover}, which impose that a vertex not assigned to the cut must be covered by a selected cluster, and vice versa. These constraints are formulated as inequalities, allowing a (fractional) LP solution to simultaneously assign a vertex to a cluster and to the cut. This is particularly advantageous during column generation. For instance, when branching forces a vertex into the cut, BP$^\star$ does not necessarily require the generation of new columns associated with clusters that explicitly exclude that vertex: indeed, the restricted master problem can remain feasible by reusing existing cluster columns. }

{Furthermore, as detailed in Section~\ref{sec:ext}, first-level branching on individual vertices in BP$^\dagger$ is not sufficient to enforce integrality. As a consequence, BP$^\dagger$ applies a second branching level based on a Ryan--Foster-type scheme on pairs of vertices: when two vertices $u,v\in\vertices$ are partially assigned to the same cluster, two child nodes are generated by enforcing that $u$ and $v$ must either belong to the same cluster or to different clusters.
This additional branching level may significantly enlarge the search tree, as fractional solutions can persist even when all vertices are integrally assigned. Moreover, it complicates the pricing subproblem, since newly generated columns must respect pairwise compatibility decisions. In contrast, BP$^\star$ does not require any additional branching on vertex pairs, leading to a smaller search tree and keeping the pricing problem unchanged.}

Finally, on the implementation side, the code by \citet{CFLMMM18} uses a custom branch-and-bound framework to implement BP$^\dagger$, where the branching tree is explored in a depth-first fashion, and first-level branching decisions are based on the lowest-index vertex~$v$ for which~${\sum_{S \in \potatoes(v)} \potvar_S \notin \Z}$ in an LP solution.
On the other hand, BP$^\star$ makes use of the branch-and-bound framework of \texttt{SCIP}.
Next to more sophisticated tree exploration and branching rules, \texttt{SCIP} also implements advanced methods such as strong branching \citep{ApplegateBCC1995} which help to reduce the size of the branching tree.
Moreover, the implementation of BP$^\star$ benefits from the primal heuristics that are available in \texttt{SCIP}, whereas the implementation of BP$^\dagger$ requires LP solutions to be integral to improve the primal bound.

\subsubsection{Evaluation of additional components of BP$^{\star}$}

We also evaluated the impact of three additional features: the symmetry breaking techniques detailed in Section \ref{sec:symmetry}, providing BP$^\star$ with the initial primal solution returned by IDH (see Section \ref{sec:furtherEnhancements}) and adding the connectivity cut \eqref{eq:connectivity_cut} to \zigbases. Symmetry handling on the $\bx$-variables has only a modest visible impact on the running time for most instances, since the corresponding branching trees are relatively small and the techniques detailed in Section~\ref{sec:symmetry} are particularly effective on large search trees.
Nonetheless, their inclusion allows BP$^\star$ to solve four additional instances within the time limit compared with the variant of the algorithm without symmetry handling, which proves useful in terms of the overall number of instances solved and motivates keeping these techniques in the final version of the algorithm. 
Providing BP$^\star$ with the initial solution returned by IDH improves the performance of the algorithm as well: with an initial primal solution, BP$^\star$ solves eight more instances within the time limit than the variant started without any initial solution.  Similarly, the connectivity cut is generally beneficial and is therefore included in our default configuration. Enabling it for moderate values of $k$ ($\le 15$), as discussed in Section \ref{sec:furtherEnhancements}, allows BP$^\star$ to solve three additional instances within the time limit compared with the variant without it.

\subsection{Quantitative numerical results} \label{sec:algoSpecs}
Table \ref{tab:zig-stats} summarizes the behavior of the BP$^\star$ configuration on the unweighted benchmark, grouped by value of $k$. 
For each block of instances, we report the total number of instances and that of the optimally solved ones (\#inst, \#opt), the average size of the initial clique family and the average size of the cliques in terms of number of vertices (\#clq, \#vert), the average solution time and the portion of time spent in pricing (time total, pricing) on the optimally solved instances, the average number of generated columns (total, at the root node, and in Farkas pricing iterations) for the closed instances, the average number of explored branch-and-bound nodes and tree depth (B\&B nodes, depth) in the optimally solved instances, and finally the average integrality gap on the subset of instances whose optimal value is known and for which our best configuration algorithm manages to compute a valid dual bound at the root node within the time limit.
More precisely, let $z^\star$ denote the optimal value of a given instance, whenever it is known either from our runs or from the literature \citep[see][]{FurLMP2020}, and let $z^{\mathrm{LP}}$ be the value of the root-node LP relaxation produced by BP$^\star$. For each such instance on which BP$^\star$ computes a valid dual bound $z^{\mathrm{LP}}$ within the time limit, we define the \emph{integrality gap} as $100 \cdot \frac{z^\star - z^{\mathrm{LP}}}{z^\star}$, which is therefore expressed as a percentage.

Overall, BP$^\star$ solves $242$ out of $304$ unweighted instances, i.e., roughly 80\% of the test set, with fairly stable performance across different values of $k$. 
The size of the initial edge-covering family of cliques remains moderate (about 195 cliques on average, with a mean of around 7 vertices per clique), with an increase in the number of clique---and an associated decrease in the number of vertices per clique---as the value of $k$ increases. 
Average total running times for solved instances are on the order of one to two minutes (104.3 seconds on average), with pricing accounting for about one third of the time, which confirms that column generation is not a bottleneck component from the computational perspective.
The number of generated columns is in the low thousands (about 2{,}760 on average), with a substantial fraction already produced at the root node and a relatively small number of Farkas columns.
The branch-and-bound tree remains shallow and compact on the solved instances: on average, fewer than 50 nodes are explored, with depth around 5, which is consistent with a strong LP relaxation and effective branching. 
Finally, the average integrality gaps are very small: on instances whose optimal value is known and for which BP$^\star$ computes a valid root-node dual bound within the time limit, the average difference between this LP bound and the optimum is about $11\%$. This indicates that the LP relaxation of our formulation is quite strong and that often our branch-and-price starts from a high-quality bound. 

\begin{table}[h!]
\small
\renewcommand\arraystretch{1.7}
\tabcolsep=3pt
\caption{Summary statistics for the final BP$^\star$ configuration on the benchmark set of unweighted instances, grouped by value of $k$, with a time limit of one hour.}
\label{tab:zig-stats}
\centering
\begin{tabular}{lrrccccccccccccccc}
\hline
        &        &       &  & \multicolumn{2}{c}{cliques} &  & \multicolumn{2}{c}{time} &  & \multicolumn{3}{c}{gen cols} &  & \multicolumn{2}{c}{branching tree} &  &          \\ \cline{5-6} \cline{8-9} \cline{11-13} \cline{15-16}
        & \#inst & \#opt &  & \#clq        & \#vert       &  & total      & pricing     &  & total    & root     & farkas &  & nodes       & depth      &  & 
        gap (\%) \\ \cline{1-3} \cline{5-6} \cline{8-9} \cline{11-13} \cline{15-16} \cline{18-18} 
        &        &       &  &              &              &  &            &             &  &          &          &        &  &             &            &  &          \\[-2ex]
$k=5$   & 107    & 85    &  & 176.3        & 8.1          &  & 98.6       & 18.3        &  & 2,395.3  & 1,124.6  & 57.4   &  & 56.1        & 4.4        &  & 14.0     \\
$k=10$  & 80     & 63    &  & 195.7        & 6.4          &  & 160.4      & 72.1        &  & 3,462.9  & 843.4    & 186.7  &  & 72.3        & 6.6        &  & 11.5     \\
$k=15$  & 65     & 50    &  & 201.5        & 5.5          &  & 48.9       & 12.7        &  & 2,124.9  & 1,200.1  & 5.4    &  & 10.2        & 3.2        &  & 7.1      \\
$k=20$  & 52     & 44    &  & 226.9        & 5.3          &  & 98.0       & 57.7        &  & 3,198.7  & 894.8    & 410.1  &  & 30.4        & 6.4        &  & 10.2     \\
        &        &       &  &              &              &  &            &             &  &          &          &        &  &             &            &  &          \\[-2ex] \cline{1-3} \cline{5-6} \cline{8-9} \cline{11-13} \cline{15-16} \cline{18-18} 
Tot/Avg & 304    & 242   &  & 195.4        & 6.6          &  & 104.3      & 38.3        &  & 2,763.4  & 1,025.2  & 144.4  &  & 46.1        & 5.1        &  & 11.2     \\ \hline
\end{tabular}
\end{table}

\subsection{Computational comparison between exact approaches}
\label{sec:comparison_exact}

In this section, we compare the three exact approaches COMP, HYB, and BP$^\star$ on the benchmark set, considering both unweighted and weighted instances and a time limit of 3600 seconds. {To ensure a fair comparison, we re-executed \textsc{COMP} and \textsc{HYB} on our hardware using the same time limit and a more recent version of \texttt{CPLEX} (22.1.0). Overall, the relative performance trends of \textsc{COMP} and \textsc{HYB} are in line with those reported in \citet{FurLMP2020}.
Nevertheless, since the computational platform used in \citet{FurLMP2020} was more powerful than ours, some of the hardest instances are not closed within the time limit in our experiments. As a result, we solve fewer instances than what reported in \citet{FurLMP2020}: on the unweighted set, we solve 7 and 3 fewer instances with \textsc{HYB} and \textsc{COMP}, respectively, while on the weighted set both methods solve 7 fewer instances.}

Table~\ref{tab:agg_k_results} reports, for each $k \in \{5,10,15,20\}$, the number of instances solved to proven optimality by each algorithm. On unweighted instances, BP$^\star$ solves 242 out of 304 cases, against 195 for HYB and 165 for COMP.
The advantage of BP$^\star$ over COMP increases with $k$: for $k=5$ the two methods are relatively close whereas, for $k=20$, BP$^\star$ solves more than twice as many instances as the other algorithm. Conversely, the performance gap between BP$^\star$ and HYB decreases as $k$ grows: although BP$^\star$ keeps solving almost the same fraction of instances for all values of $k$, the performance of HYB improves with $k=15$ and $k=20$, thus reducing the gap between the two algorithms from 18 instances ($k=5$) to just 6 instances ($k {\in \{15,20\}}$).
The results are similar on weighted instances, where BP$^\star$ solves 239 instances, while HYB and COMP solve 201 and 170 instances, respectively.
For $k=5$, COMP is slightly ahead of the other two algorithms, confirming its effectiveness on instances with a smaller value of $k$, but for $k \ge 10$ BP$^\star$ obtains the largest number of solved instances in all cases.
Overall, these results indicate that BP$^\star$ dominates HYB and COMP in terms of number of closed instances on both test sets, and across almost all values of $k$.

\begin{table}[h!]
\centering
\small
\renewcommand\arraystretch{1.1}
\tabcolsep=10pt
\caption{Performance comparison of the exact algorithms COMP, HYB and BP$^\star$ on the full set of weighted and unweighted instances with a time limit of one hour.}
\label{tab:agg_k_results}
\begin{tabular}{l r rr rrr rr rrr}
\toprule
 &  && 
\multicolumn{3}{c}{\textbf{Unweighted instances}} 
&& 
\multicolumn{3}{c}{\textbf{Weighted instances}} \\
 \cmidrule(lr){4-6} \cmidrule(lr){8-10}

 &  && BP$^\star$ & HYB & COMP 
&& BP$^\star$ & HYB & COMP \\

 & \#inst && \#opt & \#opt & \#opt 
&& \#opt & \#opt & \#opt \\
\cmidrule(lr){1-2} \cmidrule(lr){4-6} \cmidrule(lr){8-10}
\\[-1.5ex]

$k=5$    & 107 && 85 & 67 & \textbf{88} && 85 & 68 & \textbf{90} \\
$k=10$   & 80  && \textbf{63} & 46 & 32 && \textbf{61} & 49 & 34 \\
$k=15$   & 65  && \textbf{50} & 44 & 28 && \textbf{50} & 45 & 28 \\
$k=20$   & 52  && \textbf{44} & 38 & 17 && \textbf{43} & 39 & 18 \\
\\[-1.2ex]
\cmidrule(lr){1-2} \cmidrule(lr){4-6} \cmidrule(lr){8-10}

Tot      & 304 && \textbf{242} & 195 & 165 && \textbf{239} & 201 & 170 \\
\bottomrule
\end{tabular}
\end{table}

The trends observed in Table~\ref{tab:agg_k_results} are confirmed by the plots in Figure~\ref{fig:survival_plots}, showing the cumulative percentage of instances solved to optimality over time by each of the three algorithms.
For both unweighted and weighted instances, the curve of BP$^\star$ lies above those of HYB and COMP over almost the entire time horizon.
In the weighted case, HYB and BP$^\star$ exhibit very similar performance for running times up to about one second, but afterwards the curve of BP$^\star$ grows more rapidly and remains clearly ahead.
This shows that BP$^\star$ is not only competitive on the easiest instances, but also both faster and more effective on the more difficult ones, namely those requiring tens or hundreds of seconds to be solved to optimality.

\begin{figure}[h]
    \centering

    \begin{subfigure}[t]{0.48\textwidth}
        \centering
        \input{FIGURES/survivalPlots/unweighted}
        \vspace{-1em}
    \end{subfigure}
    \hfill
    \begin{subfigure}[t]{0.48\textwidth}
        \centering
        \input{FIGURES/survivalPlots/weighted}
        \vspace{-1em}
    \end{subfigure}
    \smallskip
    \caption{{Cumulative percentage of instances solved to optimality over time by each algorithm on the complete sets of 304 unweighted instances (left) and 304 weighted instances (right), with a time limit of one hour.}}
    \label{fig:survival_plots}
\end{figure}

{To further assess the contribution of BP$^\star$ relative to previously available exact approaches, we compare our results with the instance-wise outcomes reported in \citet{FurLMP2020}, obtained under the same time limit but on a substantially more powerful hardware. In that study, the computational campaign involved three methods, namely \textsc{COMP}, \textsc{HYB}, and BP$^\dagger$. An instance is considered previously solved if it was solved to proven optimality by at least one of these approaches in \citet{FurLMP2020}.
With respect to this baseline, BP$^\star$ closes 73 additional instances that were not solved by any exact method in the reference study: 33 unweighted and 40 weighted instances. For reference, among the three algorithms tested in \citet{FurLMP2020}, \textsc{HYB} closes 33 instances not solved by either of the other methods (17 unweighted and 16 weighted), \textsc{COMP} closes 8 (2 unweighted and 6 weighted), and BP$^\dagger$ closes 3 (1 unweighted and 2 weighted).
Finally, some instances that were solved in \citet{FurLMP2020} are not closed by BP$^\star$ within the time limit in our experiments.
Detailed instance-wise results, including optimal or best known objective values, are reported in Appendix~\ref{sec:suppTables}, where we highlight both the instances solved for the first time by BP$^\star$ and those on which BP$^\star$ fails while at least one method from \citet{FurLMP2020} succeeds.}

\subsection{Strength of the LP relaxations} \label{sec:lp_comparison}

In this section, we compare the strength and computational cost of the LP relaxations associated with the different formulations considered in this paper. The goal of this analysis is twofold: on the one hand, to quantify how the choice of the initial clique family in \zigbases (edge-covering family of cliques, edge-partitioning family of cliques, or single edges) affects the quality of the root-node bound; on the other hand, to position the resulting relaxations with respect to previously proposed models, namely the enhanced version of~\ILPnat introduced in Section \ref{sec:COMP} and detailed in \citet{FurLMP2020} (denoted as \ILPhyb), the natural formulation~\ILPnat, and the compact formulation \ILPcomp with the symmetry breaking enhancement mentioned in Section \ref{sec:COMP} (denoted as \ILPcompstar).

Table~\ref{tab:lp-strength} reports, for each value of $k$, the average integrality gap between the LP bound and the optimal integer solution value (columns “gap”) and the average solution time of the root-node LP (columns “time”) for all formulations under comparison, on the subset of instances whose optimal solution value is known and such that it was possible to compute a valid LP bound within one hour of running time. Variants \zigbases, \zigbasesa, and \zigbasesb correspond to the three choices of initial clique family $\cliques'$: \zigbases uses the edge-covering family of cliques described in Section~\ref{sec:RMP}, \zigbasesa uses the edge-partitioning family of cliques of \citet{CFLMMM18}, and \zigbasesb replaces cliques by single edges. None of these variants includes the connectivity cut introduced in Section~\ref{sec:furtherEnhancements}, so as to isolate the effect of the initial clique family on the strength of the relaxation.

Among all variants, \zigbases consistently has the strongest relaxation: its average gaps range from 15.6\% for~$k=5$ down to 6.4\% for $k=20$, with an overall average of 11.1\%.
The two variants based on weaker initial clique families yield progressively weaker bounds: \zigbasesa has gaps between 22.9\% and 8.3\%, while \zigbasesb is much weaker, with gaps between 52.3\% and 23.6\%, comparable to those of the natural formulation.
This shows that the initial clique family is a decisive modeling aspect: moving from an edge-covering family of cliques to an edge-partitioning family of cliques, and further to single edges, quickly weakens the LP relaxation down to the level of the natural model.

As a computational counterpart to the theoretical results in Section~\ref{sec:strength}, we observe that \zigbasesb and \ILPnat provide the same LP bound (identical gap entries for each $k$); nonetheless, \zigbasesb requires substantially more time to compute it, especially for $k=5$ and $k=10$.
In terms of relaxation strength alone, \zigbasesb therefore does not improve over the natural formulation, while incurring a higher computational overhead.
Such overhead can be explained by the high number of columns generated at the root node (roughly 50\% more than \zigbases), but also by the increased complexity of the pricing problem, which takes more time to be solved due to the greater number of cliques considered.

The hybrid model \ILPhyb yields intermediate gaps, systematically better than \ILPnat but clearly looser than \zigbases, with average gaps around 34\% versus 11.1\% for \zigbases.
Finally, the compact formulation \ILPcompstar has by far the weakest relaxation, with gaps close to 90\% across all values of $k$, although with very small LP times.

Overall, the table illustrates that the edge-covering family of cliques used in \zigbases is key to obtaining a much stronger LP bound than all other formulations considered.
Although computing the LP relaxation of \zigbases requires more root-node time than \ILPhyb, \ILPnat, or \ILPcompstar, its LPs remain reasonably fast to solve and the improvement in bound quality is substantial.
This tighter relaxation of~\zigbases is a promising indication of the effectiveness of BP$^\star$ and is likely to translate, in the full branch-and-price algorithm, into fewer explored nodes and a larger number of instances closed within the time limit compared with the other existing approaches.

\begin{table}[h!]
\centering
\small
\renewcommand\arraystretch{2}
\tabcolsep=4.5pt
\caption{Average integrality gaps (in \%) and root-node solution times (in seconds) for all formulations on the unweighted benchmark set, grouped by $k$}
\label{tab:lp-strength}
\begin{tabular}{lrrrrrrrrrrrrrrrrrrrr}
\toprule
        &        &  & \multicolumn{2}{c}{\textbf{\zigbases}} &  & \multicolumn{2}{c}{\textbf{\zigbasesa}} &  & \multicolumn{2}{c}{\textbf{\zigbasesb}} &  & \multicolumn{2}{c}{\textbf{\ILPhyb}} &  & \multicolumn{2}{c}{\textbf{\ILPnat}} &  & \multicolumn{2}{c}{\textbf{\ILPcompstar}} \\ 
\cline{4-5} \cline{7-8} \cline{10-11} \cline{13-14} \cline{16-17} \cline{19-20}
$k$     & \#inst &  & gap & time &  & gap & time &  & gap & time &  & gap & time &  & gap & time &  & gap & time \\ 
\cline{1-2} \cline{4-5} \cline{7-8} \cline{10-11} \cline{13-14} \cline{16-17} \cline{19-20}


5  & 88 && \textbf{15.6} & 59.3  && 22.9 & 119.8 && 52.3 & 193.9 && 39.6 & 17.0 && 52.3 & 1.9  && 90.6 & 0.2 \\
10 & 69 && \textbf{10.9} & 61.8  && 16.5 & 196.1 && 40.4 & 149.4 && 37.2 & 29.6 && 40.4 & 2.8  && 92.2 & 1.9 \\
15 & 52 && \textbf{7.5}  & 31.6  && 10.8 & 90.0  && 26.9 & 72.6  && 26.5 & 13.6 && 26.9 & 3.8  && 88.5 & 7.6 \\
20 & 42 && \textbf{6.4}  & 48.2  && 8.3  & 120.0 && 23.6 & 98.7  && 23.2 & 13.0 && 23.6 & 5.3  && 90.3 & 10.6 \\

\midrule
Tot/Avg & 251 && \textbf{11.1} & 52.4 && 16.2 & 134.6 && 38.9 & 140.6 && 33.5 & 19.1 && 38.9 & 3.1 && 90.5 & 4.0 \\
\bottomrule
\end{tabular}
\end{table}

\section{Conclusions}
\label{sec:CONCL}

Despite its practical relevance for numerous areas such as network design or vulnerability assessment, existing methods for exactly solving the~$k$-VCP are not powerful enough to solve medium sized instances.
This article therefore took a fresh look into the problem and identified unused potential in an existing branch-and-price framework:
while all variables in the previous extended formulation were required to take binary values, our in-depth study revealed that the integrality requirement on those variables can be dropped when introducing the binary variables of the natural formulation.
This new perspective on the extended model allowed us to significantly simplify the underlying branching scheme and to enhance the new branch-and-price algorithm by novel components.
While results by \citet{FurLMP2020} suggest that branch-and-price was not the right technique for tackling the~$k$-VCP, we demonstrate in an extensive numerical study that, when properly set up, the new branch-and-price substantially improves on the state of the art.
Our algorithmic framework thus not only provides a new benchmark for solving the~$k$-VCP to proven optimality, but also changes the perspective on how to solve the~$k$-VCP: branch-and-price strikes back by solving for the first time 73 open instances.

Moreover, although our algorithm can be considered state of the art, there are still instances in our test set that remain unsolved. This suggests several directions for further improvement. One promising line of research is the development of more powerful presolving routines aimed at simplifying particularly difficult instances. Another potential enhancement is the incorporation of connectivity constraints within the subsets associated with the variables of the extended formulation. Considering such constraints could further strengthen the LP relaxation, but this would also have to be taken into account in the pricing problem, which may significantly increase its complexity. A careful study of this trade-off between bound quality and pricing difficulty can therefore be an interesting topic for future work.
 
The theoretical insights developed in this work may indicate a more general pathway for deriving extended formulations from natural bilevel models (see, e.g., \citet{FurLMP2020, FuriniLMP22}). Such models arise frequently in graph partitioning problems, where a leader removes vertices or edges and a follower solves a subproblem that captures a structural property of the modified graph. As discussed in Section~\ref{sec:alternativeForm}, if the follower's problem admits a polyhedral description with a totally unimodular constraint matrix, then dualization and subsequent linearization yield an extended formulation of the original problem with an exponential number of variables which can then be solved by means of a branch-and-price algorithm. For the $k$-VCP, this procedure led to formulation~\zigbase, in other settings the same construction may give rise to new formulations and, in turn, potentially efficient branch-and-price algorithms. Among others, this concerns the $k$-way vertex cut problem~\citep{BGZ14}, whose optimization version aims at maximizing the number of connected components of the residual graph under a budget on the cardinality of the vertex cut. We believe that our branch-and-price algorithm can be easily adapted and used to tackle this problem as well. More generally, the same approach might apply to other critical vertex detection problems that seek to fragment the graph by deleting vertices.

\newpage
\setcounter{table}{0}
\begin{APPENDICES}
\numberwithin{table}{section}

\section{Deferred proofs}
\label{sec:proofs}

\begin{proof}{Proof of Lemma~\ref{lem:treeLP}}
  Let~$\bx \in \B{\vertices}$ and let~$\tilde{\graph} \define \graph(\bx) = (\tilde{\vertices},\tilde{\edges})$ as defined above.
  Let~$\mathscr{M} \subseteq 2^{\tilde{\edges}}$ be the collection of subsets of~$\tilde{\edges}$
  that define an acyclic subgraph of~$\tilde{\graph}$.
  The set~$\mathscr{M}$ is the so-called cycle matroid of~$\tilde{\graph}$,
  see~\cite{KorteVygen2018}, with associated rank function~$r\colon 2^{\tilde{\edges}} \to
  \Z_+$,
  \[
    S \mapsto r(S)
    \define
    \card{S} - \text{number of connected components of~$\tilde{\graph}[S]$}.
  \]
  Due to~\cite{Edmonds1970}, the polytope
  \[
    P = \{ y \in \R_+^{\tilde{\edges}} : y(S) \leq r(S) \text{ for all } S \subseteq \tilde{\edges}\}
  \]
  is integral and its vertices are the incidence vectors of (edge sets of)
  acyclic subgraphs of~$\tilde{\edges}$.
  Moreover, it is easy to see that if~$S \subseteq \tilde{\edges}$ has multiple
  connected components, then the inequality~$y(S) \leq r(S)$ can be
  dominated by the sum of the inequalities for the respective connected
  components.
  Therefore,
  \[
    P
    =
    \{ y \in \R_+^{\tilde{\edges}} : y(S) \leq \card{S} - 1 \text{ for all } S \subseteq \tilde{\edges}\}
    =
    \{ y \in \R_+^{\tilde{\edges}} : y(S) \leq \card{S} - 1 \text{ for all } S \subseteq \tilde{\vertices},\; \card{S} \geq 2\}.
  \]
  In particular, by selecting~$S = \{u,v\}$, we find~$y_{e} \leq 1$ for all~$e \in \tilde{\edges}$.

  The assertion follows by lifting~$P$ from~$\R^{\tilde{\edges}}$ to~$\R^{\edges}$.
  This is achieved by introducing variables~$y_{e}$ for all~$e = \{u,v\} \in \edges\setminus\tilde{\edges}$ and fixing~$y_{e} = 0$ (via the inequalities~$y_{e} \geq 0$, $y_{e} \leq 1 - x_u$, and~$y_e \leq 1 - x_v$).
  \qed
\end{proof}

\begin{proof}{Proof of Theorem~\ref{thm:propertiesSolEF}}
  Let~$(x',\potvar',\beta')$ be an optimal solution to~\eqref{eq:singlelevelNonlin}.
  Then, $x' \in \B{\vertices}$, and we construct an optimal solution~$(\bar{x},\bar{\potvar},\bar{\beta})$ with the claimed properties.

  Let~$\mathscr{K}$ be the set of connected components of the graph arising from~$\graph$ by removing~$\{v \in \vertices : x'_v = 1\}$.
  Since~$x'$ is a solution to the~$k$-vertex cut problem, we have~$\card{\mathscr{K}} \geq k$.
  Consider the solution
  \begin{align*}
    \bar{\bx} &= \bx',
    &
       \bar{\potvar}_S
    &=
      \begin{cases}
        1, & \text{if } S \in \mathscr{K},\\
        0, & \text{otherwise,}
      \end{cases}
    &
       \bar{\beta}_{ew}
    &=
      x'_w.
  \end{align*}
  We claim that this solution is an optimal solution of~\eqref{eq:singlelevelNonlin}.
  Because the objective value is determined by the~$\bx$-variables only and~$\bar{\bx}=\bx'$, optimality follows once we can show that~$(\bar{\bx},\bar{\potvarvec},\bar{\boldsymbol{\beta}})$ is feasible.

  To show that~\eqref{eq:singlelevelNonlin:linking} holds, note that~$(1-\bar{x}_w)\bar{\beta}_{ew} = (1-\bar{x}_w)\bar{x}_w = 0$ because~$\bar{\bx} \in \B{\vertices}$.
  Therefore,
  \[
    \sum_{S \in \potatoes} (\card{S} - 1)\bar{\potvar}_S + \sum_{e \in \edges}\sum_{w \in e}(1 - \bar{x}_w)\bar{\beta}_{ew}
    =
    \sum_{S \in \potatoes} (\card{S} - 1)\bar{\potvar}_S
    =
    \sum_{S \in \mathscr{K}} (\card{S} - 1)
    \leq
    n - \sum_{v \in \vertices} x_v - k ,
  \]
  where that last inequality holds, because~$\mathscr{K}$ are the (at least)~$k$ connected components of the graph obtained from removing~$\{v \in \vertices : x'_v = 1\}$ from~$\graph$.
  That is, \eqref{eq:singlelevelNonlin:linking} holds.

  Regarding~\eqref{eq:singlelevelNonlin:dual}, let~$e = \{u,v\} \in \edges$.
  If~$\bar{\beta}_{eu} + \bar{\beta}_{ev} < 1$, then~$\bar{\beta}_{eu} + \bar{\beta}_{ev} = \bar{x}_u + \bar{x}_v = 0$, i.e., neither~$u$ nor~$v$ is deleted from~$\graph$.
  Consequently, since~$u$ and~$v$ are adjacent, there exists a connected component in~$\mathscr{K}$ that contains both~$u$ and~$v$.
  Hence, \eqref{eq:singlelevelNonlin:dual} is satisfied.

  Since~$(\bar{\bx},\bar{\potvarvec},\bar{\boldsymbol{\beta}})$ clearly satisfies the bound and integrality constraints of~\eqref{eq:singlelevelNonlin}, this solution is feasible for~\eqref{eq:singlelevelNonlin}, which concludes the proof.\qed
\end{proof}

\begin{proof}{Proof of Corollary~\ref{cor:strengtheningZIG}}
  On the one hand, if the~$k$-vertex cut problem for~$\graph$ is infeasible, also~\eqref{eq:singlelevelNonlin} is infeasible.
  Thus, adding further constraints does not change the feasible region.
  On the other hand, when the~$k$-vertex cut problem is feasible, Theorem~\ref{thm:propertiesSolEF} guarantees the existence of an integer optimal solution, where~$\potvar_S = 1$ if and only if~$S$ corresponds to a connected component of the reduced graph.
  When this solution has exactly~$k$ connected components, it is immediate that this solution satisfies all of the constraints proposed in the assertion.
  Otherwise, when the solution has more than~$k$ connected components, observe that the solution remains feasible when setting~$\potvar_{S_1} = \potvar_{S_2} = 0$ and~$\potvar_{S_1 \cup S_2} = 1$ for two distinct connected components~$S_1$ and~$S_2$. That is, we can merge some of the~$\potvarvec$-variables to ensure that exactly~$k$ of them take value~1.
  \qed
\end{proof}

\begin{proof}{Proof of Proposition~\ref{prop:symmetries}}
  Let~$(\bx,\potvarvec) \in \B{\vertices} \times \B{\potatoes}$ and let~$\gamma$ be an automorphism of~$\graph$.
  We need to show that~$(\gamma(\bx), \gamma(\potvarvec))$ is feasible for~\zigbases if and only if~$(\bx,\potvarvec)$ is feasible, and that both solutions have the same objective value.
  The latter is an immediate consequence of~$\gamma$ being an automorphism of~$\graph$, because
  \[
    \sum_{v \in \vertices} c_v \gamma(\bx)_v
    =              
    \sum_{v \in \vertices} c_v x_{\gamma^{-1}(v)}
    =              
    \sum_{v \in \vertices} c_{\gamma(v)} x_v
    =              
    \sum_{v \in \vertices} c_v  x_v.
  \]
  To prove that feasibility is preserved, we discuss the different families of constraints of~\zigbases in turn.

  First, consider~\eqref{eq:zig-bases:cardinality}.
  Since~$\gamma$ permutes the entries of both the~$\bx$-variables and~$\potvarvec$-variables, Constraint~\eqref{eq:zig-bases:cardinality} holds for~$(\bx,\potvarvec)$ if and only if it holds for~$(\gamma(\bx),\gamma(\potvarvec))$.

  Second, regarding~\eqref{eq:zig-bases:vertex_cover}, consider a solution~$(\bx,\potvarvec)$.
  Then, for every~$v \in \vertices$, we have
  \begin{align*}
    \sum_{S \in \potatoes(v)} \potvar_S + x_v
    =
    \sum_{S \in \potatoes(v)} \potvar_{\gamma^{-1}\gamma(S)} + x_{\gamma^{-1}\gamma(v)}
    =
    \sum_{S \in \potatoes(v)} \gamma(\potvarvec)_{\gamma(S)} + \gamma(\bx)_{\gamma(v)}
    =
    \sum_{S \in \potatoes(\gamma(v))} \gamma(\potvarvec)_{S} + \gamma(\bx)_{\gamma(v)}.
  \end{align*}
  That is, $(\bx,\potvarvec)$ satisfies~\eqref{eq:zig-bases:vertex_cover} for vertex~$v$ if and only if~$(\gamma(\bx),\gamma(\potvarvec))$ satisfies~\eqref{eq:zig-bases:vertex_cover} for vertex~$\gamma(v)$.
  As a consequence, $(\bx,\potvarvec)$ adheres to family~\eqref{eq:zig-bases:vertex_cover} if and only if~$(\gamma(\bx),\gamma(\potvarvec))$ does.

  Finally, to show that~$(\bx,\potvarvec)$ satisfies~\eqref{eq:zig-bases:clique_constrs} if and only if~$(\gamma(\bx),\gamma(\potvarvec))$ satisfies these constraints, we can use the same arguments as for~\eqref{eq:zig-bases:vertex_cover}.
  This concludes the proof.  
  \qed
\end{proof}

\section{Solution analysis of the Zachary’s karate club network}
\label{sec:karate}

The {\tt karate} instance represents the classical Zachary’s karate club network \citep{zachary1977information}, a widely used benchmark in community detection and graph partitioning. In this graph, vertices correspond to members of a university karate club, and edges connect pairs of individuals who were observed interacting outside formal activities. A conflict between the instructor (vertex~1) and the club administrator (vertex~34) eventually led to the split of the club into two factions, a structural characteristic that makes this instance particularly suitable for illustrating graph disconnection phenomena.
The optimal solutions of the $k$-VCP on this network exhibit a progressively increasing level of fragmentation as $k$ grows, with the $k$-vertex cut shifting from a single highly central vertex to a larger set of influential ones. For $k=3$, an optimal $3$-vertex cut consists of the single vertex~1 (the instructor), whose removal is sufficient to disconnect the graph into exactly three components. For $k=5$, an optimal $5$-vertex cut has cardinality $2$ and includes vertices~1 and~2. For $k=10$, an optimal $10$-vertex cut has cardinality $ 4$ and includes vertices~1, 3, 33, and~34, whose removal disconnects the network into ten components.
A noticeable property of this network is its low density: many vertices have only two or three neighbors. This explains why the size of optimal vertex cuts remains moderate even for larger values of $k$, and why the corresponding solutions contain several components that are essentially singletons (one singleton for $k=3$ and seven for $k=10$). As expected, the vertices selected in the cuts tend to be those with high degree, which justifies the repeated appearance of vertices~1 and~34 (the instructor and the club administrator) in optimal solutions for larger values of~$k$.

\section{The connectivity cut}
\label{sec:connectivity_cut}

{For both the connectivity cut and the IDH (Appendix~\ref{sec:heur}), we use the following notation. Given a connected graph $\graph[U]$, with $U \subseteq \vertices$, and a vector of non-negative vertex costs $\bc_U=(c_v)_{v \in U}$, we define its \emph{weighted vertex connectivity} as the minimum total cost of a vertex set $W \subseteq U$ whose removal disconnects $\graph[U]$. Equivalently, $W$ is such that $\graph[U \setminus W]$ has more connected components than $\graph[U]$. We denote the weighted vertex connectivity by $\tilde{\kappa}(\graph[U],\bc_U)$ and we set it to $+\infty$ for complete graphs since they cannot be disconnected.
 In particular, if $\graph[U]$ is not complete and $c_v=1$ for all $v \in U$, then $\tilde{\kappa}(\graph[U],\bc_U)$ coincides with the \emph{vertex connectivity} of $\graph[U]$. We extend this definition to graphs that are not necessarily connected by taking the minimum weighted vertex connectivity over their connected components. Thus, if $\graph[U]$ is disconnected, then $\tilde{\kappa}(\graph[U],\bc_U)$ is defined as the minimum value of $\tilde{\kappa}(\graph[Q],\bc_Q)$ over all connected components $\graph[Q]$ of $\graph[U]$, with $Q \subseteq U$.}

{If the number of connected components of $\graph$ is smaller than $k$, as in all non-trivial $k$-VCP instances, then every feasible integer solution of \zigbases satisfies
\begin{equation}
  \label{eq:connectivity_cut}
  \sum_{v \in \vertices} c_v x_v \geq \tilde{\kappa}(\graph,\bc_{\vertices}),
\end{equation}
and inequality~\eqref{eq:connectivity_cut} can be added to \zigbases to strengthen its LP relaxation. We refer to~\eqref{eq:connectivity_cut} as the \emph{connectivity cut}.
Intuitively, this inequality states that, whenever the current graph has fewer than $k$ connected components, any feasible solution must remove vertices whose total cost is at least the minimum cost required to disconnect one of its connected components.}

{For a given $U \subseteq \vertices$ such that $\graph[U]$ is a non-complete and connected graph, the value $\tilde{\kappa}(\graph[U],\bc_U)$  can be computed in polynomial time as follows. We consider all pairs of nonadjacent vertices $u,v \in U$ and solve a minimum $u$-$v$ cut problem on an auxiliary directed network. The value $\tilde{\kappa}(\graph[U],\bc_U)$ is then equal to the minimum $u$-$v$ cut value obtained over all these pairs. In the auxiliary directed network, vertices $u$ and $v$ are the source and the sink, respectively, whereas each other vertex $w$ is split into two nodes $w_1$ (the ``in'' copy) and $w_2$ (the ``out'' copy), connected by an arc $(w_1,w_2)$ with capacity $c_w$. For each edge $\{p,q\}$ of the original graph, we add infinite-capacity arcs from the ``out'' copy of one endpoint to the ``in'' copy of the other endpoint. In particular, if both endpoints are split, we add arcs~$(p_2,q_1)$ and~$(q_2,p_1)$. Otherwise, if one endpoint, say $p$, coincides with $u$ or $v$, then the added arcs are~$(p,q_1)$ and~$(q_2,p)$, since $u$ and $v$ are not split. Any finite-capacity minimum $u$-$v$ cut in this auxiliary network therefore cuts only arcs of the form~$(w_1,w_2)$. This construction follows the standard way of computing the vertex connectivity of a graph, with the only difference that, in our case, we set the capacities of the arcs~$(w_1,w_2)$ to the weights of the corresponding vertices~$w$, instead of $1$. As an illustration, consider the example graph $\graph[U]$ displayed in Figure~\ref{fig:aux_network_connectivity_cut:graph}, where $U=\{u, p, q, v\}$ and $\edges[U] = \{\{u, p\},\{p, q\}, \{q, v\}\}$. To illustrate the construction of the auxiliary network, we consider the pair of nonadjacent vertices $\{u,v\}$.Hence, vertex~$p$ is split into $p_1$ and $p_2$, connected by an arc of capacity~$c_p$. Similarly, vertex~$q$ is split into $q_1$ and $q_2$, connected by an arc of capacity~$c_q$. The source node $u$ and the sink node $v$ are linked with infinite-capacity arcs to their neighbors (see Figure~\ref{fig:aux_network_connectivity_cut:network}). Assuming $c_p < c_q$, a minimum~$u$-$v$ cut is crossed by arc~$(p_1,p_2)$ and has value~$c_p$.}

\begin{figure}[htbp]
  \centering
  \begin{subfigure}[c]{0.35\textwidth}
    \centering
    \begin{tikzpicture}[baseline=(current bounding box.center)]
      \tikzstyle{v} += [circle,draw=black,thick,inner sep=2pt,minimum size=2mm];
      \node (u)  at (0,0)   [v,label=above:{\scriptsize $u$}] {};
      \node (p)  at (1.2,0) [v,label=above:{\scriptsize $p$},
                               label=below:{\scriptsize $c_p$}] {};
      \node (q)  at (2.4,0) [v,label=above:{\scriptsize $q$},
                               label=below:{\scriptsize $c_q$}] {};
      \node (vv) at (3.6,0) [v,label=above:{\scriptsize $v$}] {};
      \draw[-,thick] (u) -- (p) -- (q) -- (vv);
    \end{tikzpicture}
    \subcaption{ }
    \label{fig:aux_network_connectivity_cut:graph}
  \end{subfigure}
  \qquad
  \begin{subfigure}[c]{0.5\textwidth}
    \centering
    \begin{tikzpicture}
      \tikzstyle{v} += [circle,draw=black,thick,inner sep=2pt,minimum size=2mm];
      \node (u)  at (0,0)   [v,fill=black,label=below:{\scriptsize $u$}] {};
      \node (p1) at (2,1) [v,fill=black,label=above:{\scriptsize $p_1$}] {};
      \node (p2) at (2,-1) [v,label=below:{\scriptsize $p_2$}] {};
      \node (q1) at (4,1) [v,label=above:{\scriptsize $q_1$}] {};
      \node (q2) at (4,-1) [v,label=below:{\scriptsize $q_2$}] {};
      \node (vv) at (6,0) [v,label=below:{\scriptsize $v$}] {};
      \draw[->,thick]        (u)  to (p1);
      \draw[->,thick]        (p2)  to (u);
      \draw[->,thick,dashed] (p1) to node[left]{\scriptsize $c_p$}   (p2);
      \draw[->,thick]        (q1) to node[right]{\scriptsize $c_q$}   (q2);
      \draw[->,thick]        (p2) to   (q1);
      \draw[->,thick]        (q2) to   (p1);      
      \draw[->,thick]        (q2) to (vv);
      \draw[->,thick]        (vv) to (q1);
    \end{tikzpicture}
    \subcaption{ }
    \label{fig:aux_network_connectivity_cut:network}
  \end{subfigure}

  \smallskip
  
  \caption{{Auxiliary network construction for a graph $\graph[U]$, where $U=\{u, p, q, v\}$ and~$\edges[U] = \{\{u, p\},\{p, q\}, \{q, v\}\}$. In the right subfigure, the dashed arc $(p_1,p_2)$ belongs to a minimum $u$-$v$ cut of value $c_p$. Black-filled nodes belong to the source side of the cut, while white nodes belong to the sink side.  Arcs without labels have infinite capacity.}
  }
  \label{fig:aux_network_connectivity_cut}
\end{figure}

{Preliminary computational experiments indicate that the effectiveness of the connectivity cut depends on the value of $k$.
For moderate values (up to $k = 15$), the cut typically strengthens the LP relaxation and improves the overall performance, while incurring only a limited computational overhead.
For larger values of $k$, the cut may instead become less beneficial and interfere with the internal branching mechanisms of the solver.
For this reason, in our computational study we activate the connectivity cut only for $k\le 15$.}

\section{{The Iterative Disconnection Heuristic}}
\label{sec:heur}
{
Our heuristic, denoted by the Iterative Disconnection Heuristic (IDH), aims to 
find a set~$\vertices_0 \subseteq \vertices$ such that 
$\graph[\vertices \setminus \vertices_0]$ has at least~$k$ connected components, i.e., a heuristic solution to the $k$-VCP.
Starting with an empty set $\vertices_0 \gets \emptyset$, IDH repeatedly considers the 
induced subgraph $\graph[\vertices \setminus \vertices_0]$. As long as this 
subgraph has fewer than~$k$ connected components, the heuristic proceeds as 
follows. For each non-complete connected component $\graph[H]$ of 
$\graph[\vertices \setminus \vertices_0]$, with 
$H \subseteq \vertices \setminus \vertices_0$, let 
$\bc_H = (c_v)_{v \in H}$ and compute 
$\tilde{\kappa}(\graph[H],\bc_H)$, together with a minimum-cost vertex set 
$W_H \subseteq H$ whose removal disconnects $\graph[H]$ (see Appendix \ref{sec:connectivity_cut} for the definition of the notation).  IDH then selects a 
component $H^\star$ attaining the smallest value of 
$\tilde{\kappa}(\graph[H],\bc_H)$ and updates~$\vertices_0 \gets \vertices_0 \cup W_{H^\star}$. If, after this first 
iteration, $\graph[\vertices \setminus \vertices_0]$ already has at least~$k$ 
connected components, then the solution found is optimal for the $k$-VCP: indeed, 
any feasible solution must remove vertices of total cost at least 
$\tilde{\kappa}(\graph,\bc)$, and IDH has removed a vertex set attaining this 
bound. The heuristic terminates once~$\graph[\vertices \setminus \vertices_0]$ has at least~$k$ connected components. 
If, before reaching this condition, all connected components of~$\graph[\vertices \setminus \vertices_0]$ are complete graphs, IDH fails, 
since no vertex removal can disconnect any of the remaining components. In this case, the heuristic does not find a solution, even if one may exist.
}

\section{Benchmark instances}
\label{sec:compInstances}
For our numerical experiments, we use the same instances as~\citet{FurLMP2020}.
Their test set consists of~118 graphs, where~50 originate from the 2nd DIMACS challenge 
(called the {\tt Coloring} family), 9 come from the 10th DIMACS challenge on graph 
partitioning and clustering (called the {\tt Partitioning} family), and~59 intersection graphs from 
\citet{BS05b} (called the {\tt Intersection} family) which have been proposed
for the vertex separator problem.
Following~\citet{FurLMP2020}, we consider the~$k$-VCP for~$k \in \{5,10,15,20\}$ for these graphs.

Across these three families of graphs, the number of vertices ranges from 
11 to 297 (with an average of about~101), and the number of edges ranges 
from 20 to 7763 (with an average of about 1426).
On average, instances in the {\tt Coloring} and {\tt Intersection} families 
tend to have substantially more edges (typically between one and several thousand), 
whereas the {\tt Partitioning} family contains graphs with noticeably larger vertex sets 
(on average close to 150 vertices).
Tables~\ref{tab:unweighted_best_known_coloring_partitioning}--\ref{tab:weighted_best_known_intersection} provide for each graph in our test set the exact number of vertices and edges as well as the stability number.
Recall that the stability number~$\alpha(\graph$) of a graph~$\graph$ is the maximum size of a stable set in~$\graph$.

Due to Lemma~\ref{lem:stableset}, an instance of the~$k$-VCP for graph~$\graph$ is feasible if and only if~$\alpha(\graph) \geq k$.
As~\citet{FurLMP2020}, we therefore consider, for fixed~$k \in \{5,10,15,20\}$, only graphs~$\graph$ with~$\alpha(\graph) \geq k$.
For~$k=5$, all graphs in our test set have a stability number of at least~5, i.e., no graph is discarded.
Following~\citet{FurLMP2020}, we further simplify an instance~$(\graph,k)$ of the~$k$-VCP by
(i) identifying vertices that must belong to any feasible $k$-vertex cut for the given~$k$; such vertices must be fixed to the vertex cut and removed from the graph; and  
(ii) instances whose reduced graphs already contain at least~$k$ connected components are eliminated because they are trivial.
For details on this process, we refer to \citet{FurLMP2020}, which also reports the number of vertices and edges of the graphs after preprocessing.  
This reduction is effective for a substantial subset of instances: depending on the value of~$k$,
between 11 and 20 graphs undergo vertex deletions, and while many reductions remove only a few vertices,
in some cases more than one hundred vertices are eliminated, leading in several instances to a substantial shrinkage
of the input graph.

After the filtering and reduction procedure, a total of 304 feasible and non-trivial $k$-VCP instances are obtained.
Table~\ref{tab:instances-final} reports the number of instances for each family of graphs and each
value of~$k$.
Based on these graphs we define two test sets.
The \emph{unweighted} test set consists of all~304 graphs with unit vertex costs.
The \emph{weighted} test set uses the same graphs, but assign each vertex~$v \in \vertices$ a cost~$c_v$ drawn uniformly from $\{1,2,\dots,10\}$.

The preprocessing and filtering of instances as  described above is implemented in \texttt{Gurobi} 11.0.3 and is applied off-line to all instances. 

\begin{table}[h!]
  \centering
  \small
  \renewcommand\arraystretch{1.3}
  \tabcolsep=12pt
  \caption{Number of $k$-VCP instances used in the experiments, grouped by family of graphs and value of $k$.}
  \label{tab:instances-final}
  \begin{tabular}{lrrrrr}
    \toprule
    {Family} & \textbf{$k=5$} & \textbf{$k=10$} & \textbf{$k=15$} & \textbf{$k=20$} & {Total} \\
    \midrule
    {\tt Coloring}      & 42 & 32 & 29 & 28 & 131 \\
    {\tt Partitioning}  & 9  & 9  & 9  & 8  & 35  \\
    {\tt Intersection}  & 56 & 39 & 27 & 16 & 138 \\
    \midrule
    {Total} & {107} & {80} & {65} & {52} & {304} \\
    \bottomrule
  \end{tabular}
\end{table}

\section{Separation of the clique-partitioning constraints}
\label{sec:sep}

The clique-partitioning constraints~\eqref{eq:zig-base:clique_constrs} of \oldBP and \zigbase guarantee that the subsets of vertices selected by the variables~$\potvarvec$ are pairwise disjoint on the induced graph~$\graph[\vertices\setminus\vertices_0]$.
Since the family of clique-partitioning constraints can be exponentially large, it is practically infeasible to explicitly add all of them to~\eqref{eq:zig-base} though.
As observed by \citet{CFLMMM18}, to obtain a valid formulation, it is sufficient to consider a family of clique-partitioning inequalities associated with a clique cover of the edge set. This is the approach adopted in formulation \oldBP, where a collection of maximal cliques is used to partition all edges.
The quality of the LP relaxation, however, strongly depends on the family of clique-partitioning constraints that is selected, cf.\ Section~\ref{sec:lp_comparison}.

A natural question is to understand which dual bound can be obtained when we explicitly separate violated clique-partitioning inequalities, instead of restricting ourselves to those associated with a fixed family of cliques.
To address this question, we introduce and analyze the separation problem for the clique-partitioning inequalities.
Given a vector~$\potvarvec \in \R^{\potatoes}$, the separation problem of clique-partitioning inequalities is to decide whether there exists a clique~$C \in \cliques$ such that~$\sum_{S \in \potatoes(C)} \potvar_S > 1$.
In the following, we prove that the separation problem is $\mathcal{NP}$-hard for solutions of a restricted master problem.
\begin{proposition}
  \label{prop:sep-np-hard}
  The separation problem for clique-partitioning inequalities in $\mathcal{NP}$-hard.
\end{proposition}
\begin{proof}{Proof}
We prove the claim by a reduction from the $\mathcal{NP}$-hard Maximum Clique Problem (MCP).
That is, given an undirected graph~$\graph = (\vertices,\edges)$ and a positive integer~$\ell \geq 2$, one needs to decide whether there exists a clique of size at least~$\ell$ in~$\graph$.
We construct an instance of the separation problem as follows.

Let $\potatoes$ be the family of all singleton subsets of $\vertices$, i.e., $\potatoes = \{\{v\} : v \in \vertices \}$.
For each $S = \{v\} \in \potatoes$, set $\potvar_S = \frac{1}{\ell - 1}$.
On the one hand, if there exists a clique~$C$ of size at least~$\ell$ in~$\graph$, then~$\sum_{S \in \potatoes(C)} \potvar_S = \sum_{v \in C} \potvar_{\{v\}}\geq \frac{\ell}{\ell - 1} > 1$, i.e., there exists a violated clique-covering inequality. 
On the other hand, if there exists a violated clique-covering inequality, there exists a clique~$C$ in~$\graph$ with~$\sum_{v \in C} \potvar_{\{v\}} = \sum_{S \in \potatoes(C)} \potvar_S > 1$.
Since all~$\potvarvec$-variables have the same value~$\frac{1}{\ell - 1}$, this means that the left-hand side of this inequality evaluates to~$\frac{\card{C}}{\ell - 1} > 1$.
That is, $C$ has size at least~$\ell$, which concludes the proof.
\qed
\end{proof}
Due to the separation problem being $\mathcal{NP}$-hard, we cannot expect to find a polynomial time algorithm for separating clique-covering inequalities.
We therefore suggest to solve the separation problem by means of integer programming.
To this end, observe that there is a clique-covering inequality that is violated by a non-negative vector~$\potvarvec$ if and only if
\[
\max_{C \in \cliques} \sum_{S \in \potatoes(C)} \potvar_C > 1.
\]
We can find a maximizing clique~$C$ by solving an integer program.
 Let us define a variable $\beta_v \in \B{}$, for all $v \in \vertices$, taking value $1$ if vertex $v$ is in $C$, and 0 otherwise; we also define a variable $\gamma_S \in \B{}$ for all $S \in \potatoes$, equal to $1$ if clique $C$ intersects $S$, and $0$ otherwise. An ILP formulation for the separation problem reads as follows:

\begin{subequations}
    \label{eq:separation_ilp}
        \begin{align}
          \text{(SEP)} & & \max_{{\substack{\boldsymbol{\beta} \in \{0,1\}^{\vertices}\\ \boldsymbol{\gamma} \in \{0,1\}^{\potatoes}}}} \, \, \sum_{S \in \potatoes} \potvar_S \, \gamma_S  \label{separation_ilp:obj}&&&\\[1 ex]
          & &  \beta_u + \beta_v & \le 1,  && \{u, v\} \in \overline{\edges},\label{separation_ilp:clique_constrs}\\[1 ex]
          & &  \gamma_S &\le \sum_{v \in S} \beta_v, && S \in \potatoes. \label{separation_ilp:intersection_constrs}
        \end{align}
\end{subequations}
where $\overline{\edges} = \{\{u, v\}  \, : \, u, v \in \vertices \text{ and } \{u, v\} \notin \edges \}$ is the set of non-edges of $\graph$. The objective function \eqref{separation_ilp:obj} maximizes the total weight of the subsets in $\potatoes$ that are intersected by the selected clique, where the weight of each subset $S$ is given by the value $\potvar_S$ of the corresponding variable in the current fractional solution. Constraints \eqref{separation_ilp:clique_constrs} ensure that, for each non-edge $\{u, v\} \in \overline{\edges}$, at most one of its endpoints can be part of the selected clique, thus guaranteeing that any subset of vertices $C = \{v \in V \, : \, \beta_v = 1\}$ is a clique. Finally, Constraints \eqref{separation_ilp:intersection_constrs} enforce that, for each subset of vertices $S \in \potatoes$, its associated variable $\gamma_S$ can take value $1$ only if at least one of the vertices in $S$ is part of the solution clique~$C$. We remark that, since $\potvar_S \ge 0$ for all $S \in \potatoes$, any solution $(\boldsymbol{\beta}, \boldsymbol{\gamma})$ to \eqref{eq:separation_ilp} with $\gamma_S < \min \{1, \, \sum_{v \in S} \beta_v \}$ for some $S \in \mathcal{S}$ cannot be optimal. This is because setting $\gamma_S = \min \{1, \, \sum_{v \in S} \beta_v \}$ we obtain a solution with a greater objective function value, without violating any constraint.
It follows that the integrality of the $\boldsymbol{\beta}$ variables suffices to guarantee that the $\boldsymbol \gamma$ variables are integral.
As a consequence, the integrality of the latter set of variables can be relaxed.

\medskip 
We now turn to a computational assessment of the effectiveness of the exact separation procedure described above. Given the $\mathcal{NP}$-hardness result of Proposition~\ref{prop:sep-np-hard}, we restrict the use of Model~\eqref{eq:separation_ilp} to the root node of the branch-and-bound tree.
In this way, we can evaluate the impact of optimizing over the closure of the clique-partitioning constraints at the root node, without incurring an excessive computational overhead.
In our experiments, we compare two settings on the entire set of unweighted instances.
In the first one, denoted by \BPs, we do not perform any separation of clique-partitioning inequalities.
Instead, we generate at the beginning a family of maximal cliques that covers all edges of the graph, and we add the corresponding clique-partitioning constraints to the formulation.
In the second setting, denoted by \BPCs, we start from the same initial family of maximal cliques and then apply the exact separation procedure at the root node, by solving~\eqref{eq:separation_ilp} iteratively until no violated clique-partitioning inequality is found.
The initial family of maximal cliques used in both settings is constructed by means of the greedy edge-cover procedure detailed in Section~\ref{sec:initRMP}.
The resulting family of maximal cliques provides an edge cover of $\edges$ and is used to define the initial set of clique-partitioning constraints in both \BPs\ and \BPCs.

Table~\ref{tab:root-sep} reports a comparison between the \BPCs\ and \BPs\ settings at the root node. The table includes only the instances for which \BPCs\ generates at least one violated clique-partitioning inequality at the root node. Over these~${47}$ instances, \BPCs\ generates on average $8.5$ clique-partitioning cuts at the root node, with a maximum of $44$ cuts on a single instance. Rows are grouped by the value of parameter $k$. For each value of $k$, {column ``\# inst'' reports the number of instances for which \BPCs\ generates at least one violated clique-partitioning inequality at the root node, over the total number of instances with that value of $k$.}
The first part of Table~\ref{tab:root-sep} (columns \emph{integrality gap}) reports the integrality gap, defined as the percentage difference between the best known integer solution value and the optimal value of the LP relaxation at the root node for a given setting, taken with respect to the best known integer solution value, i.e, $100 \cdot \frac{z^{\text{BEST}} - z^{\text{LP}}}{z^{\text{BEST}}}$,
where $z^{\text{BEST}}$ is the best known integer solution objective value and $z^{\text{LP}}$ is the optimal value of the root-node LP relaxation under the considered setting.
For each $k$, we then report the average and maximum gap over the corresponding instances.
Columns ``LP time'' report instead the average and maximum CPU time (in seconds) to solve the root-node LP relaxation in each setting.

\begin{table}
\centering
\small
\renewcommand
\arraystretch{1.5}
\tabcolsep=4.0pt
\caption{Evaluation of the effect of separating clique-partitionig inequalities at the root node. {The last row reports the total number of instances for the ``\# inst'' column and average values for the remaining columns.}}
\label{tab:root-sep}
\begin{tabular}{l d{3.0} rrrrrrrrrrrrrr}
\hline
                &         &  &  & \multicolumn{5}{c}{\BPCs}                                             &  &  & \multicolumn{5}{c}{\BPs}                                             \\ \cline{5-9} \cline{12-16} 
                &         &  &  & \multicolumn{2}{r}{LP time} &  & \multicolumn{2}{r}{integrality gap} &  &  & \multicolumn{2}{r}{LP time} &  & \multicolumn{2}{r}{integrality gap} \\ \cline{5-6} \cline{8-9} \cline{12-13} \cline{15-16} 
$k$               & \multicolumn{1}{c}{\phantom{aaa,}\# inst} &  &  & avg            & time       &  & avg                & max            &  &  & avg            & max        &  & avg                & max            \\ \hline
                &         &  &  &                &            &  &                    &                &  &  &                &            &  &                    &                \\[-1.5ex]
5               & 22/107  &  &  & 9.7            & 35.3       &  & 30.7               & 63.5           &  &  & 0.8            & 4.6        &  & 32.6               & 63.6           \\
10              & 14/80   &  &  & 18.5           & 90.3       &  & 17.7               & 59.0           &  &  & 1.2            & 5.1        &  & 19.7               & 59.0           \\
15              & 7/65    &  &  & 30.9           & 113.6      &  & 22.7               & 55.3           &  &  & 0.8            & 1.2        &  & 24.0               & 55.4           \\
20              & 4/52    &  &  & 16.9           & 31.2       &  & 18.9               & 49.5           &  &  & 0.7            & 0.8        &  & 20.6               & 49.5           \\[1.5ex] \hline
{Overall} & 47/304  &  &  & {16.1}  &   &  & {24.6}      &       &  &  & {0.9}   &   &  & {26.4}      &       \\ \hline
\end{tabular}
\end{table}

We observe that the use of exact separation at the root node significantly increases the LP time.
For example, for~$k = 5$ the average root-node time grows from $0.8$ seconds in \BPs\ to $9.7$ seconds in \BPCs\, and for $k = 10$ it grows from $1.2$ to $18.5$ seconds.
On average over all values of $k$, the root-node time increases from $0.9$ seconds for \BPs\ to~$16.1$ seconds for \BPCs.
Similar trends can be seen for the maximum times, where \BPCs\ frequently requires one to two orders of magnitude more CPU time than \BPs\ to process the root node.
In contrast, the improvement in the root-node integrality gap obtained by \BPCs\ is limited.
For $k = 5$, the average gap decreases only from $32.6\%$ to~$30.7\%$, and for $k = 10$ from $19.7\%$ to $17.7\%$.
For larger values of $k$, the average gaps of the two settings are very close.
Overall, the average gap over all instances with cuts decreases from $26.4\%$ in \BPs\ to $24.6\%$ in \BPCs, which corresponds to a rather modest improvement compared with the increase in root-node time.
We also evaluated the impact of exact separation at the root node on the overall performance of the algorithm.
Running the full algorithm with \BPCs\ (separating clique-partitioning inequalities only at the root node) does not lead to any noticeable advantage compared with \BPs\ in terms of total solution time or number of explored nodes.
On the contrary, when looking at the optimality gaps at the time limit on the subset of instances that remain unsolved, the behavior of \BPCs\ is slightly worse.
On the $28$ out of $47$ instances with root-node cuts that are not solved to proven optimality within the time limit, the average optimality gap --computed as $100 \cdot ( \text{UB} - \text{LB} ) / \text{UB}$ at the time limit-- is $29.08\%$ for \BPCs\ and $28.88\%$ for \BPs.
These results indicate that, for the instances considered, the exact separation of clique-partitioning inequalities at the root node is not an effective strategy.
The reduction in the root-node integrality gap is marginal and does not translate into better global performance, while the additional cost in terms of LP time is significant.

\section{Supplementary detailed computational results }
\label{sec:suppTables}

Figure~\ref{fig:perf-all} reports performance profiles for COMP, HYB, and BP$^\star$ on the unweighted (left column) and weighted (right column) instances, for each value of $k$.
Overall, the plots confirm the already observed good behavior of BP$^\star$, but they also highlight some regimes where COMP or HYB can be competitive.
For unweighted instances, BP$^\star$ is generally the most robust algorithm: its profile is typically above that of HYB, which in turn dominates COMP for most values of the performance ratio~$\tau$ and for most $k$. In particular, for $k \in \{10,15,20\}$ the curve of BP$^\star$ remains clearly above that of COMP over the whole range of~$\tau$, while the gap with HYB depends on~$k$ and tends to decrease for larger values of~$k$.
The performance difference is less pronounced on weighted instances.
For $k=5$, COMP has the best profile, especially for large values of~$\tau$, while BP$^\star$ dominates HYB over the entire range.
For $k=10$, BP$^\star$ keeps dominating HYB for all values of $\tau$, while for $k=15$ and $k=20$ the two algorithms have very close profiles, although the curve of BP$^\star$ eventually overtakes that of HYB. The performance of COMP deteriorates as the value of $k$ increases, and its profile stays well below both HYB and BP$^\star$ for all $k \in \{10, 15, 20\}$.
Taken together, these profiles indicate that BP$^\star$ is overall the most effective algorithm, particularly on unweighted instances and moderate values of $k$, while COMP can still be preferable on the easiest instances with small $k$ (particularly if weighted), and HYB remains competitive on all the instances, but mostly on weighted instances with higher values of $k$.

\begin{figure}[!h]
    \centering

    \begin{subfigure}[t]{0.48\textwidth}
        \centering
        \begin{tikzpicture}
 \begin{axis}[const plot,
  cycle list={
  {blue, dashed},
  {green!80!black, dashdotted},
  {red,solid},
  {pink, solid}},
                xmode = log,
                log basis x = 10,
                xmin = 1, xmax = 1000,
                xtick = {10, 100, 1000},
                xticklabels = {$10$, $10^2$, $10^3$},
                ymin=0.00, ymax=0.85,
                ymajorgrids,
                ytick={0,0.2,0.4,0.6,0.8,1.0},
                yticklabels={0,20,40,60,80,100},
                xlabel={ $\tau$ ($k=5$)},
                ylabel={ \% of instances},
                legend pos={south east},
                width=\textwidth,
                height=5cm,
                tick align=outside,
                tick style={semithick},
                grid style={dashed,gray!70}
            ]
  \addplot+[mark=none, line width=1.2pt] coordinates {
    (1.0000,0.1589)
    (1.0524,0.1682)
    (1.2765,0.1682)
    (1.2881,0.1776)
    (1.3085,0.1776)
    (1.3665,0.1869)
    (1.5183,0.1869)
    (1.5284,0.1963)
    (1.5435,0.2056)
    (1.5477,0.2150)
    (1.6603,0.2150)
    (1.7338,0.2243)
    (1.9883,0.2243)
    (2.0863,0.2336)
    (2.1483,0.2430)
    (2.3568,0.2523)
    (2.4205,0.2617)
    (2.4927,0.2710)
    (2.4942,0.2710)
    (2.5488,0.2804)
    (2.6697,0.2897)
    (2.8834,0.2991)
    (3.0876,0.2991)
    (3.4169,0.3084)
    (3.5554,0.3084)
    (3.7540,0.3178)
    (3.9031,0.3271)
    (4.2503,0.3364)
    (4.2579,0.3458)
    (4.3476,0.3551)
    (4.5962,0.3551)
    (4.9172,0.3645)
    (4.9725,0.3738)
    (5.1881,0.3832)
    (5.2666,0.3925)
    (5.5258,0.3925)
    (5.5383,0.4019)
    (5.5596,0.4019)
    (5.9979,0.4112)
    (6.5570,0.4206)
    (7.1200,0.4206)
    (7.2202,0.4299)
    (7.8524,0.4299)
    (8.2825,0.4393)
    (9.0268,0.4486)
    (9.3476,0.4486)
    (9.8226,0.4579)
    (11.2098,0.4579)
    (11.4140,0.4673)
    (11.7978,0.4766)
    (12.0010,0.4860)
    (13.6336,0.4953)
    (14.1364,0.5047)
    (16.5539,0.5140)
    (16.8817,0.5140)
    (17.6049,0.5234)
    (17.7529,0.5327)
    (20.5411,0.5421)
    (22.7966,0.5421)
    (23.8326,0.5514)
    (26.0474,0.5514)
    (29.8645,0.5607)
    (34.1242,0.5607)
    (35.5712,0.5701)
    (39.7143,0.5701)
    (42.4466,0.5794)
    (43.6629,0.5888)
    (48.7887,0.5981)
    (52.3329,0.5981)
    (53.3261,0.6075)
    (55.2531,0.6075)
    (70.7793,0.6168)
    (73.7658,0.6168)
    (83.7091,0.6262)
    (84.9671,0.6262)
    (85.3663,0.6355)
    (102.0609,0.6449)
    (105.9213,0.6542)
    (108.0463,0.6636)
    (111.3129,0.6729)
    (120.2771,0.6729)
    (141.1637,0.6822)
    (147.8013,0.6916)
    (167.4308,0.7009)
    (172.4598,0.7009)
    (174.9621,0.7103)
    (176.2219,0.7103)
    (177.8958,0.7196)
    (185.3844,0.7290)
    (202.9538,0.7383)
    (250.1390,0.7477)
    (332.6204,0.7570)
    (395.9383,0.7570)
    (432.2144,0.7664)
    (563.6042,0.7664)
    (615.7330,0.7757)
    (624.1752,0.7757)
    (790.5379,0.7850)
    (810.6394,0.7944)
    (812.5865,0.8037)
    (833.9576,0.8037)
    (3355.3001,0.8131)
    (4747.1468,0.8224)
    (9216.8265,0.8224)
  };
  \addlegendentry{COMP}
  \addplot+[mark=none, line width=1.2pt] coordinates {
    (1.0000,0.3084)
    (1.0524,0.3084)
    (1.1211,0.3178)
    (1.2881,0.3178)
    (1.3085,0.3271)
    (1.3665,0.3271)
    (1.3894,0.3364)
    (1.3979,0.3364)
    (1.4653,0.3458)
    (1.4956,0.3458)
    (1.5074,0.3551)
    (1.7338,0.3551)
    (1.7485,0.3645)
    (1.7491,0.3645)
    (1.8810,0.3738)
    (1.9360,0.3738)
    (1.9549,0.3832)
    (1.9883,0.3925)
    (2.4927,0.3925)
    (2.4942,0.4019)
    (2.9269,0.4019)
    (2.9525,0.4112)
    (2.9770,0.4206)
    (2.9792,0.4206)
    (3.0118,0.4299)
    (3.0876,0.4393)
    (4.3476,0.4393)
    (4.5962,0.4486)
    (5.2666,0.4486)
    (5.2954,0.4579)
    (6.5570,0.4579)
    (7.1200,0.4673)
    (7.3901,0.4673)
    (7.8524,0.4766)
    (9.2139,0.4766)
    (9.3476,0.4860)
    (9.8226,0.4860)
    (10.7285,0.4953)
    (10.9565,0.4953)
    (11.2098,0.5047)
    (29.8645,0.5047)
    (32.0356,0.5140)
    (48.7887,0.5140)
    (50.3812,0.5234)
    (52.3329,0.5327)
    (53.3261,0.5327)
    (55.2531,0.5421)
    (70.7793,0.5421)
    (73.7658,0.5514)
    (83.7091,0.5514)
    (84.9671,0.5607)
    (167.4308,0.5607)
    (172.4598,0.5701)
    (174.9621,0.5701)
    (176.2219,0.5794)
    (475.9526,0.5794)
    (484.7676,0.5888)
    (563.6042,0.5981)
    (615.7330,0.5981)
    (624.1752,0.6075)
    (812.5865,0.6075)
    (833.9576,0.6168)
    (4747.1468,0.6168)
    (5160.8910,0.6262)
    (9216.8265,0.6262)
  };
  \addlegendentry{HYB}
  \addplot+[mark=none, line width=1.2pt] coordinates {
    (1.0000,0.5327)
    (1.1211,0.5327)
    (1.1267,0.5421)
    (1.1269,0.5514)
    (1.2765,0.5607)
    (1.3894,0.5607)
    (1.3979,0.5701)
    (1.4653,0.5701)
    (1.4956,0.5794)
    (1.5074,0.5794)
    (1.5183,0.5888)
    (1.5477,0.5888)
    (1.6603,0.5981)
    (1.7485,0.5981)
    (1.7491,0.6075)
    (1.8810,0.6075)
    (1.9360,0.6168)
    (2.8834,0.6168)
    (2.9269,0.6262)
    (2.9770,0.6262)
    (2.9792,0.6355)
    (3.4169,0.6355)
    (3.5554,0.6449)
    (5.2954,0.6449)
    (5.5258,0.6542)
    (5.5383,0.6542)
    (5.5596,0.6636)
    (7.2202,0.6636)
    (7.3901,0.6729)
    (9.0268,0.6729)
    (9.2139,0.6822)
    (10.7285,0.6822)
    (10.9565,0.6916)
    (16.5539,0.6916)
    (16.8817,0.7009)
    (20.5411,0.7009)
    (20.9393,0.7103)
    (22.7966,0.7196)
    (23.8326,0.7196)
    (26.0474,0.7290)
    (32.0356,0.7290)
    (33.4061,0.7383)
    (34.1242,0.7477)
    (35.5712,0.7477)
    (39.7143,0.7570)
    (111.3129,0.7570)
    (120.2771,0.7664)
    (332.6204,0.7664)
    (395.9383,0.7757)
    (432.2144,0.7757)
    (475.9526,0.7850)
    (5160.8910,0.7850)
    (8777.9300,0.7944)
    (9216.8265,0.7944)
  };
  \addlegendentry{BP$^\star$}
  \end{axis}
\end{tikzpicture}
        \vspace{-1em}
    \end{subfigure}
    \hfill
    \begin{subfigure}[t]{0.48\textwidth}
        \centering
        \begin{tikzpicture}
 \begin{axis}[const plot,
  cycle list={
  {blue, dashed},
  {green!80!black, dashdotted},
  {red,solid},
  {pink, solid}},
                xmode = log,
                log basis x = 10,
                xmin = 1, xmax = 1000,
                xtick = {10, 100, 1000},
                xticklabels = {$10$, $10^2$, $10^3$},
                ymin=0.00, ymax=0.85,
                ymajorgrids,
                ytick={0,0.2,0.4,0.6,0.8,1.0},
                yticklabels={0,20,40,60,80,100},
                xlabel={ $\tau$ ($k=5$)},
                ylabel={ \% of weighted instances},
                legend pos={south east},
                width=\textwidth,
                height=5cm,
                tick align=outside,
                tick style={semithick},
                grid style={dashed,gray!70}
            ]
\addplot+[mark=none, line width=1.2pt] coordinates {
    (1.0000,0.1869)
    (1.0440,0.1869)
    (1.0440,0.1963)
    (1.1258,0.1963)
    (1.1277,0.2056)
    (1.1766,0.2056)
    (1.2437,0.2150)
    (1.2718,0.2243)
    (1.2935,0.2243)
    (1.3453,0.2336)
    (1.4291,0.2430)
    (1.5644,0.2430)
    (1.6494,0.2523)
    (1.7080,0.2617)
    (1.8215,0.2617)
    (1.8947,0.2710)
    (1.9454,0.2804)
    (1.9833,0.2897)
    (2.2099,0.2991)
    (2.3224,0.2991)
    (2.3567,0.3084)
    (2.4045,0.3178)
    (2.4052,0.3178)
    (2.4093,0.3271)
    (2.6154,0.3271)
    (2.6490,0.3364)
    (2.6548,0.3458)
    (2.6938,0.3551)
    (2.7934,0.3645)
    (2.8823,0.3645)
    (3.0599,0.3738)
    (3.1385,0.3832)
    (3.5793,0.3832)
    (3.9363,0.3925)
    (4.0979,0.4019)
    (4.7979,0.4019)
    (4.8036,0.4112)
    (4.8938,0.4112)
    (5.1657,0.4206)
    (5.4092,0.4206)
    (5.4287,0.4299)
    (5.4362,0.4299)
    (5.4909,0.4393)
    (5.9052,0.4393)
    (6.0591,0.4486)
    (6.3105,0.4579)
    (6.6281,0.4579)
    (6.9630,0.4673)
    (7.2993,0.4766)
    (8.8205,0.4766)
    (9.0377,0.4860)
    (9.3025,0.4953)
    (9.4056,0.5047)
    (9.8964,0.5140)
    (10.1804,0.5140)
    (10.2146,0.5234)
    (10.6001,0.5327)
    (11.4973,0.5421)
    (11.5108,0.5514)
    (12.1468,0.5514)
    (15.0980,0.5607)
    (16.3866,0.5701)
    (16.6532,0.5794)
    (17.4783,0.5794)
    (18.9134,0.5888)
    (19.1226,0.5981)
    (19.7606,0.6075)
    (21.5315,0.6075)
    (35.2552,0.6168)
    (39.7630,0.6262)
    (45.5495,0.6355)
    (49.6936,0.6449)
    (49.7760,0.6542)
    (63.0790,0.6542)
    (73.0702,0.6636)
    (91.1618,0.6636)
    (95.3602,0.6729)
    (97.2838,0.6729)
    (102.3565,0.6822)
    (102.6308,0.6916)
    (115.4220,0.7009)
    (116.6605,0.7103)
    (118.0136,0.7196)
    (133.2275,0.7196)
    (133.5716,0.7290)
    (156.3620,0.7290)
    (178.0815,0.7383)
    (185.7077,0.7477)
    (200.6736,0.7570)
    (204.2394,0.7664)
    (219.9668,0.7757)
    (257.4009,0.7757)
    (273.1901,0.7850)
    (310.5584,0.7850)
    (322.4711,0.7944)
    (341.2830,0.7944)
    (362.8952,0.8037)
    (431.4761,0.8037)
    (454.2858,0.8131)
    (472.0728,0.8224)
    (1014.9977,0.8224)
    (1384.5025,0.8318)
    (1962.7813,0.8411)
    (28575.7500,0.8411)
  };
  \addlegendentry{COMP}
  \addplot+[mark=none, line width=1.2pt] coordinates {
    (1.0000,0.3084)
    (1.0431,0.3084)
    (1.0440,0.3178)
    (1.0977,0.3271)
    (1.1258,0.3364)
    (1.1497,0.3364)
    (1.1539,0.3458)
    (1.4291,0.3458)
    (1.4329,0.3551)
    (1.4912,0.3645)
    (1.5531,0.3645)
    (1.5644,0.3738)
    (1.7456,0.3738)
    (1.7560,0.3832)
    (1.8215,0.3925)
    (2.2099,0.3925)
    (2.3224,0.4019)
    (2.4045,0.4019)
    (2.4052,0.4112)
    (2.4093,0.4112)
    (2.6154,0.4206)
    (3.5333,0.4206)
    (3.5793,0.4299)
    (4.4088,0.4299)
    (4.5570,0.4393)
    (4.8036,0.4393)
    (4.8938,0.4486)
    (5.4909,0.4486)
    (5.9052,0.4579)
    (6.3105,0.4579)
    (6.3503,0.4673)
    (6.6281,0.4766)
    (7.6746,0.4766)
    (7.9403,0.4860)
    (8.8205,0.4953)
    (11.5128,0.4953)
    (12.1468,0.5047)
    (16.6532,0.5047)
    (17.0117,0.5140)
    (17.4783,0.5234)
    (95.3602,0.5234)
    (97.2838,0.5327)
    (118.0136,0.5327)
    (127.6684,0.5421)
    (133.5716,0.5421)
    (137.3442,0.5514)
    (137.4734,0.5514)
    (155.6672,0.5607)
    (219.9668,0.5607)
    (222.4440,0.5701)
    (285.7113,0.5701)
    (310.5584,0.5794)
    (322.4711,0.5794)
    (331.7577,0.5888)
    (362.8952,0.5888)
    (431.4761,0.5981)
    (472.0728,0.5981)
    (705.5636,0.6075)
    (1962.7813,0.6075)
    (2047.4809,0.6168)
    (2873.6971,0.6262)
    (3049.7439,0.6262)
    (6169.9778,0.6355)
    (28575.7500,0.6355)
  };
  \addlegendentry{HYB}
  \addplot+[mark=none, line width=1.2pt] coordinates {
    (1.0000,0.4766)
    (1.0431,0.4860)
    (1.1277,0.4860)
    (1.1315,0.4953)
    (1.1497,0.5047)
    (1.1539,0.5047)
    (1.1766,0.5140)
    (1.2718,0.5140)
    (1.2935,0.5234)
    (1.4912,0.5234)
    (1.5531,0.5327)
    (1.7080,0.5327)
    (1.7456,0.5421)
    (2.7934,0.5421)
    (2.8823,0.5514)
    (3.1385,0.5514)
    (3.5333,0.5607)
    (4.0979,0.5607)
    (4.1237,0.5701)
    (4.4088,0.5794)
    (4.5570,0.5794)
    (4.7979,0.5888)
    (5.1657,0.5888)
    (5.4092,0.5981)
    (5.4287,0.5981)
    (5.4352,0.6075)
    (5.4362,0.6168)
    (7.2993,0.6168)
    (7.6746,0.6262)
    (9.8964,0.6262)
    (10.1804,0.6355)
    (11.5108,0.6355)
    (11.5128,0.6449)
    (19.7606,0.6449)
    (20.6476,0.6542)
    (21.5315,0.6636)
    (49.7760,0.6636)
    (63.0790,0.6729)
    (73.0702,0.6729)
    (76.9778,0.6822)
    (90.7957,0.6916)
    (91.1618,0.7009)
    (127.6684,0.7009)
    (133.2275,0.7103)
    (137.3442,0.7103)
    (137.4734,0.7196)
    (155.6672,0.7196)
    (156.3620,0.7290)
    (222.4440,0.7290)
    (257.4009,0.7383)
    (273.1901,0.7383)
    (285.7113,0.7477)
    (331.7577,0.7477)
    (341.2830,0.7570)
    (705.5636,0.7570)
    (811.3577,0.7664)
    (1014.9977,0.7757)
    (2873.6971,0.7757)
    (3049.7439,0.7850)
    (6169.9778,0.7850)
    (27215.0000,0.7944)
    (28575.7500,0.7944)
  };
  \addlegendentry{BP$^\star$}
  \end{axis}
\end{tikzpicture}
        \vspace{-1em}
    \end{subfigure}

    \vspace{0.7em}

    \begin{subfigure}[t]{0.48\textwidth}
        \centering
        \begin{tikzpicture}
 \begin{axis}[const plot,
  cycle list={
  {blue, dashed},
  {green!80!black, dashdotted},
  {red,solid},
  {pink, solid}},
                xmode = log,
                log basis x = 10,
                xmin = 1, xmax = 1000,
                xtick = {10, 100, 1000},
                xticklabels = {$10$, $10^2$, $10^3$},
                ymin=0.00, ymax=0.85,
                ymajorgrids,
                ytick={0,0.2,0.4,0.6,0.8,1.0},
                yticklabels={0,20,40,60,80,100},
                xlabel={ $\tau$ ($k=10$)},
                ylabel={ \% of instances},
                width=\textwidth,
                height=5cm,
                tick align=outside,
                tick style={semithick},
                grid style={dashed,gray!70}
            ]
 \addplot+[mark=none, line width=1.2pt] coordinates {
    (1.0000,0.0625)
    (1.1867,0.0625)
    (1.1977,0.0750)
    (2.0255,0.0750)
    (2.0574,0.0875)
    (2.2261,0.1000)
    (2.7425,0.1000)
    (2.8671,0.1125)
    (4.3657,0.1125)
    (4.6198,0.1250)
    (4.6486,0.1375)
    (4.9159,0.1500)
    (4.9286,0.1500)
    (5.0970,0.1625)
    (5.6899,0.1625)
    (5.8550,0.1750)
    (6.2153,0.1750)
    (6.6593,0.1875)
    (6.7614,0.2000)
    (8.1520,0.2000)
    (8.7157,0.2125)
    (10.3401,0.2125)
    (12.0930,0.2250)
    (12.4555,0.2375)
    (14.5909,0.2375)
    (17.7483,0.2500)
    (17.8941,0.2500)
    (20.1045,0.2625)
    (26.5951,0.2625)
    (28.4231,0.2750)
    (30.5639,0.2750)
    (32.8433,0.2875)
    (38.3864,0.2875)
    (42.6384,0.3000)
    (52.0443,0.3000)
    (54.8980,0.3125)
    (86.2026,0.3125)
    (88.5444,0.3250)
    (139.8211,0.3375)
    (184.7880,0.3375)
    (214.7861,0.3500)
    (218.8080,0.3625)
    (264.9498,0.3750)
    (557.8579,0.3875)
    (3289.9394,0.4000)
    (3454.4363,0.4000)
  };
  \addplot+[mark=none, line width=1.2pt] coordinates {
    (1.0000,0.2875)
    (1.1451,0.2875)
    (1.1471,0.3000)
    (1.1977,0.3000)
    (1.2862,0.3125)
    (1.5877,0.3125)
    (1.6455,0.3250)
    (1.8360,0.3250)
    (1.8682,0.3375)
    (1.8732,0.3375)
    (2.0255,0.3500)
    (2.2261,0.3500)
    (2.5757,0.3625)
    (2.8671,0.3625)
    (3.0076,0.3750)
    (4.2596,0.3750)
    (4.3657,0.3875)
    (4.9159,0.3875)
    (4.9286,0.4000)
    (5.8550,0.4000)
    (5.9400,0.4125)
    (6.2153,0.4250)
    (8.7157,0.4250)
    (9.4505,0.4375)
    (12.4691,0.4375)
    (12.5831,0.4500)
    (14.5909,0.4625)
    (17.7483,0.4625)
    (17.8941,0.4750)
    (21.8254,0.4750)
    (26.5951,0.4875)
    (28.4231,0.4875)
    (30.5639,0.5000)
    (32.8433,0.5000)
    (33.8291,0.5125)
    (42.6384,0.5125)
    (51.4602,0.5250)
    (52.0443,0.5375)
    (54.8980,0.5375)
    (62.2206,0.5500)
    (80.9847,0.5500)
    (86.2026,0.5625)
    (147.7360,0.5625)
    (184.7880,0.5750)
    (3454.4363,0.5750)
  };
  \addplot+[mark=none, line width=1.2pt] coordinates {
    (1.0000,0.5625)
    (1.0095,0.5750)
    (1.1451,0.5875)
    (1.1471,0.5875)
    (1.1867,0.6000)
    (1.2862,0.6000)
    (1.5877,0.6125)
    (1.6455,0.6125)
    (1.8360,0.6250)
    (1.8682,0.6250)
    (1.8732,0.6375)
    (2.5757,0.6375)
    (2.7425,0.6500)
    (3.0076,0.6500)
    (3.7438,0.6625)
    (4.2596,0.6750)
    (5.0970,0.6750)
    (5.6899,0.6875)
    (6.7614,0.6875)
    (8.1520,0.7000)
    (9.4505,0.7000)
    (10.3401,0.7125)
    (12.4555,0.7125)
    (12.4691,0.7250)
    (20.1045,0.7250)
    (21.8254,0.7375)
    (33.8291,0.7375)
    (37.4043,0.7500)
    (38.3864,0.7625)
    (62.2206,0.7625)
    (80.9847,0.7750)
    (139.8211,0.7750)
    (147.7360,0.7875)
    (3454.4363,0.7875)
  };
  \end{axis}
\end{tikzpicture}
        \vspace{-1em}
    \end{subfigure}
    \hfill
    \begin{subfigure}[t]{0.48\textwidth}
        \centering
        \begin{tikzpicture}
 \begin{axis}[const plot,
  cycle list={
  {blue, dashed},
  {green!80!black, dashdotted},
  {red,solid},
  {pink, solid}},
                xmode = log,
                log basis x = 10,
                xmin = 1, xmax = 1000,
                xtick = {10, 100, 1000},
                xticklabels = {$10$, $10^2$, $10^3$},
                ymin=0.00, ymax=0.85,
                ymajorgrids,
                ytick={0,0.2,0.4,0.6,0.8,1.0},
                yticklabels={0,20,40,60,80,100},
                xlabel={ $\tau$ ($k=10$)},
                ylabel={ \% of weighted instances},
                width=\textwidth,
                height=5cm,
                tick align=outside,
                tick style={semithick},
                grid style={dashed,gray!70}
            ]
 \addplot+[mark=none, line width = 1.2pt] coordinates {
    (1.0000,0.0500)
    (1.5251,0.0500)
    (1.5767,0.0625)
    (1.7245,0.0625)
    (1.9884,0.0750)
    (2.0277,0.0750)
    (2.0949,0.0875)
    (2.2730,0.0875)
    (2.4551,0.1000)
    (3.0779,0.1000)
    (3.6372,0.1125)
    (4.0888,0.1125)
    (4.1386,0.1250)
    (5.3503,0.1250)
    (5.4724,0.1375)
    (5.5712,0.1500)
    (5.5869,0.1500)
    (5.9567,0.1625)
    (6.0553,0.1625)
    (6.6193,0.1750)
    (6.8826,0.1875)
    (7.1113,0.1875)
    (8.0826,0.2000)
    (9.9580,0.2125)
    (10.9792,0.2125)
    (13.0655,0.2250)
    (15.8720,0.2250)
    (17.3674,0.2375)
    (18.3213,0.2375)
    (20.0567,0.2500)
    (21.1383,0.2625)
    (21.1546,0.2750)
    (23.1652,0.2875)
    (27.8015,0.2875)
    (30.4676,0.3000)
    (31.7440,0.3000)
    (31.8956,0.3125)
    (32.8713,0.3125)
    (37.1990,0.3250)
    (39.4467,0.3375)
    (62.7710,0.3375)
    (86.9791,0.3500)
    (108.5249,0.3500)
    (111.6601,0.3625)
    (132.2587,0.3750)
    (186.8991,0.3750)
    (216.2333,0.3875)
    (236.9186,0.3875)
    (261.3157,0.4000)
    (305.4618,0.4000)
    (639.6030,0.4125)
    (1347.1068,0.4125)
    (1367.9852,0.4250)
    (1824.4786,0.4250)
  };
  \addplot+[mark=none, line width = 1.2pt] coordinates {
    (1.0000,0.4000)
    (1.1289,0.4125)
    (1.4892,0.4125)
    (1.5251,0.4250)
    (2.0990,0.4250)
    (2.2730,0.4375)
    (4.1386,0.4375)
    (4.3124,0.4500)
    (5.5712,0.4500)
    (5.5869,0.4625)
    (6.8826,0.4625)
    (6.9972,0.4750)
    (13.0655,0.4750)
    (14.9346,0.4875)
    (15.8720,0.5000)
    (17.3674,0.5000)
    (18.3213,0.5125)
    (23.1652,0.5125)
    (27.8015,0.5250)
    (30.4676,0.5250)
    (30.7536,0.5375)
    (31.7440,0.5500)
    (31.8956,0.5500)
    (32.8713,0.5625)
    (39.4467,0.5625)
    (57.3645,0.5750)
    (102.3794,0.5750)
    (108.5249,0.5875)
    (261.3157,0.5875)
    (267.6973,0.6000)
    (639.6030,0.6000)
    (1347.1068,0.6125)
    (1824.4786,0.6125)
  };
  \addplot+[mark=none, line width = 1.2pt] coordinates {
    (1.0000,0.4500)
    (1.1289,0.4500)
    (1.2813,0.4625)
    (1.3527,0.4750)
    (1.3871,0.4875)
    (1.4400,0.5000)
    (1.4892,0.5125)
    (1.5767,0.5125)
    (1.7245,0.5250)
    (1.9884,0.5250)
    (2.0277,0.5375)
    (2.0949,0.5375)
    (2.0990,0.5500)
    (2.4551,0.5500)
    (2.6550,0.5625)
    (3.0779,0.5750)
    (3.6372,0.5750)
    (4.0888,0.5875)
    (4.3124,0.5875)
    (5.1286,0.6000)
    (5.3503,0.6125)
    (5.9567,0.6125)
    (6.0553,0.6250)
    (6.9972,0.6250)
    (7.1113,0.6375)
    (9.9580,0.6375)
    (9.9831,0.6500)
    (10.6455,0.6625)
    (10.9792,0.6750)
    (57.3645,0.6750)
    (62.7710,0.6875)
    (86.9791,0.6875)
    (102.3794,0.7000)
    (132.2587,0.7000)
    (186.8991,0.7125)
    (216.2333,0.7125)
    (221.5870,0.7250)
    (236.9186,0.7375)
    (267.6973,0.7375)
    (305.4618,0.7500)
    (1367.9852,0.7500)
    (1737.5987,0.7625)
    (1824.4786,0.7625)
  };
  \end{axis}
\end{tikzpicture}
        \vspace{-1em}
    \end{subfigure}

    \vspace{0.7em}

    \begin{subfigure}[t]{0.48\textwidth}
        \centering
        \begin{tikzpicture}
 \begin{axis}[const plot,
  cycle list={
  {blue, dashed},
  {green!80!black, dashdotted},
  {red,solid},
  {pink, solid}},
                xmode = log,
                log basis x = 10,
                xmin = 1, xmax = 1000,
                xtick = {10, 100, 1000},
                xticklabels = {$10$, $10^2$, $10^3$},
                ymin=0.00, ymax=0.85,
                ymajorgrids,
                ytick={0,0.2,0.4,0.6,0.8,1.0},
                yticklabels={0,20,40,60,80,100},
                xlabel={ $\tau$ ($k=15$)},
                ylabel={ \% of instances},
                width=\textwidth,
                height=5cm,
                tick align=outside,
                tick style={semithick},
                grid style={dashed,gray!70}
            ]
   \addplot+[mark=none, line width=1.2pt] coordinates {
    (1.0000,0.0154)
    (1.0742,0.0308)
    (1.1669,0.0308)
    (1.2900,0.0462)
    (2.8973,0.0462)
    (3.5179,0.0615)
    (3.5636,0.0615)
    (4.7059,0.0769)
    (6.3981,0.0769)
    (6.4592,0.0923)
    (6.5327,0.0923)
    (9.0175,0.1077)
    (9.1334,0.1231)
    (10.3924,0.1231)
    (11.0255,0.1385)
    (11.0737,0.1538)
    (11.3439,0.1692)
    (13.7300,0.1692)
    (15.9131,0.1846)
    (22.9957,0.1846)
    (23.0452,0.2000)
    (33.7945,0.2000)
    (34.7656,0.2154)
    (59.7357,0.2154)
    (89.3634,0.2308)
    (156.7895,0.2462)
    (158.8168,0.2615)
    (176.5662,0.2769)
    (256.3028,0.2769)
    (262.8137,0.2923)
    (270.2555,0.3077)
    (318.0624,0.3231)
    (423.9607,0.3231)
    (799.8141,0.3385)
    (2305.7361,0.3538)
    (4570.3407,0.3692)
    (4645.4717,0.3846)
    (4944.3457,0.4000)
    (6612.0000,0.4000)
    (9988.6362,0.4154)
    (12066.7956,0.4308)
    (12670.1354,0.4308)
  };
  \addplot+[mark=none, line width=1.2pt] coordinates {
    (1.0000,0.3231)
    (2.1937,0.3231)
    (2.2108,0.3385)
    (2.3844,0.3538)
    (2.5066,0.3692)
    (2.5245,0.3692)
    (2.8973,0.3846)
    (3.5179,0.3846)
    (3.5636,0.4000)
    (4.7059,0.4000)
    (4.7077,0.4154)
    (5.2009,0.4308)
    (5.8364,0.4462)
    (5.9364,0.4615)
    (6.1045,0.4769)
    (6.4592,0.4769)
    (6.5327,0.4923)
    (9.1334,0.4923)
    (10.3924,0.5077)
    (11.5004,0.5077)
    (13.7300,0.5231)
    (15.9131,0.5231)
    (22.9957,0.5385)
    (23.0452,0.5385)
    (24.1600,0.5538)
    (25.0737,0.5692)
    (26.4215,0.5692)
    (27.9032,0.5846)
    (29.0308,0.6000)
    (33.7515,0.6000)
    (33.7945,0.6154)
    (34.7656,0.6154)
    (35.6650,0.6308)
    (36.9638,0.6462)
    (176.5662,0.6462)
    (256.3028,0.6615)
    (318.0624,0.6615)
    (423.9607,0.6769)
    (12670.1354,0.6769)
  };
  \addplot+[mark=none, line width=1.2pt] coordinates {
    (1.0000,0.6000)
    (1.0742,0.6000)
    (1.1669,0.6154)
    (1.2900,0.6154)
    (1.3054,0.6308)
    (2.1937,0.6462)
    (2.5066,0.6462)
    (2.5245,0.6615)
    (6.1045,0.6615)
    (6.3981,0.6769)
    (11.3439,0.6769)
    (11.5004,0.6923)
    (25.0737,0.6923)
    (26.4215,0.7077)
    (29.0308,0.7077)
    (31.3221,0.7231)
    (33.7515,0.7385)
    (36.9638,0.7385)
    (59.7357,0.7538)
    (4944.3457,0.7538)
    (6612.0000,0.7692)
    (12670.1354,0.7692)
  };
  \end{axis}
\end{tikzpicture}
        \vspace{-1em}
    \end{subfigure}
    \hfill
    \begin{subfigure}[t]{0.48\textwidth}
        \centering
        \begin{tikzpicture}
 \begin{axis}[const plot,
  cycle list={
  {blue, dashed},
  {green!80!black, dashdotted},
  {red,solid},
  {pink, solid}},
                xmode = log,
                log basis x = 10,
                xmin = 1, xmax = 1000,
                xtick = {10, 100, 1000},
                xticklabels = {$10$, $10^2$, $10^3$},
                ymin=0.00, ymax=0.85,
                ymajorgrids,
                ytick={0,0.2,0.4,0.6,0.8,1.0},
                yticklabels={0,20,40,60,80,100},
                xlabel={ $\tau$ ($k=15$)},
                ylabel={ \% of weighted instances},
                width=\textwidth,
                height=5cm,
                tick align=outside,
                tick style={semithick},
                grid style={dashed,gray!70}
            ]
  \addplot+[mark=none, line width = 1.2pt] coordinates {
    (1.0000,0.0308)
    (1.1405,0.0462)
    (1.7506,0.0462)
    (1.8070,0.0615)
    (2.4809,0.0615)
    (2.6872,0.0769)
    (2.7580,0.0769)
    (2.8729,0.0923)
    (3.1299,0.0923)
    (3.2084,0.1077)
    (6.2374,0.1077)
    (8.7197,0.1231)
    (10.5840,0.1231)
    (22.0363,0.1385)
    (32.8382,0.1385)
    (34.5872,0.1538)
    (35.0070,0.1692)
    (40.1010,0.1846)
    (53.7204,0.2000)
    (62.7914,0.2154)
    (99.1811,0.2154)
    (102.4474,0.2308)
    (108.7748,0.2462)
    (115.0896,0.2615)
    (116.2426,0.2769)
    (118.3864,0.2923)
    (125.8371,0.2923)
    (158.4490,0.3077)
    (174.0919,0.3231)
    (182.1680,0.3385)
    (239.5450,0.3385)
    (328.6079,0.3538)
    (330.6965,0.3538)
    (357.9282,0.3692)
    (598.5584,0.3692)
    (672.0688,0.3846)
    (2846.2967,0.3846)
    (4256.3936,0.4000)
    (4403.7769,0.4154)
    (8063.9390,0.4308)
    (35766.0450,0.4308)
  };
  \addplot+[mark=none, line width = 1.2pt] coordinates {
    (1.0000,0.4462)
    (1.1405,0.4462)
    (1.3371,0.4615)
    (1.4865,0.4615)
    (1.5083,0.4769)
    (1.6014,0.4769)
    (1.7207,0.4923)
    (1.7506,0.5077)
    (1.8070,0.5077)
    (2.2770,0.5231)
    (2.4132,0.5385)
    (2.4809,0.5538)
    (2.8729,0.5538)
    (3.1299,0.5692)
    (3.2084,0.5692)
    (3.7864,0.5846)
    (4.1776,0.6000)
    (4.4164,0.6154)
    (4.8989,0.6154)
    (5.2263,0.6308)
    (8.7197,0.6308)
    (9.2242,0.6462)
    (81.1166,0.6462)
    (99.1811,0.6615)
    (372.8373,0.6615)
    (598.5584,0.6769)
    (672.0688,0.6769)
    (2846.2967,0.6923)
    (35766.0450,0.6923)
  };
  \addplot+[mark=none, line width = 1.2pt] coordinates {
    (1.0000,0.4615)
    (1.3371,0.4615)
    (1.4865,0.4769)
    (1.5083,0.4769)
    (1.6014,0.4923)
    (2.6872,0.4923)
    (2.7580,0.5077)
    (4.4164,0.5077)
    (4.6230,0.5231)
    (4.8111,0.5385)
    (4.8989,0.5538)
    (5.2263,0.5538)
    (5.6448,0.5692)
    (6.2374,0.5846)
    (9.2242,0.5846)
    (9.8965,0.6000)
    (10.5840,0.6154)
    (22.0363,0.6154)
    (22.7868,0.6308)
    (23.5486,0.6462)
    (32.8382,0.6615)
    (62.7914,0.6615)
    (81.1166,0.6769)
    (118.3864,0.6769)
    (125.8371,0.6923)
    (182.1680,0.6923)
    (239.5450,0.7077)
    (328.6079,0.7077)
    (330.6965,0.7231)
    (357.9282,0.7231)
    (372.8373,0.7385)
    (8063.9390,0.7385)
    (22972.7000,0.7538)
    (34062.9000,0.7692)
    (35766.0450,0.7692)
  };
  \end{axis}
\end{tikzpicture}
        \vspace{-1em}
    \end{subfigure}

    \vspace{0.7em}

    \begin{subfigure}[t]{0.48\textwidth}
        \centering
        \begin{tikzpicture}
 \begin{axis}[const plot,
  cycle list={
  {blue, dashed},
  {green!80!black, dashdotted},
  {red,solid},
  {pink, solid}},
                xmode = log,
                log basis x = 10,
                xmin = 1, xmax = 1000,
                xtick = {10, 100, 1000},
                xticklabels = {$10$, $10^2$, $10^3$},
                ymin=0.00, ymax=0.85,
                ymajorgrids,
                ytick={0,0.2,0.4,0.6,0.8,1.0},
                yticklabels={0,20,40,60,80,100},
                xlabel={ $\tau$ ($k=20$)},
                ylabel={ \% of instances},
                width=\textwidth,
                height=5cm,
                tick align=outside,
                tick style={semithick},
                grid style={dashed,gray!70}
            ]
   \addplot+[mark=none, line width=1.2pt] coordinates {
    (1.0000,0.0192)
    (1.2407,0.0385)
    (1.7042,0.0385)
    (1.9162,0.0577)
    (2.0444,0.0769)
    (2.1401,0.0769)
    (2.4883,0.0962)
    (25.6483,0.0962)
    (26.4518,0.1154)
    (95.7478,0.1154)
    (105.7809,0.1346)
    (125.7441,0.1538)
    (145.9650,0.1731)
    (193.9794,0.1923)
    (250.7633,0.2115)
    (671.8985,0.2115)
    (859.8274,0.2308)
    (968.5870,0.2500)
    (1020.0760,0.2500)
    (1303.7618,0.2692)
    (1474.5209,0.2885)
    (2494.3227,0.3077)
    (4705.3150,0.3269)
    (4940.5808,0.3269)
  };
  \addplot+[mark=none, line width=1.2pt] coordinates {
    (1.0000,0.4615)
    (1.2407,0.4615)
    (1.2549,0.4808)
    (1.3889,0.5000)
    (2.0444,0.5000)
    (2.0646,0.5192)
    (2.1401,0.5385)
    (7.0201,0.5385)
    (7.2744,0.5577)
    (13.3030,0.5577)
    (13.4668,0.5769)
    (13.6630,0.5769)
    (16.4922,0.5962)
    (24.1781,0.5962)
    (25.2763,0.6154)
    (40.3207,0.6154)
    (41.6219,0.6346)
    (42.7421,0.6538)
    (71.9681,0.6731)
    (85.2343,0.6923)
    (95.6492,0.6923)
    (95.7478,0.7115)
    (250.7633,0.7115)
    (671.8985,0.7308)
    (4940.5808,0.7308)
  };
  \addplot+[mark=none, line width=1.2pt] coordinates {
    (1.0000,0.5577)
    (1.3889,0.5577)
    (1.7042,0.5769)
    (2.4883,0.5769)
    (2.5526,0.5962)
    (2.9606,0.6154)
    (3.2582,0.6346)
    (7.0201,0.6538)
    (7.2744,0.6538)
    (13.3030,0.6731)
    (13.4668,0.6731)
    (13.6630,0.6923)
    (16.4922,0.6923)
    (16.6698,0.7115)
    (24.1781,0.7308)
    (25.2763,0.7308)
    (25.6483,0.7500)
    (26.4518,0.7500)
    (32.8086,0.7692)
    (40.3207,0.7885)
    (85.2343,0.7885)
    (88.9553,0.8077)
    (95.6492,0.8269)
    (968.5870,0.8269)
    (1020.0760,0.8462)
    (4940.5808,0.8462)
  };
  \end{axis}
\end{tikzpicture}
        \vspace{-1em}
    \end{subfigure}
    \hfill
    \begin{subfigure}[t]{0.48\textwidth}
        \centering
        \begin{tikzpicture}
 \begin{axis}[const plot,
  cycle list={
  {blue, dashed},
  {green!80!black, dashdotted},
  {red,solid},
  {pink, solid}},
                xmode = log,
                log basis x = 10,
                xmin = 1, xmax = 1000,
                xtick = {10, 100, 1000},
                xticklabels = {$10$, $10^2$, $10^3$},
                ymin=0.00, ymax=0.85,
                ymajorgrids,
                ytick={0,0.2,0.4,0.6,0.8,1.0},
                yticklabels={0,20,40,60,80,100},
                xlabel={ $\tau$ ($k=20$)},
                ylabel={ \% of weighted instances},
                width=\textwidth,
                height=5cm,
                tick align=outside,
                tick style={semithick},
                grid style={dashed,gray!70}
            ]
  \addplot+[mark=none, line width=1.2pt] coordinates {
    (1.0000,0.0000)
    (1.1099,0.0192)
    (1.3040,0.0192)
    (1.4745,0.0385)
    (1.5752,0.0385)
    (1.7835,0.0577)
    (3.9803,0.0577)
    (4.2573,0.0769)
    (6.4826,0.0769)
    (6.5722,0.0962)
    (70.0689,0.0962)
    (90.6284,0.1154)
    (103.9594,0.1346)
    (128.3948,0.1538)
    (190.3462,0.1538)
    (229.2843,0.1731)
    (244.2913,0.1923)
    (289.4041,0.1923)
    (358.5663,0.2115)
    (370.0695,0.2308)
    (385.9903,0.2500)
    (431.0781,0.2500)
    (507.3267,0.2692)
    (805.1536,0.2885)
    (848.1569,0.3077)
    (3005.2399,0.3269)
    (3963.8500,0.3269)
    (8115.5606,0.3462)
    (8521.3386,0.3462)
  };
  \addplot+[mark=none, line width=1.2pt] coordinates {
    (1.0000,0.5385)
    (1.1395,0.5385)
    (1.3040,0.5577)
    (1.4745,0.5577)
    (1.5752,0.5769)
    (1.7835,0.5769)
    (2.3153,0.5962)
    (2.9221,0.5962)
    (3.3871,0.6154)
    (3.4699,0.6154)
    (3.5232,0.6346)
    (3.9545,0.6346)
    (3.9803,0.6538)
    (6.5722,0.6538)
    (7.7515,0.6731)
    (14.3229,0.6923)
    (29.4595,0.6923)
    (40.2714,0.7115)
    (41.1486,0.7308)
    (49.3756,0.7308)
    (55.2078,0.7500)
    (8521.3386,0.7500)
  };
  \addplot+[mark=none, line width=1.2pt] coordinates {
    (1.0000,0.4808)
    (1.1099,0.4808)
    (1.1395,0.5000)
    (2.3153,0.5000)
    (2.6316,0.5192)
    (2.9221,0.5385)
    (3.3871,0.5385)
    (3.4699,0.5577)
    (3.5232,0.5577)
    (3.9545,0.5769)
    (4.2573,0.5769)
    (5.6472,0.5962)
    (6.4826,0.6154)
    (14.3229,0.6154)
    (24.6037,0.6346)
    (25.0763,0.6538)
    (29.4595,0.6731)
    (41.1486,0.6731)
    (43.5118,0.6923)
    (46.4218,0.7115)
    (49.3756,0.7308)
    (55.2078,0.7308)
    (70.0689,0.7500)
    (128.3948,0.7500)
    (190.3462,0.7692)
    (244.2913,0.7692)
    (289.4041,0.7885)
    (385.9903,0.7885)
    (431.0781,0.8077)
    (3005.2399,0.8077)
    (3963.8500,0.8269)
    (8521.3386,0.8269)
  };
  \end{axis}
\end{tikzpicture}
        \vspace{-1em}
    \end{subfigure}

    \medskip

    \caption{Performance profiles of algorithms COMP, HYB and BP$^\star$ on the full set of unweighted and weighted benchmark instances, with a time limit of one hour.}
    \label{fig:perf-all}
\end{figure}

Table \ref{tab:detailed_comparison_results} reports a detailed performance comparison of algorithms {COMP}, {HYB} and {BP$^\star$} on the full benchmark set, separately for unweighted and weighted instances. For each family of instances and each value of $k$, the table lists the number of instances considered (\#inst) and, for each algorithm, the number of instances solved to proven optimality (\#opt), the average running time in seconds on solved instances (time), the average optimality gap at the time limit, computed over unsolved instances only (gap), and the average number of explored branch-and-bound nodes (nodes) on solved instances. To quantify how far an algorithm is from proving optimality when the time limit expires, we measure an \emph{optimality gap} based on its best primal and dual information: let $z^{\mathrm{UB}}$ denote the best incumbent value found by an algorithm at time limit and $z^{\mathrm{LB}}$ the best dual bound available at that time. For each instance that is not closed to proven optimality within the time limit, we define the {optimality gap} as $100 \cdot \frac{z^{\mathrm{UB}} - z^{\mathrm{LB}}}{z^{\mathrm{UB}}}$, which is therefore expressed as a percentage.
If an algorithm fails to produce a valid primal or dual bound before the time limit, we conservatively set the optimality gap to $100\%$, corresponding to the worst possible case.

Beyond the number of instances solved, already discussed in Section \ref{sec:comparison_exact}, {BP$^\star$} also tends to explore substantially fewer branch-and-bound nodes on the instances that it solves. This is consistent with the tighter LP relaxation bound provided by the \zigbases formulation, which leads to stronger pruning in the search tree. In many families and values of $k$, {BP$^\star$} also has smaller average optimality gaps on the unsolved instances. These features are visible in both the unweighted and weighted sections of the table.
The distribution of the results across families reflects the complementary strengths of the algorithms. For {\tt Intersection} instances, {BP$^\star$} clearly dominates: it solves almost all unweighted and weighted cases, often with fewer nodes and smaller gaps than {HYB} and {COMP}. For {Partitioning} instances, the picture is more mixed, with {BP$^\star$}, {HYB}, and {COMP} each taking the lead in different $(\text{family},k)$ combinations, but {BP$^\star$} still maintaining competitive performance in terms of number of solved instances. Finally, on {\tt Coloring} instances, {BP$^\star$} remains competitive but less dominant: {HYB} and {COMP} solve more instances in several~$k$ configurations, often closing them in a smaller amount of time as well.

\begin{table}[h!]
\centering
\small
\renewcommand
\arraystretch{1.5}
\tabcolsep=2.5pt
\caption{Performance comparison of algorithms COMP, HYB and BP$^\star$ on the full set {of} benchmark instances, with a time limit of one hour.}
\label{tab:detailed_comparison_results}
\begin{tabular}{lrrrrrrrrrrrrrrrrr}
\hline
                      &     &        &  & \multicolumn{4}{c}{\textbf{BP$^\star$}} &  & \multicolumn{4}{c}{\textbf{HYB}}   &  & \multicolumn{4}{c}{\textbf{COMP}}        \\ \cline{5-8} \cline{10-13} \cline{15-18} 
Family                & $k$ & \#inst &  & \#opt         & time   & gap    & nodes &  & \#opt       & time  & gap  & nodes &  & \#opt       & time  & gap & nodes       \\ \cline{1-3} \cline{5-8} \cline{10-13} \cline{15-18} 
\\[-2ex]
\multicolumn{18}{c}{\textbf{\normalsize Unweighted Instances}}                                                                                                                      \\[1.5ex]
\texttt{Coloring}     & 5   & 42     &  & 23            & 54.8   & 42.3   & 20.3  &  & 26          & 131.0 & 70.6 & 45.0  &  & \textbf{32} & 250.8 & 61.9 & 19{,}597.8  \\
                      & 10  & 32     &  & 20            & 222.5  & 35.6   & 8.8   &  & \textbf{22} & 65.4  & 52.0 & 15.9  &  & 11          & 358.6 & 63.2 & 30{,}866.1  \\
                      & 15  & 29     &  & 19            & 71.5   & 41.4   & 3.9   &  & \textbf{25} & 78.3  & 68.5 & 14.1  &  & 11          & 302.1 & 64.1 & 30{,}689.5  \\
                      & 20  & 28     &  & 21            & 172.3  & 37.4   & 39.9  &  & \textbf{26} & 125.1 & 61.7 & 22.4  &  & 10          & 874.9 & 68.0 & 91{,}338.9  \\[1.5ex]
\texttt{Partitioning} & 5   & 9      &  & \textbf{8}    & 273.6  & 23.2   & 5.6   &  & \textbf{8}  & 40.2  & 72.8 & 11.0  &  & \textbf{8}  & 27.0  & 29.3 & 2{,}627.4   \\
                      & 10  & 9      &  & 5             & 19.6   & 59.0   & 10.2  &  & \textbf{8}  & 228.7 & 83.5 & 12.1  &  & 7           & 274.7 & 60.8 & 2{,}414.6   \\
                      & 15  & 9      &  & 6             & 39.9   & 70.6   & 4.0   &  & \textbf{7}  & 37.4  & 83.3 & 6.0   &  & 6           & 261.5 & 56.2 & 15{,}589.7  \\
                      & 20  & 8      &  & \textbf{7}    & 33.6   & 100.0  & 12.4  &  & 6           & 225.8 & 61.0 & 19.3  &  & 2           & 9.7   & 44.9 & 1{,}176.0   \\[1.5ex]
\texttt{Intersection} & 5   & 56     &  & \textbf{54}   & 91.4   & 57.9   & 78.9  &  & 33          & 207.3 & 54.9 & 145.0 &  & 48          & 201.2 & 54.0 & 38{,}283.7  \\
                      & 10  & 39     &  & \textbf{38}   & 146.2  & 15.7   & 113.8 &  & 16          & 208.4 & 51.4 & 507.2 &  & 14          & 18.0  & 62.2 & 5{,}716.1   \\
                      & 15  & 27     &  & \textbf{25}   & 34.0   & 18.4   & 16.4  &  & 12          & 97.6  & 50.0 & 61.1  &  & 11          & 133.5 & 70.2 & 14{,}642.1  \\
                      & 20  & 16     &  & \textbf{16}   & 28.5   & -      & 25.8  &  & 6           & 18.6  & 48.4 & 63.5  &  & 5           & 10.7  & 65.0 & 1{,}091.2   \\[2ex] \cline{1-3} \cline{5-8} \cline{10-13} \cline{15-18} 
Tot               &     & 304    &  & \textbf{242}  &        &        &       &  & 195         &       &      &       &  & 165         &       &      &             \\ \hline
\\[-2ex]
\multicolumn{18}{c}{\textbf{\normalsize Weighted Instances}}                                                                                                                        \\[1.5ex] 
\texttt{Coloring}     & 5   & 42     &  & 24            & 336.3  & 45.1   & 29.4  &  & 25          & 78.3  & 62.2 & 42.3  &  & \textbf{33} & 287.0 & 66.2 & 20{,}756.3  \\
                      & 10  & 32     &  & 18            & 51.0   & 42.0   & 17.3  &  & \textbf{22} & 115.7 & 50.7 & 31.8  &  & 14          & 215.3 & 59.2 & 29{,}885.9  \\
                      & 15  & 29     &  & 20            & 368.3  & 52.6   & 44.0  &  & \textbf{25} & 191.9 & 61.0 & 24.3  &  & 13          & 671.4 & 62.0 & 62{,}914.5  \\
                      & 20  & 28     &  & 20            & 331.7  & 40.0   & 38.4  &  & \textbf{26} & 228.1 & 62.9 & 36.0  &  & 10          & 712.6 & 53.2 & 17{,}182.6  \\[1.5ex]
\texttt{Partitioning} & 5   & 9      &  & 7             & 35.9   & 100.0  & 38.0  &  & 8           & 14.7  & 69.6 & 2.4   &  & \textbf{9}  & 376.4 & -    & 33{,}411.7  \\
                      & 10  & 9      &  & 7             & 510.1  & 60.5   & 9.1   &  & \textbf{8}  & 178.2 & 81.5 & 16.9  &  & 6           & 34.1  & 38.2 & 1{,}291.7   \\
                      & 15  & 9      &  & 6             & 232.9  & 69.6   & 11.7  &  & \textbf{7}  & 43.3  & 79.3 & 13.6  &  & 5           & 216.3 & 43.1 & 21{,}711.4  \\
                      & 20  & 8      &  & \textbf{7}    & 316.8  & 100.0  & 18.6  &  & 6           & 262.1 & 58.0 & 21.5  &  & 3           & 865.2 & 47.8 & 116{,}003.7 \\[1.5ex]
\texttt{Intersection} & 5   & 56     &  & \textbf{54}   & 119.7  & 59.2   & 30.4  &  & 35          & 210.5 & 53.1 & 160.1 &  & 48          & 185.0 & 39.6 & 25{,}108.7  \\
                      & 10  & 39     &  & \textbf{36}   & 26.3   & 43.0   & 25.3  &  & 19          & 230.4 & 51.2 & 206.6 &  & 14          & 26.5  & 57.2 & 8{,}957.6   \\
                      & 15  & 27     &  & \textbf{24}   & 57.1   & 14.4   & 9.2   &  & 13          & 35.7  & 40.9 & 84.2  &  & 10          & 11.3  & 60.2 & 5{,}592.8   \\
                      & 20  & 16     &  & \textbf{16}   & 77.6   & -      & 24.4  &  & 7           & 73.0  & 44.4 & 45.3  &  & 5           & 457.2 & 59.5 & 92{,}402.4  \\[2ex] \cline{1-3} \cline{5-8} \cline{10-13} \cline{15-18} 
Tot               &     & 304    &  & \textbf{239}  &        &        &       &  & 201         &       &      &       &  & 170         &       &      &             \\ \hline
\end{tabular}
\end{table}

In Tables \ref{tab:unweighted_best_known_coloring_partitioning}-\ref{tab:weighted_best_known_intersection} we report, for each instance family, the best known solution values for all values of $k$ for which the instances are defined. For each instance, the tables list the number of vertices $|\vertices|$, the number of edges $|\edges|$, the stability number $\alpha(G)$, and the best known objective value for each considered value of $k \in \{5,10,15,20\}$. These values are either proven optimal or represent the best solutions found in the literature or by our algorithm BP$^\star$.
Throughout these tables, bold values indicate proven optimal solution values. Values annotated with a red star (\textcolor{red}{$\star$}) are new optimal solution values obtained by our algorithm {BP$^\star$} within a time limit of one hour on our machines. Values marked with a dagger superscript ({$\dag$}) are optimal solution values that are attained by at least one among the algorithms {HYB} or {COMP}, but not by {BP$^\star$} in our experiments. Finally, values with a double dagger superscript ({$\ddag$}) are optimal solution values already known in the literature, but not matched by any of the algorithms considered in our computational study. Finally, instances denoted by a dot (.) are trivial instances solved in preprocessing.
For unweighted instances only, we also performed an additional campaign of long runs with a time limit of three days. These runs were restricted to a subset of instances that were not solved to optimality by any algorithm in the literature, including those tested in \citet{FurLMP2020}. The goal of this experiment was to assess how many of these open instances could be closed when allowing for a substantially larger computational effort. The instances that we were able to close in this setting are annotated with the red double star (\textcolor{red}{${\star\star}$}) in the tables.

The comparison with \citet{FurLMP2020} must be interpreted with care, because the computational results therein were obtained on more powerful machines than those available for our experiments. As a consequence, {HYB} and {COMP} solve a larger number of instances in \citet{FurLMP2020} than what is reflected by our own runs of these algorithms. All counts reported henceforth, including the classification into closed and open instances, are therefore based both on executions carried out on our hardware and on results from the literature.
Overall, 19 unweighted instances and 22 weighted instances are still open. {BP$^\star$} closes 33 unweighted instances that were previously open in the literature when using a time limit of one hour on our machines, and closes 6 additional unweighted instances when the time limit is extended to three days. For weighted instances, {BP$^\star$} closes 40 instances that were previously open in the literature. Most of these new optimal solutions concern the {\tt Intersection} family: 30 unweighted and 32 weighted instances from this family are closed for the first time by {BP$^\star$}.
On the other hand, there remain several instances that {BP$^\star$} does not solve to optimality within the imposed time limits, while other algorithms do. For unweighted instances, there are 37 such cases, either solved by {HYB} or {COMP} in \citet{FurLMP2020} or by our own runs of these algorithms, but not by {BP$^\star$}. For weighted instances, this situation occurs 43 times. The majority of these missed instances belong to the {\tt Coloring} family (30 in the unweighted case and 33 in the weighted case), which indicates that the approach underlying {BP$^\star$} has more difficulty with {\tt Coloring} instances, while being particularly effective on {\tt Intersection} instances. This behavior is in contrast with {HYB} and {COMP}, which tend to perform better on {\tt Coloring} and worse on {\tt Intersection}.

\begin{table}[h!]
\centering
\small
\renewcommand\arraystretch{1.7}
\tabcolsep=2.35pt
\caption{Best known values for unweighted instances from the \texttt{Coloring} and \texttt{Partitioning} families.}
\label{tab:unweighted_best_known_coloring_partitioning}
\begin{tabular}{lrrrlccccllrrrlcccc}
\hline
                    &       &       &             &  & \multicolumn{4}{c}{$k$}                                                                                                                  &  &                 &       &       &             &  & \multicolumn{4}{c}{$k$}                                                                                                                   \\ \cline{6-9} \cline{16-19} 
Instance            & $|\vertices|$ & $|\edges|$ & $\alpha(G)$ &  & 5                 & 10                                             & 15                                             & 20                 &  & Instance        & $|\vertices|$ & $|\edges|$ & $\alpha(G)$ &  & 5                                              & 10                                             & 15                 & 20                 \\ \cline{1-4} \cline{6-9} \cline{11-14} \cline{16-19} 
\tt 1-FullIns\_3    & 30    & 100   & 14          &  & \textbf{7}        & \textbf{11}                                    &                                                &                    &  & \tt mug88\_1    & 88    & 146   & 29          &  & \textbf{4}                                     & \textbf{9}                                     & \textbf{15}        & \textbf{20}        \\
\tt 1-FullIns\_4    & 93    & 593   & 45          &  & \textbf{9}$^\dag$ & \textbf{13}                                    & \textbf{18}$^\dag$                             & \textbf{22}        &  & \tt mug88\_25   & 88    & 146   & 29          &  & \textbf{4}                                     & \textbf{9}                                     & \textbf{14}        & \textbf{19}        \\
\tt 1-Insertions\_4 & 67    & 232   & 32          &  & \textbf{7}        & \textbf{12}                                    & \textbf{16}                                    & \textbf{22}        &  & \tt mulsol.i.2  & 188   & 3885  & 90          &  & .                                              & .                                              & .                  & \textbf{18}        \\
\tt 2-FullIns\_3    & 52    & 201   & 25          &  & \textbf{8}        & \textbf{13}                                    & \textbf{17}                                    & \textbf{23}        &  & \tt mulsol.i.3  & 184   & 3916  & 86          &  & .                                              & .                                              & \textbf{18}        & \textbf{19}        \\
\tt 2-Insertions\_3 & 37    & 72    & 18          &  & \textbf{6}        & \textbf{10}                                    & \textbf{16}                                    &                    &  & \tt mulsol.i.4  & 185   & 3946  & 86          &  & .                                              & .                                              & \textbf{18}        & \textbf{19}        \\
\tt 2-Insertions\_4 & 149   & 541   & 74          &  & \textbf{7}$^\dag$ & \textbf{11}                                    & \textbf{17}$^\ddag$ & \textbf{22}$^\dag$ &  & \tt mulsol.i.5  & 186   & 3973  & 88          &  & .                                              & .                                              & \textbf{18}        & \textbf{19}        \\
\tt 3-FullIns\_3    & 80    & 346   & 37          &  & \textbf{9}$^\dag$ & \textbf{14}$^\dag$                             & \textbf{17}                                    & \textbf{25}        &  & \tt myciel3     & 11    & 20    & 5           &  & .                                              &                                                &                    &                    \\
\tt 3-Insertions\_3 & 56    & 110   & 27          &  & \textbf{6}        & \textbf{11}                                    & \textbf{16}                                    & \textbf{21}        &  & \tt myciel4     & 23    & 71    & 11          &  & \textbf{7}                                     & \textbf{12}                                    &                    &                    \\
\tt 4-FullIns\_3    & 114   & 541   & 55          &  & \textbf{9}$^\dag$ & \textbf{15}$^\dag$                             & \textbf{18}$^\dag$                             & \textbf{23}$^\dag$ &  & \tt myciel5     & 47    & 236   & 23          &  & \textbf{8}                                     & \textbf{13}                                    & \textbf{18}        & \textbf{23}        \\
\tt 4-Insertions\_3 & 79    & 156   & 39          &  & \textbf{6}        & \textbf{11}                                    & \textbf{16}                                    & \textbf{21}        &  & \tt myciel6     & 95    & 755   & 47          &  & \textbf{9}$^\dag$                              & \textbf{14}$^\dag$                             & \textbf{19}$^\dag$ & \textbf{24}$^\dag$ \\
\tt 5-FullIns\_3    & 154   & 792   & 72          &  & \textbf{9}$^\dag$ & \textbf{15}$^\ddag$ & \textbf{19}$^\dag$                             & \textbf{22}$^\dag$ &  & \tt myciel7     & 191   & 2360  & 95          &  & \textbf{10}$^\dag$                             & \textbf{15}$^\ddag$ & \textbf{20}$^\dag$ & \textbf{25}$^\dag$ \\
\tt anna            & 138   & 493   & 80          &  & \textbf{1}        & \textbf{1}$^\dag$                              & \textbf{2}$^\dag$                              & \textbf{2}$^\dag$  &  & \tt queen10\_10 & 100   & 1470  & 10          &  & 70                                             & \textbf{90}                                    &                    &                    \\
\tt david           & 87    & 406   & 36          &  & .                 & .                                              & \textbf{4}                                     & \textbf{9}         &  & \tt queen11\_11 & 121   & 1980  & 11          &  & 82                                             & 108                                            &                    &                    \\
\tt DSJC125.1       & 125   & 736   & 34          &  & 21                & 42                                             & 60                                             & 72                 &  & \tt queen12\_12 & 144   & 2596  & 12          &  & 92                                             & 130                                            &                    &                    \\
\tt DSJC125.5       & 125   & 3891  & 10          &  & 102               & \textbf{115}                                   &                                                &                    &  & \tt queen13\_13 & 169   & 3328  & 13          &  & 105                                            & 151                                            &                    &                    \\
\tt games120        & 120   & 638   & 22          &  & 19                & \textbf{39}\textcolor{red}{$^{\star\star}$}                                             & \textbf{55}\textcolor{red}{$^{\star\star}$}                                             & \textbf{67}        &  & \tt queen14\_14 & 196   & 4186  & 14          &  & 117                                            & 174                                            &                    &                    \\
\tt huck            & 74    & 301   & 27          &  & \textbf{1}        & \textbf{3}                                     & \textbf{6}                                     & \textbf{9}         &  & \tt queen5\_5   & 25    & 160   & 5           &  & \textbf{20}                                    &                                                &                    &                    \\
\tt jean            & 80    & 254   & 38          &  & \textbf{1}        & \textbf{1}                                     & \textbf{2}                                     & \textbf{4}         &  & \tt queen6\_6   & 36    & 290   & 6           &  & \textbf{28}                                    &                                                &                    &                    \\
\tt miles1000       & 128   & 3216  & 8           &  & \textbf{53}       &                                                &                                                &                    &  & \tt queen7\_7   & 49    & 476   & 7           &  & \textbf{38}$^\dag$                             &                                                &                    &                    \\
\tt miles1500       & 128   & 5198  & 5           &  & \textbf{115}      &                                                &                                                &                    &  & \tt queen8\_12  & 96    & 1368  & 8           &  & 67                                             &                                                &                    &                    \\
\tt miles250        & 128   & 387   & 44          &  & .                 & .                                              & \textbf{4}$^\dag$                              & \textbf{11}        &  & \tt queen8\_8   & 64    & 728   & 8           &  & \textbf{48}$^\dag$                             &                                                &                    &                    \\
\tt miles500        & 128   & 1170  & 18          &  & \textbf{7}        & \textbf{25}\textcolor{red}{$^\star$}           & \textbf{55}\textcolor{red}{$^\star$}           &                    &  & \tt queen9\_9   & 81    & 1056  & 9           &  & \textbf{59}$^\ddag$ &                                                &                    &                    \\
\tt miles750        & 128   & 2113  & 12          &  & \textbf{20}       & \textbf{75}                                    &                                                &                    &  & \tt r125.1      & 125   & 209   & 49          &  & .                                              & .                                              & \textbf{1}         & \textbf{5}         \\
\tt mug100\_1       & 100   & 166   & 33          &  & \textbf{5}        & \textbf{10}                                    & \textbf{15}                                    & \textbf{20}        &  & \tt r125.1c     & 125   & 7501  & 7           &  & \textbf{116}                                   &                                                &                    &                    \\
\tt mug100\_25      & 100   & 166   & 33          &  & \textbf{5}        & \textbf{10}                                    & \textbf{15}                                    & \textbf{20}        &  & \tt r125.5      & 125   & 3838  & 5           &  & \textbf{91}                                    &                                                &                    &                    \\ \hline
\tt adjnoun        & 112   & 425   & 53          &  & \textbf{2} & \textbf{6}$^\dag$ & \textbf{11}       & \textbf{16}       &  & \tt jazz          & 198   & 2742  & 40          &  & \textbf{4} & \textbf{12}$^\dag$ & \textbf{25}$^\ddag$ & \textbf{44}\textcolor{red}{$^\star$} \\
\tt celegansneural & 297   & 2148  & 110         &  & \textbf{1} & \textbf{1}$^\dag$ & \textbf{2}$^\dag$ & \textbf{6}$^\dag$ &  & \tt karate        & 34    & 78    & 20          &  & \textbf{2} & \textbf{4}         & \textbf{6}                                     & \textbf{11}                          \\
\tt chesapeake     & 39    & 170   & 17          &  & \textbf{7} & \textbf{12}       & \textbf{17}       &                   &  & \tt lesmis        & 77    & 254   & 35          &  & \textbf{1} & \textbf{2}         & \textbf{3}                                     & \textbf{5}                           \\
\tt dolphins       & 62    & 159   & 28          &  & \textbf{2} & \textbf{7}        & \textbf{13}       & \textbf{19}       &  & \tt polbooks      & 105   & 441   & 43          &  & \textbf{8} & \textbf{15}        & \textbf{19}                                    & \textbf{25}                          \\
\tt football       & 115   & 613   & 21          &  & \textbf{21}\textcolor{red}{$^{\star\star}$}         & \textbf{43}\textcolor{red}{$^{\star\star}$}                & \textbf{60}\textcolor{red}{$^{\star\star}$}                & \textbf{71}       &  &                   &       &       &             &  &            &                    &                                                &                                      \\ \hline
\end{tabular}
\end{table}

\begin{table}[h!]
\centering
\small
\renewcommand\arraystretch{1.7}
\tabcolsep=1.75pt
\caption{Best known values for unweighted instances from the \texttt{Intersection} family. }
\label{tab:unweighted_best_known_intersection}
\begin{tabular}{lrrrlccccllrrrlcccc}
\hline
                  &       &       &             &  & \multicolumn{4}{c}{$k$}                                                                                                                                   &  &                     &       &       &             &  & \multicolumn{4}{c}{$k$}                                                                                                                                     \\ \cline{6-9} \cline{16-19} 
Instance          & $|\vertices|$ & $|\edges|$ & $\alpha(G)$ &  & 5                                    & 10                                   & 15                                   & 20                                   &  & Instance            & $|\vertices|$ & $|\edges|$ & $\alpha(G)$ &  & 5                                    & 10                                    & 15                                    & 20                                   \\ \cline{1-4} \cline{6-9} \cline{11-14} \cline{16-19} 
\tt arc130        & 130   & 7763  & 6           &  & \textbf{83}                          &                                      &                                      &                                      &  & \tt L120.fidap022   & 120   & 4307  & 5           &  & \textbf{87}                          &                                       &                                       &                                      \\
\tt ash219        & 85    & 219   & 29          &  & \textbf{7}                           & \textbf{16}                          & \textbf{26}                          & \textbf{34}                          &  & \tt L120.fidap025   & 120   & 2787  & 5           &  & .                                    &                                       &                                       &                                      \\
\tt ash331        & 104   & 331   & 30          &  & \textbf{8}                           & \textbf{21}                          & \textbf{33}\textcolor{red}{$^\star$} & \textbf{45}\textcolor{red}{$^\star$} &  & \tt L120.fidapm02   & 120   & 4626  & 5           &  & \textbf{91}                          &                                       &                                       &                                      \\
\tt ash85         & 85    & 616   & 14          &  & \textbf{22}                          & \textbf{46}\textcolor{red}{$^\star$} &                                      &                                      &  & \tt L120.rbs480a    & 120   & 3273  & 6           &  & \textbf{76}                          &                                       &                                       &                                      \\
\tt bcspwr01      & 39    & 118   & 13          &  & \textbf{7}                           & \textbf{16}                          &                                      &                                      &  & \tt L120.wm2        & 120   & 3387  & 23          &  & \textbf{3}$^\dag$                    & \textbf{8}                            & \textbf{13}                           & \textbf{41}                          \\
\tt bcspwr02      & 49    & 177   & 16          &  & \textbf{7}                           & \textbf{16}                          & \textbf{24}                          &                                      &  & \tt L125.ash608     & 125   & 390   & 37          &  & \textbf{8}                           & \textbf{19}\textcolor{red}{$^\star$}  & \textbf{30}\textcolor{red}{$^\star$}  & \textbf{41}\textcolor{red}{$^\star$} \\
\tt bcspwr03      & 118   & 576   & 32          &  & \textbf{10}                          & \textbf{23}                          & \textbf{35}                          & \textbf{46}                          &  & \tt L125.bcsstk05   & 125   & 2701  & 9           &  & \textbf{41}                          &                                       &                                       &                                      \\
\tt bfw62a        & 62    & 639   & 8           &  & \textbf{22}                          &                                      &                                      &                                      &  & \tt L125.can\_\_161 & 125   & 1257  & 15          &  & \textbf{36}\textcolor{red}{$^\star$} & \textbf{74}\textcolor{red}{$^\star$}  & \textbf{102}\textcolor{red}{$^\star$} &                                      \\
\tt can61         & 61    & 866   & 6           &  & \textbf{39}                          &                                      &                                      &                                      &  & \tt L125.can\_\_187 & 125   & 1022  & 20          &  & \textbf{26}\textcolor{red}{$^\star$} & \textbf{54}\textcolor{red}{$^\star$}  & \textbf{76}\textcolor{red}{$^\star$}  & \textbf{102}                         \\
\tt can62         & 62    & 210   & 18          &  & \textbf{7}                           & \textbf{17}                          & \textbf{27}                          &                                      &  & \tt L125.dwt\_\_162 & 125   & 943   & 16          &  & \textbf{23}\textcolor{red}{$^\star$} & \textbf{50}\textcolor{red}{$^\star$}  & \textbf{80}\textcolor{red}{$^\star$}  &                                      \\
\tt can73         & 73    & 652   & 13          &  & \textbf{28}                          & \textbf{52}\textcolor{red}{$^\star$} &                                      &                                      &  & \tt L125.dwt\_\_193 & 125   & 2982  & 8           &  & \textbf{56}                          &                                       &                                       &                                      \\
\tt can96         & 96    & 912   & 10          &  & \textbf{38}\textcolor{red}{$^\star$} & \textbf{72}\textcolor{red}{$^\star$} &                                      &                                      &  & \tt L125.fs\_183\_1 & 125   & 3392  & 9           &  & \textbf{16}                          &                                       &                                       &                                      \\
\tt can\_\_144    & 144   & 1656  & 12          &  & \textbf{45}\textcolor{red}{$^\star$} & \textbf{90}\textcolor{red}{$^\star$} &                                      &                                      &  & \tt L125.gre\_\_185 & 125   & 1177  & 19          &  & \textbf{27}                          & 60                                    & 89                                    &                                      \\
\tt curtis54      & 54    & 337   & 9           &  & \textbf{16}                          &                                      &                                      &                                      &  & \tt L125.lop163     & 125   & 1218  & 17          &  & \textbf{26}\textcolor{red}{$^{\star\star}$}                                   & \textbf{54}\textcolor{red}{$^\star$}  & 97                                    &                                      \\
\tt dwt66         & 66    & 255   & 13          &  & \textbf{15}                          & \textbf{35}\textcolor{red}{$^\star$} &                                      &                                      &  & \tt L125.west0167   & 125   & 444   & 39          &  & \textbf{5}                           & \textbf{11}                           & \textbf{17}                           & \textbf{24}                          \\
\tt dwt72         & 72    & 170   & 24          &  & \textbf{7}                           & \textbf{16}                          & \textbf{26}                          & \textbf{36}                          &  & \tt L125.will199    & 125   & 386   & 45          &  & \textbf{5}                           & \textbf{13}                           & \textbf{20}                           & \textbf{27}                          \\
\tt dwt87         & 87    & 726   & 16          &  & \textbf{11}                          & \textbf{29}                          & \textbf{54}                          &                                      &  & \tt L80.cavity01    & 80    & 1201  & 31          &  & \textbf{10}                          & \textbf{10}                           & \textbf{20}                           & \textbf{31}                          \\
\tt dwt\_\_\_59   & 59    & 256   & 15          &  & \textbf{10}                          & \textbf{25}                          & \textbf{41}                          &                                      &  & \tt L80.fidap025    & 80    & 1201  & 5           &  & .                                    &                                       &                                       &                                      \\
\tt gre\_\_115    & 115   & 576   & 33          &  & \textbf{12}                          & \textbf{24}                          & \textbf{38}\textcolor{red}{$^\star$} & \textbf{51}\textcolor{red}{$^\star$} &  & \tt L80.steam2      & 80    & 1272  & 6           &  & \textbf{48}                          &                                       &                                       &                                      \\
\tt ibm32         & 32    & 179   & 8           &  & \textbf{16}                          &                                      &                                      &                                      &  & \tt L80.wm1         & 80    & 1786  & 15          &  & \textbf{15}                          & \textbf{36}                           & \textbf{49}                           &                                      \\
\tt impcol\_b     & 59    & 329   & 20          &  & \textbf{5}                           & \textbf{13}                          & \textbf{23}                          & \textbf{38}                          &  & \tt L80.wm2         & 80    & 1848  & 11          &  & \textbf{4}                           & \textbf{48}                           &                                       &                                      \\
\tt L100.cavity01 & 100   & 1844  & 36          &  & \textbf{10}                          & \textbf{19}                          & \textbf{21}                          & \textbf{32}                          &  & \tt L80.wm3         & 80    & 1739  & 13          &  & \textbf{4}                           & \textbf{12}                           &                                       &                                      \\
\tt L100.fidap025 & 100   & 2031  & 5           &  & .                                    &                                      &                                      &                                      &  & \tt lund\_a         & 147   & 2837  & 10          &  & \textbf{55}\textcolor{red}{$^\star$} & \textbf{117}\textcolor{red}{$^\star$} &                                       &                                      \\
\tt L100.fidapm02 & 100   & 3090  & 5           &  & \textbf{80}                          &                                      &                                      &                                      &  & \tt pores\_1        & 30    & 179   & 6           &  & \textbf{20}                          &                                       &                                       &                                      \\
\tt L100.rbs480a  & 100   & 2550  & 5           &  & \textbf{66}                          &                                      &                                      &                                      &  & \tt rw136           & 136   & 641   & 39          &  & \textbf{7}                           & \textbf{20}\textcolor{red}{$^\star$}  & \textbf{34}\textcolor{red}{$^\star$}  & \textbf{48}\textcolor{red}{$^\star$} \\
\tt L100.steam2   & 100   & 1766  & 6           &  & \textbf{56}                          &                                      &                                      &                                      &  & \tt steam3          & 80    & 712   & 7           &  & \textbf{32}                          &                                       &                                       &                                      \\
\tt L100.wm1      & 100   & 2956  & 17          &  & \textbf{15}                          & \textbf{28}                          & \textbf{48}                          &                                      &  & \tt west0067        & 67    & 411   & 12          &  & \textbf{20}                          & \textbf{38}\textcolor{red}{$^\star$}  &                                       &                                      \\
\tt L100.wm2      & 100   & 3039  & 12          &  & \textbf{4}                           & \textbf{41}                          &                                      &                                      &  & \tt west0132        & 132   & 560   & 39          &  & \textbf{5}                           & \textbf{12}                           & \textbf{21}                           & \textbf{29}                          \\
\tt L100.wm3      & 100   & 2934  & 15          &  & \textbf{4}                           & \textbf{12}                          & \textbf{53}                          &                                      &  & \tt will57          & 57    & 304   & 10          &  & \textbf{7}                           & \textbf{22}                           &                                       &                                      \\
\tt L120.cavity01 & 120   & 2972  & 36          &  & \textbf{10}                          & \textbf{21}                          & \textbf{23}                          & \textbf{32}                          &  &                     &       &       &             &  &                                      &                                       &                                       &                                      \\ \hline
\end{tabular}
\end{table}


\begin{table}[h!]
\centering
\small
\renewcommand\arraystretch{1.7}
\tabcolsep=2.0pt
\caption{Best known values for weighted instances from the \texttt{Coloring} and \texttt{Partitioning} families. }
\label{tab:weighted_best_known_coloring_partitioning}
\begin{tabular}{lrrrlccccllrrrlcccc}
\hline
                    &       &       &             &  & \multicolumn{4}{c}{$k$}                                                                                                                                                          &  &                 &       &       &             &  & \multicolumn{4}{c}{$k$}                                                                                          \\ \cline{6-9} \cline{16-19} 
Instance            & $|\vertices|$ & $|\edges|$ & $\alpha(G)$ &  & 5                                               & 10                                             & 15                                    & 20                                    &  & Instance        & $|\vertices|$ & $|\edges|$ & $\alpha(G)$ &  & 5                                               & 10                 & 15                  & 20                  \\ \cline{1-4} \cline{6-9} \cline{11-14} \cline{16-19} 
\tt 1-FullIns\_3    & 30    & 100   & 14          &  & \textbf{35}                                     & \textbf{53}                                    &                                       &                                       &  & \tt mug88\_1    & 59    & 256   & 15          &  & \textbf{20}                                     & \textbf{43}        & \textbf{68}         & \textbf{99}         \\
\tt 1-FullIns\_4    & 93    & 593   & 45          &  & \textbf{35}$^\dag$                              & \textbf{66}$^\dag$                             & \textbf{90}$^\dag$                    & \textbf{122}$^\dag$                   &  & \tt mug88\_25   & 66    & 255   & 13          &  & \textbf{14}                                     & \textbf{38}        & \textbf{63}         & \textbf{93}         \\
\tt 1-Insertions\_4 & 67    & 232   & 32          &  & \textbf{40}                                     & \textbf{68}                                    & \textbf{100}                          & \textbf{125}                          &  & \tt mulsol.i.2  & 72    & 170   & 24          &  & .                                               & .                  & .                   & \textbf{96}         \\
\tt 2-FullIns\_3    & 52    & 201   & 25          &  & \textbf{42}                                     & \textbf{71}                                    & \textbf{92}                           & \textbf{125}                          &  & \tt mulsol.i.3  & 87    & 726   & 16          &  & .                                               & .                  & \textbf{96}         & \textbf{98}         \\
\tt 2-Insertions\_3 & 37    & 72    & 18          &  & \textbf{18}                                     & \textbf{50}                                    & \textbf{73}                           &                                       &  & \tt mulsol.i.4  & 115   & 613   & 21          &  & .                                               & .                  & \textbf{96}         & \textbf{98}         \\
\tt 2-Insertions\_4 & 149   & 541   & 74          &  & \textbf{42}$^\dag$                              & \textbf{69}$^\ddag$ & \textbf{99}$^\dag$                    & \textbf{124}$^\dag$                   &  & \tt mulsol.i.5  & 120   & 638   & 22          &  & .                                               & .                  & \textbf{96}         & \textbf{98}         \\
\tt 3-FullIns\_3    & 80    & 346   & 37          &  & \textbf{33}$^\dag$                              & \textbf{53}$^\dag$                             & \textbf{76}$^\dag$                    & \textbf{106}                          &  & \tt myciel3     & 115   & 576   & 33          &  & .                                               &                    &                     &                     \\
\tt 3-Insertions\_3 & 56    & 110   & 27          &  & \textbf{22}                                     & \textbf{47}                                    & \textbf{72}                           & \textbf{95}                           &  & \tt myciel4     & 74    & 301   & 27          &  & \textbf{38}                                     & \textbf{68}        &                     &                     \\
\tt 4-FullIns\_3    & 114   & 541   & 55          &  & \textbf{40}$^\dag$                              & \textbf{81}$^\dag$                             & \textbf{98}$^\dag$                    & \textbf{127}$^\dag$                   &  & \tt myciel5     & 32    & 179   & 8           &  & \textbf{47}                                     & \textbf{77}        & \textbf{105}        & \textbf{129}        \\
\tt 4-Insertions\_3 & 79    & 156   & 39          &  & \textbf{17}                                     & \textbf{43}                                    & \textbf{68}                           & \textbf{94}                           &  & \tt myciel6     & 59    & 329   & 20          &  & \textbf{57}$^\dag$                              & \textbf{87}$^\dag$ & \textbf{115}$^\dag$ & \textbf{138}$^\dag$ \\
\tt 5-FullIns\_3    & 154   & 792   & 72          &  & \textbf{35}$^\dag$                              & \textbf{72}$^\dag$                             & \textbf{95}$^\dag$                    & \textbf{113}$^\dag$                   &  & \tt myciel7     & 198   & 2742  & 40          &  & \textbf{67}$^\dag$                              & 97                 & 125                 & \textbf{148}$^\dag$ \\
\tt anna            & 138   & 493   & 80          &  & \textbf{7}                                      & \textbf{7}$^\dag$                              & \textbf{9}$^\dag$                     & \textbf{15}$^\dag$                    &  & \tt queen10\_10 & 34    & 78    & 20          &  & 360                                             & \textbf{486}       &                     &                     \\
\tt david           & 104   & 331   & 30          &  & .                                               & .                                              & \textbf{17}                           & \textbf{50}                           &  & \tt queen11\_11 & 100   & 1844  & 36          &  & 413                                             & 570                &                     &                     \\
\tt DSJC125.1       & 39    & 118   & 13          &  & \textbf{106}$^\ddag$ & 210                                            & 291                                   & 363                                   &  & \tt queen12\_12 & 100   & 2031  & 5           &  & 461                                             & 693                &                     &                     \\
\tt DSJC125.5       & 49    & 177   & 16          &  & 560                                             & \textbf{645}                                   &                                       &                                       &  & \tt queen13\_13 & 100   & 3090  & 5           &  & 547                                             & 833                &                     &                     \\
\tt games120        & 62    & 639   & 8           &  & \textbf{90}\textcolor{red}{$^\star$}            & 198                                            & \textbf{283}\textcolor{red}{$^\star$} & \textbf{353}\textcolor{red}{$^\star$} &  & \tt queen14\_14 & 100   & 2550  & 5           &  & 639                                             & 970                &                     &                     \\
\tt huck            & 144   & 1656  & 12          &  & \textbf{7}                                      & \textbf{17}                                    & \textbf{33}                           & \textbf{54}                           &  & \tt queen5\_5   & 100   & 1766  & 6           &  & \textbf{103}                                    &                    &                     &                     \\
\tt jean            & 62    & 210   & 18          &  & \textbf{2}                                      & \textbf{4}                                     & \textbf{14}                           & \textbf{19}                           &  & \tt queen6\_6   & 100   & 2956  & 17          &  & \textbf{149}                                    &                    &                     &                     \\
\tt miles1000       & 297   & 2148  & 110         &  & \textbf{297}                                    &                                                &                                       &                                       &  & \tt queen7\_7   & 100   & 3039  & 12          &  & \textbf{199}$^\dag$                             &                    &                     &                     \\
\tt miles1500       & 39    & 170   & 17          &  & \textbf{626}                                    &                                                &                                       &                                       &  & \tt queen8\_12  & 100   & 2934  & 15          &  & \textbf{339}$^\ddag$ &                    &                     &                     \\
\tt miles250        & 54    & 337   & 9           &  & .                                               & .                                              & \textbf{7}                            & \textbf{30}                           &  & \tt queen8\_8   & 120   & 2972  & 36          &  & \textbf{239}$^\dag$                             &                    &                     &                     \\
\tt miles500        & 87    & 406   & 36          &  & \textbf{42}                                     & \textbf{129}\textcolor{red}{$^\star$}          & \textbf{272}\textcolor{red}{$^\star$} &                                       &  & \tt queen9\_9   & 120   & 4307  & 5           &  & \textbf{296}$^\dag$                             &                    &                     &                     \\
\tt miles750        & 62    & 159   & 28          &  & \textbf{120}                                    & \textbf{400}\textcolor{red}{$^\star$}          &                                       &                                       &  & \tt r125.1      & 120   & 2787  & 5           &  & .                                               & .                  & \textbf{2}          & \textbf{9}          \\
\tt mug100\_1       & 125   & 736   & 34          &  & \textbf{10}                                     & \textbf{27}                                    & \textbf{46}                           & \textbf{69}                           &  & \tt r125.1c     & 120   & 4626  & 5           &  & \textbf{648}                                    &                    &                     &                     \\
\tt mug100\_25      & 125   & 3891  & 10          &  & \textbf{11}                                     & \textbf{30}                                    & \textbf{52}                           & \textbf{77}                           &  & \tt r125.5      & 120   & 3273  & 6           &  & \textbf{505}                                    &                    &                     &                     \\ 
\hline
\tt adjnoun        & 112   & 425   & 53          &  & \textbf{11}$^\dag$ & \textbf{29} & \textbf{51}        & \textbf{81}                           &  & \tt jazz          & 61    & 866   & 6           &  & \textbf{23}$^\dag$ & \textbf{70}$^\dag$ & \textbf{133}$^\ddag$ & \textbf{228}\textcolor{red}{$^\star$} \\
\tt celegansneural & 130   & 7763  & 6           &  & \textbf{5}         & \textbf{5}  & \textbf{15}$^\dag$ & \textbf{37}$^\dag$                    &  & \tt karate        & 73    & 652   & 13          &  & \textbf{11}        & \textbf{23}        & \textbf{34}                                     & \textbf{61}                           \\
\tt chesapeake     & 85    & 219   & 29          &  & \textbf{28}        & \textbf{60} & \textbf{92}        &                                       &  & \tt lesmis        & 96    & 912   & 10          &  & \textbf{4}         & \textbf{6}         & \textbf{13}                                     & \textbf{21}                           \\
\tt dolphins       & 85    & 616   & 14          &  & \textbf{10}        & \textbf{30} & \textbf{66}        & \textbf{89}                           &  & \tt polbooks      & 80    & 254   & 38          &  & \textbf{34}        & \textbf{79}        & \textbf{103}                                    & \textbf{136}                          \\
\tt football       & 118   & 576   & 32          &  & \textbf{101}       & 221         & 314                & \textbf{384}\textcolor{red}{$^\star$} &  &                   &       &       &             &  &                    &                    &                                                 &                                       \\ \hline
\end{tabular}
\end{table}

\begin{table}[h!]
\centering
\small
\renewcommand\arraystretch{1.7}
\tabcolsep=1.25pt
\caption{Best known values for weighted instances from the \texttt{Intersection} family.}
\label{tab:weighted_best_known_intersection}
\begin{tabular}{lrrrlccccllrrrlcccc}
\hline
                  &       &       &             &  & \multicolumn{4}{c}{$k$}                                                                                                                                       &  &                     &       &       &             &  & \multicolumn{4}{c}{$k$}                                                                                                                                                \\ \cline{6-9} \cline{16-19} 
Instance & $|\vertices|$ & $|\edges|$ & $\alpha(G)$ &  & 5                                     & 10                                    & 15                                    & 20                                    &  & Instance   & $|\vertices|$ & $|\edges|$ & $\alpha(G)$ &  & 5                                     & 10                                    & 15                                             & 20                                    \\ \cline{1-4} \cline{6-9} \cline{11-14} \cline{16-19} 
\tt arc130        & 120   & 3387  & 23          &  & \textbf{442}                          &                                       &                                       &                                       &  & \tt L120.fidap022   & 184   & 3916  & 86          &  & \textbf{486}                          &                                       &                                                &                                       \\
\tt ash219        & 125   & 390   & 37          &  & \textbf{36}                           & \textbf{78}                           & \textbf{120}                          & \textbf{164}                          &  & \tt L120.fidap025   & 185   & 3946  & 86          &  & .                                     &                                       &                                                &                                       \\
\tt ash331        & 125   & 2701  & 9           &  & \textbf{39}                           & \textbf{101}\textcolor{red}{$^\star$} & \textbf{161}\textcolor{red}{$^\star$} & \textbf{216}\textcolor{red}{$^\star$} &  & \tt L120.fidapm02   & 186   & 3973  & 88          &  & \textbf{509}                          &                                       &                                                &                                       \\
\tt ash85         & 125   & 1257  & 15          &  & \textbf{117}                          & \textbf{240}\textcolor{red}{$^\star$} &                                       &                                       &  & \tt L120.rbs480a    & 11    & 20    & 5           &  & \textbf{433}                          &                                       &                                                &                                       \\
\tt bcspwr01      & 125   & 1022  & 20          &  & \textbf{28}                           & \textbf{70}                           &                                       &                                       &  & \tt L120.wm2        & 23    & 71    & 11          &  & \textbf{7}                            & \textbf{28}                           & \textbf{67}                                    & \textbf{239}                          \\
\tt bcspwr02      & 125   & 943   & 16          &  & \textbf{38}                           & \textbf{87}                           & \textbf{133}                          &                                       &  & \tt L125.ash608     & 47    & 236   & 23          &  & \textbf{37}                           & \textbf{90}\textcolor{red}{$^\star$}  & \textbf{142}\textcolor{red}{$^\star$}          & \textbf{197}\textcolor{red}{$^\star$} \\
\tt bcspwr03      & 125   & 2982  & 8           &  & \textbf{53}                           & \textbf{113}                          & \textbf{168}                          & \textbf{235}                          &  & \tt L125.bcsstk05   & 95    & 755   & 47          &  & \textbf{218}                          &                                       &                                                &                                       \\
\tt bfw62a        & 125   & 3392  & 9           &  & \textbf{114}                          &                                       &                                       &                                       &  & \tt L125.can\_\_161 & 191   & 2360  & 95          &  & \textbf{193}\textcolor{red}{$^\star$} & \textbf{388}\textcolor{red}{$^\star$} & \textbf{543}\textcolor{red}{$^\star$}          &                                       \\
\tt can61         & 125   & 1218  & 17          &  & \textbf{207}                          &                                       &                                       &                                       &  & \tt L125.can\_\_187 & 105   & 441   & 43          &  & \textbf{134}\textcolor{red}{$^\star$} & \textbf{270}\textcolor{red}{$^\star$} & \textbf{402}\textcolor{red}{$^\star$}          & \textbf{541}                          \\
\tt can62         & 125   & 444   & 39          &  & \textbf{31}                           & \textbf{78}                           & \textbf{130}                          &                                       &  & \tt L125.dwt\_\_162 & 30    & 179   & 6           &  & \textbf{109}\textcolor{red}{$^\star$} & \textbf{245}\textcolor{red}{$^\star$} & \textbf{395}\textcolor{red}{$^\star$}          &                                       \\
\tt can73         & 125   & 386   & 45          &  & \textbf{144}                          & \textbf{271}\textcolor{red}{$^\star$} &                                       &                                       &  & \tt L125.dwt\_\_193 & 100   & 1470  & 10          &  & \textbf{291}                          &                                       &                                                &                                       \\
\tt can96         & 80    & 1201  & 31          &  & \textbf{190}\textcolor{red}{$^\star$} & \textbf{354}\textcolor{red}{$^\star$} &                                       &                                       &  & \tt L125.fs\_183\_1 & 121   & 1980  & 11          &  & \textbf{71}                           &                                       &                                                &                                       \\
\tt can\_\_144    & 125   & 1177  & 19          &  & \textbf{217}\textcolor{red}{$^\star$} & \textbf{472}\textcolor{red}{$^\star$} &                                       &                                       &  & \tt L125.gre\_\_185 & 144   & 2596  & 12          &  & \textbf{144}\textcolor{red}{$^\star$} & 348                                   & 502                                            &                                       \\
\tt curtis54      & 80    & 1201  & 5           &  & \textbf{74}                           &                                       &                                       &                                       &  & \tt L125.lop163     & 169   & 3328  & 13          &  & \textbf{142}\textcolor{red}{$^\star$} & 300                                   & 521                                            &                                       \\
\tt dwt66         & 80    & 1786  & 15          &  & \textbf{54}                           & \textbf{155}\textcolor{red}{$^\star$} &                                       &                                       &  & \tt L125.west0167   & 196   & 4186  & 14          &  & \textbf{19}                           & \textbf{46}                           & \textbf{73}                                    & \textbf{109}                          \\
\tt dwt72         & 80    & 1848  & 11          &  & \textbf{26}                           & \textbf{64}                           & \textbf{105}                          & \textbf{169}                          &  & \tt L125.will199    & 25    & 160   & 5           &  & \textbf{21}                           & \textbf{60}                           & \textbf{92}                                    & \textbf{127}                          \\
\tt dwt87         & 80    & 1739  & 13          &  & \textbf{66}                           & \textbf{164}\textcolor{red}{$^\star$} & \textbf{313}                          &                                       &  & \tt L80.cavity01    & 36    & 290   & 6           &  & \textbf{43}                           & \textbf{49}                           & \textbf{92}                                    & \textbf{154}                          \\
\tt dwt\_\_\_59   & 80    & 1272  & 6           &  & \textbf{59}                           & \textbf{141}                          & \textbf{226}                          &                                       &  & \tt L80.fidap025    & 49    & 476   & 7           &  & .                                     &                                       &                                                &                                       \\
\tt gre\_\_115    & 77    & 254   & 35          &  & \textbf{44}                           & \textbf{108}                          & \textbf{181}\textcolor{red}{$^\star$} & \textbf{255}\textcolor{red}{$^\star$} &  & \tt L80.steam2      & 96    & 1368  & 8           &  & \textbf{257}                          &                                       &                                                &                                       \\
\tt ibm32         & 147   & 2837  & 10          &  & \textbf{80}                           &                                       &                                       &                                       &  & \tt L80.wm1         & 64    & 728   & 8           &  & \textbf{88}                           & \textbf{218}                          & \textbf{281}                                   &                                       \\
\tt impcol\_b     & 128   & 3216  & 8           &  & \textbf{22}                           & \textbf{58}                           & \textbf{109}                          & \textbf{202}                          &  & \tt L80.wm2         & 81    & 1056  & 9           &  & \textbf{24}                           & \textbf{264}                          &                                                &                                       \\
\tt L100.cavity01 & 128   & 5198  & 5           &  & \textbf{49}                           & \textbf{91}                           & \textbf{100}                          & \textbf{162}                          &  & \tt L80.wm3         & 125   & 209   & 49          &  & \textbf{23}                           & \textbf{74}                           &                                                &                                       \\
\tt L100.fidap025 & 128   & 387   & 44          &  & .                                     &                                       &                                       &                                       &  & \tt lund\_a         & 125   & 7501  & 7           &  & \textbf{304}\textcolor{red}{$^\star$} & \textbf{661}\textcolor{red}{$^\star$} &                                                &                                       \\
\tt L100.fidapm02 & 128   & 1170  & 18          &  & \textbf{443}                          &                                       &                                       &                                       &  & \tt pores\_1        & 125   & 3838  & 5           &  & \textbf{99}                           &                                       &                                                &                                       \\
\tt L100.rbs480a  & 128   & 2113  & 12          &  & \textbf{370}                          &                                       &                                       &                                       &  & \tt rw136           & 136   & 641   & 39          &  & \textbf{40}                           & \textbf{107}\textcolor{red}{$^\star$} & \textbf{173}\textcolor{red}{$^\star$}          & \textbf{249}\textcolor{red}{$^\star$} \\
\tt L100.steam2   & 100   & 166   & 33          &  & \textbf{303}                          &                                       &                                       &                                       &  & \tt steam3          & 80    & 712   & 7           &  & \textbf{145}                          &                                       &                                                &                                       \\
\tt L100.wm1      & 100   & 166   & 33          &  & \textbf{78}                           & \textbf{169}                          & \textbf{274}                          &                                       &  & \tt west0067        & 67    & 411   & 12          &  & \textbf{97}                           & \textbf{188}                          &                                                &                                       \\
\tt L100.wm2      & 88    & 146   & 29          &  & \textbf{20}                           & \textbf{237}                          &                                       &                                       &  & \tt west0132        & 132   & 560   & 39          &  & \textbf{21}$^\dag$                    & \textbf{57}$^\dag$                    & \textbf{97}$^\ddag$ & \textbf{132}                          \\
\tt L100.wm3      & 88    & 146   & 29          &  & \textbf{20}                           & \textbf{76}                           & \textbf{303}                          &                                       &  & \tt will57          & 57    & 304   & 10          &  & \textbf{33}                           & \textbf{113}                          &                                                &                                       \\
\tt L120.cavity01 & 188   & 3885  & 90          &  & \textbf{49}$^\dag$                    & \textbf{100}                          & \textbf{115}                          & \textbf{168}                          &  &                     &       &       &             &  &                                       &                                       &                                                &                                       \\ \hline
\end{tabular}
\end{table}

\end{APPENDICES}
\end{document}